\documentclass[10pt]{amsart}
\usepackage{latexsym}
\usepackage{amsmath}
\usepackage{amssymb}
\usepackage{mathrsfs}
\usepackage{graphicx}
\usepackage{color}
\usepackage{pgfpages}
\usepackage{ifthen}
\usepackage{leftidx,tensor}
\usepackage[T1]{fontenc}
\usepackage[latin1]{inputenc}
\usepackage{mathtools}
\usepackage{comment}
\usepackage{dsfont}
\usepackage[nocompress]{cite}

\usepackage{tikz,pgfplots}
\usepackage{tikz-cd}
\pgfplotsset{compat=1.18}
\usetikzlibrary{arrows.meta}
\usetikzlibrary[math]

\usepackage[shortlabels]{enumitem}
\usepackage{aliascnt}
\usepackage[bookmarks=true,pdfstartview=FitH, pdfborder={0 0 0}, colorlinks=true,citecolor=red, linkcolor=blue]{hyperref}
\usepackage{bbm}
\usepackage{nicefrac}

\usepackage{graphicx}
\usepackage{subcaption}

\theoremstyle{plain}
\newtheorem{thm}{Theorem}[section]
\newaliascnt{cor}{thm}
\newaliascnt{prop}{thm}
\newaliascnt{lem}{thm}
\newtheorem{cor}[cor]{Corollary}
\newtheorem{prop}[prop]{Proposition}
\newtheorem{lem}[lem]{Lemma}
\aliascntresetthe{cor}
\aliascntresetthe{prop}
\aliascntresetthe{lem} 
\theoremstyle{definition}
\newaliascnt{defn}{thm}
\newaliascnt{asu}{thm}
\newaliascnt{con}{thm}
\newtheorem{defn}[defn]{Definition}

\aliascntresetthe{defn}
\aliascntresetthe{asu}
\aliascntresetthe{con}
\newcounter{stp}
\newcounter{stpi}
\newcounter{stpci}
\newcounter{stpiii}

\theoremstyle{remark}
\newaliascnt{rem}{thm}
\newaliascnt{exa}{thm}
\newaliascnt{masu}{thm}
\newaliascnt{nota}{thm}
\newaliascnt{sett}{thm}
\newtheorem{rem}[rem]{Remark}

\aliascntresetthe{rem}
\aliascntresetthe{exa}
\aliascntresetthe{masu}
\aliascntresetthe{nota}
\aliascntresetthe{sett}
\numberwithin{equation}{section}

\setlist[enumerate]{font = \normalfont}

\newcommand{\F}{\mathbb{F}}
\newcommand{\E}{\mathbb{E}}
\newcommand{\R}{\mathbb{R}}
\newcommand{\C}{\mathbb{C}}
\newcommand{\B}{\mathbb{B}}
\newcommand{\N}{\mathbb{N}}
\newcommand{\bP}{\mathbb{P}}

\newcommand{\G}{\mathbb{G}}

\newcommand{\bA}{\mathbb{A}}

\newcommand{\rW}{\mathrm{W}} 
\newcommand{\rL}{\mathrm{L}}

\newcommand{\rH}{\mathrm{H}}

\newcommand{\rC}{\mathrm{C}}
\newcommand{\rB}{\mathrm{B}}

\newcommand{\rD}{\mathrm{D}}

\newcommand{\rX}{\mathrm{X}}

\newcommand{\rI}{\mathrm{I}}
\newcommand{\rII}{\mathrm{II}}

\newcommand{\rd}{\mathrm{d}}
\newcommand{\srd}{\,\mathrm{d}}

\newcommand{\mre}{\mathrm{e}}

\newcommand{\mrm}{\mathrm{m}}

\newcommand{\rp}{\mathrm{p}}
\newcommand{\per}{\mathrm{per}}

\newcommand{\cA}{\mathcal{A}}
\newcommand{\cS}{\mathcal{S}}
\newcommand{\cF}{\mathcal{F}}

\newcommand{\cO}{\mathcal{O}}
\newcommand{\cB}{\mathcal{B}}
\newcommand{\cL}{\mathcal{L}}

\newcommand{\cR}{\mathcal{R}}

\newcommand{\cV}{\mathcal{V}}
\newcommand{\cI}{\mathcal{I}}

\newcommand{\eps}{\varepsilon}
\newcommand{\del}{\partial}

\DeclareMathOperator{\tr}{tr}
\DeclareMathOperator{\diag}{diag}
\DeclareMathOperator{\Id}{Id}
\DeclareMathOperator{\mdiv}{div}
\DeclareMathOperator{\Rep}{Re}

\newcommand{\tin}{\enspace \text{in} \enspace}
\newcommand{\ton}{\enspace \text{on} \enspace}
\newcommand{\tfor}{\enspace \text{for} \enspace}
\newcommand{\tforall}{\enspace \text{for all} \enspace}
\newcommand{\tand}{\enspace \text{and} \enspace}

\newcommand{\twhere}{\enspace \text{where} \enspace}
\newcommand{\twith}{\enspace \text{with} \enspace}

\newcommand{\qbar}{\overline{q}}
\newcommand{\qbr}{[q]}
\newcommand{\sbar}{\overline{s}}
\newcommand{\sbr}{[s]}
\newcommand{\btau}{\boldsymbol{\tau}}
\newcommand{\bn}{\boldsymbol{n}}
\newcommand{\bu}{\boldsymbol{u}}
\newcommand{\bv}{\boldsymbol{v}}
\newcommand{\bx}{\boldsymbol{x}}
\newcommand{\by}{\boldsymbol{y}}
\newcommand{\bff}{\boldsymbol{f}}
\newcommand{\bg}{\boldsymbol{g}}
\newcommand{\ba}{\boldsymbol{a}}
\newcommand{\bsigma}{\boldsymbol{\sigma}}
\newcommand{\bnu}{\boldsymbol{\nu}}
\newcommand{\bI}{\boldsymbol{I}}
\newcommand{\bD}{\boldsymbol{D}}
\newcommand{\bF}{\boldsymbol{F}}

\newcommand{\bfe}{\textbf{e}}
\newcommand{\oOmega}{\overline{\Omega}}
\newcommand{\oGamma}{\overline{\Gamma}}
\newcommand{\sol}{\mathrm{sol}}
\newcommand{\test}{\mathrm{test}}
\newcommand{\sd}{\mathrm{sd}}

\begin{document}

\title[Fluid-poroelastic structure interaction with a fully averaged plate]
{A Fully Averaged Poroelastic Kirchhoff Plate Interacting with an Incompressible, Viscous Fluid: Analysis and Numerical Simulation}

\author{Felix Brandt}
\address{Department of Mathematics, University of California at Berkeley, Berkeley, 94720, CA, USA.}
\email{fbrandt@berkeley.edu}
\author{Sun\v{c}ica \v{C}ani\'c}
\address{Department of Mathematics, University of California at Berkeley, Berkeley, 94720, CA, USA.}
\email{canics@berkeley.edu}
\author{Andrew Scharf}
\address{Department of Mathematics, University of California at Berkeley, Berkeley, 94720, CA, USA.}
\email{andrew\_scharf@berkeley.edu}
\author{Josip Tamba\v{c}a}
\address{Department of Mathematics, Faculty of Science, University of Zagreb, Zagreb, Croatia.}
\email{tambaca@math.hr}

\begin{abstract}
We study a new fully averaged poroelastic Kirchhoff plate model coupled with the flow of an incompressible, viscous fluid governed by the time-dependent Stokes equations. 
The fully averaged formulation offers several advantages over the classical Biot poroelastic plate model: both the elastodynamic and pressure equations are posed on a codimension-one interface, the resulting numerical schemes are simpler to implement and computationally more efficient, and the coupling with an incompressible viscous fluid is more natural than in the full Biot plate setting.
In this work we analyze a linearly coupled fluid--structure interaction problem with kinematic and dynamic interface conditions enforcing continuity of normal velocities, the Beavers--Joseph--Saffman slip in the tangential velocities, and balance of forces between the fluid and the poroelastic structure. 
We first establish the existence of weak solutions to the coupled FSI system using energy methods, and then prove global-in-time existence of a unique strong solution to a regularized version of the problem using sectoriality of the associated spatial operator and maximal $\rL^p$-regularity of the resulting Cauchy problem. 
For data with exponential decay, we prove exponential decay of solutions. 
Finally, we develop a finite element method for the numerical approximation of the coupled system and show that the fully averaged Kirchhoff poroelastic plate model coupled with the time-dependent Stokes equations provides an excellent approximation of the full Biot-Stokes system in the thin-structure regime. The main advantage of this model lies in the remarkably simple implementation of the numerical scheme, stemming from the fact that the poroelastic plate equations constitute a surface model bounding a bulk fluid domain.

These results provide a rigorous analytical and computational framework for the study of coupled fluid--poroelastic structure interactions involving thin poroelastic interfaces modeled by the fully averaged Kirchhoff poroelastic plate equations.
\end{abstract}

\maketitle

\section{Introduction}

In this work, we introduce and analyze a new model for fluid--poroelastic structure interaction (FPSI) describing the coupling between an incompressible, viscous fluid and a thin Kirchhoff poroelastic plate, recently derived by the authors in \cite{SCTB:26}. 
Such configurations arise naturally in applications involving permeable membranes and biological tissues, where fluid flow interacts with deformable porous structures.
In the {design of bioartificial organs}, these models are used to characterize the behavior of thin, semi-permeable membranes that must withstand physiological fluid pressures while allowing for the transport of solutes \cite{Bociu2016, Bukac2016, Miotto2021}. 
Within the broader scope of {biomedical engineering}, such models are critical for simulating the mechanics of vascular grafts, the dynamics of the blood-brain barrier, and the filtration efficiency of the kidneys \cite{Bociu2021, Cou:04, Sho:00}. 
Furthermore, these interactions are prevalent in {other branches of science} and technology, ranging from the engineering of soft actuators in robotics to the analysis of energy-absorbing poroelastic cladding in naval architecture and the study of sub-surface fluid flow through fractured geological thin-beds \cite{Biot1941, Terzaghi1943}.

The fully averaged Kirchhoff poroelastic plate model considered in this work was recently derived in \cite{SCTB:26} from the classical Biot poroelastic plate model, originally introduced by Biot \cite{Biot:55} and later rigorously justified in \cite{MCM:15}.
A notable drawback of the Biot poroelastic plate model in \cite{Biot:55} is that it is posed on domains of mixed dimensions. 
The elastodynamics of the plate is defined on a codimension-one manifold (the middle surface of the plate), whereas the pressure equation describing fluid permeation through the plate is defined on a codimension-zero three-dimensional domain. 
This dimensional mismatch introduces significant challenges, both in the numerical simulation of FSI problems involving Biot poroelastic plates~\cite{SBC:25} and in their mathematical analysis \cite{BCMW:21}.

The fully averaged poroelastic Kirchhoff plate model introduced in \cite{SCTB:26} overcomes this difficulty by averaging the pressure equation across the thickness of the plate in a manner that preserves the underlying energy structure of the original Biot plate model. 
The pressure equation becomes an equation for the jump in the pressure across the middle surface of the plate, coupled with a ``closure'' equation for the transverse averaged pressure. 
The resulting reduced model admits an energy functional consistent with that of the three-dimensional Biot system and is naturally suited for coupling with an external fluid flow. 

In this fully averaged formulation derived in \cite{SCTB:26}, both the elastodynamics of the plate and the jump in the pressure equation are defined on the same codimension-one domain, which corresponds to the middle surface of the plate. 
This dimensional consistency eliminates the mixed-dimensional coupling present in the classical Biot plate model and leads to a system that is significantly more tractable for both mathematical analysis and numerical approximation.
In particular, the model derived in \cite{SCTB:26} provides a framework that can be coupled to incompressible, viscous fluid flow in a straightforward and energetically consistent manner. 

An important feature of the averaged poroelastic plate model is that filtration occurs predominantly in the normal direction, as proved in \cite{MCM:15} for the Biot poroelastic plate proposed in \cite{Biot:55}, which makes it particularly suitable for FSI problems.
This contrasts a poroelastic plate model studied in \cite{Ruiz2026, Iliev2016} where filtration in the poroelastic plate (through Darcy's law) occurs in the tangent plane to the middle surface of the plate, and the apparent fluid pressure in that system has the physical meaning of the first moment of the pressure across the thickness of the plate. 

In this manuscript, we couple the fully averaged poroelastic plate model derived in \cite{SCTB:26} with the flow of an incompressible, viscous fluid governed by the time-dependent Stokes equations. 
The two systems are linearly coupled across a fixed interface via kinematic and dynamic conditions. 
They enforce normal velocity continuity and tangential slip described by the Beavers--Joseph--Saffman condition (kinematic condition), and the balance of forces (dynamic condition). 
This model is formally presented in \autoref{sec:FPSI problem}.

In \autoref{sec:weak sols}, we prove the existence of a weak solution to the proposed coupled FPSI problem, where by weak solution, we mean a solution belonging to a finite--energy space. 
To prove the existence of weak solutions, we follow an approach similar to that in \cite{BCMW:21}. The proof is constructive. 
We semidiscretize the problem in time and show that for each $\Delta t > 0$, the corresponding approximate weak formulation admits a unique solution after an appropriate rescaling of the test functions.
We then establish uniform in $\Delta t$ \emph{a priori} energy estimates for the approximate solutions, which allow us to extract subsequences converging in suitable weak and weak* topologies to a weak solution of the coupled problem. 

In \autoref{sec:strong sols}, we establish one of the key contributions of this manuscript, namely, the existence of a unique global-in-time strong solution to a regularized version of the FPSI problem. 
This regularization incorporates viscoelastic terms within the Kirchhoff poroelastic plate's elastodynamics and an elliptic regularization term in the equation for the jump in the fluid pressure across the poroelastic plate middle surface. 

We reformulate the regularized coupled FPSI problem as an abstract Cauchy problem and establish global well-posedness by proving that the associated spatial operator is sectorial and generates an analytic semigroup. 
This approach allows us to exploit maximal $\rL^p$-regularity and semigroup theory in a systematic way.

A central difficulty in this analysis stems from the structure of the coupled system. 
The problem is not parabolic in a straightforward sense: 
the elastodynamics of the plate, the pressure variables, and the fluid velocity are coupled through interface conditions, and several of these coupling terms appear as lower-order boundary contributions -- particularly challenging is the handling of the Beavers--Joseph--Saffman coupling condition. 
As a result, the associated operator matrix is not diagonally dominant in an obvious way, and classical perturbation results cannot be applied directly.

To overcome these difficulties, we treat the lower-order boundary terms by working in suitable extrapolation spaces and by employing a perturbation framework based on Neumann operators associated with the corresponding stationary problems with inhomogeneous boundary data. 
This approach allows us to isolate and control the most delicate coupling mechanisms, in particular those involving the fluid velocity and the pressure jump. 

Building on this reduction, we then apply the theory of diagonally dominant operator matrices (see, e.g., \cite{Tre:08, AH:23}) to progressively incorporate the remaining components of the system, including the averaged pressure and the plate elastodynamics.
As a consequence, we establish the sectoriality of the full operator matrix and deduce the existence and uniqueness of a global-in-time strong solution to the regularized FPSI problem in an appropriate Hilbert space.

Finally, in \autoref{sec:numerics}, we develop a finite element method (FEM) for the proposed (nonregularized) system introduced in \autoref{sec:FPSI problem}, and demonstrate that the resulting FPSI solutions provide an accurate approximation of the full  Stokes--Biot coupled problem in the small structural thickness regime. 
Additionally, we show that as the thickness of the poroelastic medium decreases, the differences between the two models vanish and their solutions become indistinguishable.
Moreover, the propagation speed of the traveling wave generated by the inlet/outlet pressure data agrees remarkably well with the classical Moens--Korteweg pulse wave velocity.
Therefore, we show that the reduced model retains all the essential features of the full system while significantly lowering the computational cost.

\section{Literature review}

The classical theory of poroelasticity providing a continuum description of fluid flow through deformable porous media was first introduced by Biot~\cite{Biot:55,Biot1941}. 
Alternative approaches derive poroelastic models from pore-scale fluid--structure interaction via homogenization techniques; see, for example, \cite{Aur:97, FM:03, MW:12} and the monograph \cite{Cou:04}. 
Their importance in many applications, including biological applications such as tissue perfusion, filtration, and transport has been addressed in many works, see, for example, \cite{BBMBNG:19, BGSW:16, BW:21, BMW:22, BMW:23, CTGMHR:06, CGHPST:14, HAvCR:92, VGBS:18}.
Mathematically, the well-posedness of such models has been studied, e.g., in \cite{Sho:00, BMTSM:05, Biotwell6, Sho:05,Biotwell3, BMTSM:05, Biotwell2} and the references therein. 

The focus of the current paper is on an FPSI problem involving thin poroelastic structures. 
A major step toward the derivation and analysis of thin, reduced lower-dimensional poroelastic models was obtained by Marciniak-Czochra and Mikeli\'c in \cite{MCM:15} and by Mikeli\'c and Tamba\v{c}a in \cite{MT:16}, where poroelastic plate and shell models were rigorously derived from three-dimensional Biot systems. 
The work in \cite{MCM:15,MT:16} provides a key foundation for the derivation of the reduced fully averaged Kirchhoff plate model considered in the present paper.

The fully averaged Kirchhoff plate model in this work is coupled to the flow of a viscous, incompressible fluid modeled by the time-dependent Stokes equations. 
The interaction between flows of viscous, incompressible fluids modeled by the Stokes and Navier--Stokes equations, and poroelastic media has been studied in a number of works; 
see \cite{AENY:19, BCMW:21, AGW:25, AW:25, Ces:17, RN:20, Sho:05, BGM23, Biotwell1, Biotwell2, BMTSM:05, Sho:00, Biotwell5, Biotwell6, BGSW:16, BW:21, BMW:22, BMW:23, NonlinearFPSI_CRM23, KCM:24, SingularLimit,BCM:25}. 
Most of these results are very recent. 
In particular, a subset of those results addresses FPSI problems with multilayered poroelastic structures \cite{NonlinearFPSI_CRM23, KCM:24, SingularLimit,BCMW:21,BCM:25} consisting of a thick, bulk Biot model and a Stokes or Navier--Stokes flow, coupled over an interface which has inertia and poroelastic properties, modeled by the Biot poroelastic plate model, discussed in \cite{MCM:15,MT:16}.
The presence of the poroelastic plate interface is motivated by various biological applications, including the recent advances in the design of bioartificial organs and vascular devices, see \cite{FluidsCanic, YifanDES}.
Mathematically, the thin interface with mass and poroelastic properties, regularizes the coupled problem in the sense that it provides sufficient regularity of the weak solution to make sense of the traces of the unknown functions along the interface.
None of these works consider an FPSI problem involving the newly derived fully averaged Kirchhoff poroelastic plate model. 

One of the goals of the present work is to prove the existence of weak solutions to the FPSI problem with the fully averaged poroelastic plate model. 
Among the available results on weak solution well-posedness, the work in \cite{BCMW:21} is most closely related to our constructive proof of the existence of weak solutions.
While the underlying ideas share similarities, a key distinction is that \cite{BCMW:21} addresses Stokes flow interacting with multilayered poroelastic structures, whereas in the present work we consider a single thin poroelastic structure modeled by a fully averaged Kirchhoff plate, which introduces both conceptual and analytical differences.

Another goal of the present work is to address the well-posedness of strong solutions to the proposed model. 
This is one of the major result of the current manuscript.
In contrast to weak solution theories, results addressing strong-solution well-posedness for FPSI systems are comparatively scarce. 
To the best of our knowledge, only two works treat this problem in a rigorous manner, namely \cite{AGW:25} and \cite{BHR:25}.
While none of these contributions considers a fully averaged Kirchhoff plate coupled with the time-dependent Stokes system, they provide important benchmarks for comparison.

In \cite{AGW:25}, weak and strong solutions are established within a semigroup framework based on a mixed variational formulation and the Lumer--Phillips theorem. 
In contrast, our analysis provides an explicit characterization of the operator domain, which enables us to establish sectoriality and, consequently, maximal $\mathrm{L}^p$-regularity.
Moreover, while \cite{AGW:25} focuses on finite time intervals, our framework allows for global-in-time analysis and the study of asymptotic behavior. 
Another important distinction lies in the regularity setting: 
our solutions belong to the maximal $\mathrm{L}^p$-regularity class and admit lower-regularity initial data via interpolation spaces, whereas \cite{AGW:25} requires initial data in the operator domain. 
Finally, our approach extends naturally to $\mathrm{L}^q$ spaces, which is advantageous for nonlinear extensions.

Finally, in \cite{BHR:25}, strong solutions are obtained for a Navier--Stokes fluid coupled with a three-dimensional porous medium via boundary perturbation arguments and an operator matrix formulation.
A key difference is that \cite{BHR:25} models only the evolution of pore pressure, whereas our formulation includes the full elastodynamics of a Kirchhoff plate. 
In addition, the pressure variables in \cite{BHR:25} are defined in a three-dimensional domain, while in our fully averaged model the relevant quantities, such as the pressure jump and averaged pressure, are defined directly on the two-dimensional interface. 
This dimensional reduction leads to additional analytical challenges, including a more delicate spectral analysis due to the potential lack of compactness of the resolvent and the need to handle inhomogeneous interface terms.

We conclude this literature review with a brief discussion of numerical methods for FPSI systems and their relation to the problem studied in this work. 
Numerical simulation of FPSI systems has seen substantial progress over the past fifteen years; see, for example, \cite{AKYZ:18, BQQ:09, BYZ:15, BYZ:17, KBM:25, OB:20, PB:25, SBC:25, Biotwell3, BYZZ:15}. However, with the exception of \cite{SBC:25}, none of these works focuses on FPSI involving thin poroelastic structures. In \cite{SBC:25}, a partitioned solver was developed for an FPSI problem coupling Stokes flow with a multilayered poroelastic structure consisting of a bulk Biot poroelastic medium coupled to a thin Biot poroelastic plate at the interface.

In the present manuscript, we develop a monolithic finite element method (FEM) solver for the Stokes--Kirchhoff poroelastic plate system (without the regularizations introduced for the analysis of strong solutions), with the primary goal of validating the proposed model and assessing its accuracy.
To this end, we compare our model with the full Stokes--Biot coupled problem in the thin-structure regime. The full Stokes--Biot problem is simulated using the numerical methods developed in \cite{BYZZ:15}. Our numerical simulations show excellent comparison between the two models.

In the context of numerical simulations, we also mention the recent works \cite{Iliev2016, KRB:25}, which study quasistatic problems for a different fully averaged poroelastic plate model. 
The approaches and results presented in \cite{Iliev2016, KRB:25} differ substantially from the Kirchhoff-type model considered here. In addition to the differences in the reduced equations themselves, the models studied in \cite{Iliev2016, KRB:25} are quasistatic in nature and are formulated in terms of the pressure moment. 
To the best of our knowledge, their coupling with Stokes flow has not been investigated, possibly due to difficulties associated with formulating the coupled problem in an energetically stable manner.

\medskip

Before we proceed, we summarize the main contributions of this work, 
which introduces and analyzes a new FPSI model based on a fully averaged poroelastic plate formulation:
\begin{itemize}
    \item \textbf{Modeling:} We derive a new Stokes-Kirchhoff poroelastic plate coupled problem based on 
    a recently derived fully averaged Kirchhoff poroelastic plate model \cite{SCTB:26}, preserving the essential coupling mechanisms and energy structure of the full Stokes--Biot problem.
    \item \textbf{Coupling structure:} The model incorporates physically relevant Beavers--Joseph--Saffman interface conditions and accounts for both the averaged pore pressure and the pressure jump across the structure, leading to a novel coupling between fluid and plate dynamics.
    \item \textbf{Analytical framework:} We prove the existence of both weak and strong solutions. 
    In the context of weak solutions we use ideas similar to those developed in \cite{BCMW:21}. 
    In the context of strong solutions, we formulate the coupled FPSI system in operator form with a carefully characterized domain, allowing us to establish sectoriality of the associated operator and, consequently, maximal $\mathrm{L}^p$-regularity.
    \item \textbf{Well-posedness and regularity:} 
    As a result, we obtain strong solutions in a maximal $\mathrm{L}^p$-regularity setting, which provides a robust framework for studying nonlinear extensions, which is left to future study, and long-time behavior.
    \item \textbf{Numerical comparison:} We develop a numerical solver to compare the full Biot model coupled with Stokes flow in the thin-structure regime, and demonstrate that the proposed FPSI model based on a fully averaged poroelastic Kirchhoff plate provides an excellent approximation of the full system.
    \item \textbf{Perspective:} The reduced nature of the model makes it particularly suitable for both analytical investigations and numerical simulation, offering a promising approach for studying thin poroelastic structures in complex fluid environments.
\end{itemize}

\section{The model problem}\label{sec:FPSI problem}

In this section we state the coupled FSI problem describing the interaction between the fully averaged poroelastic plate, introduced in \cite{SCTB:26}, and the flow of an incompressible, viscous fluid modeled by the time-dependent Stokes equations. 

\medskip

\noindent{{\bfseries The fully averaged poroelastic plate model.}}
We begin by describing the original poroelastic plate model derived by Biot \cite{Biot1941,Biot:55} and rigorously justified by Marciniak-Czochra and Mikeli\'{c} \cite{MCM:15}. 
The model captures the flow of a fluid through a thin poroelastic plate of thickness $H$ in terms of the transverse (vertical)  displacement $w$ of the plate's middle surface, the pore pressure $q$, and the transverse filtration velocity $u_p$.

Biot's formulation describes the elastodynamics of the thin plate coupled with fluid flow through the porous matrix via the following system of three equations:
\begin{itemize}
    \item The elastodynamics of thin plate, see \eqref{eq:thin1} below, defined on the reference middle surface of the plate $\Gamma$ (a codimension-one surface):
    \begin{equation*}
        \Gamma = \left\{ (x,y,z) : x\in (0,L), \ y \in (0,L), \  z = 0 \right\}; 
    \end{equation*}
    \item Conservation of mass of the seeping fluid, see \eqref{eq:thin2} below,  written in terms of fluid pressure $q$, defined on the full codimension-zero domain $\Omega_p$:
    \begin{equation}\label{Omegap}
        \Omega_p = \left\{ (x,y,z) : x\in (0,L), \ y \in (0,L), \ z \in \Bigl(-\frac{H}{2},\frac{H}{2}\Bigr) \right\} = \Gamma \times \Bigl(-\frac{H}{2},\frac{H}{2}\Bigr);
    \end{equation}
    \item Darcy's law, see \eqref{eq:thin3} below, which associates the fluid velocity to the pressure gradient, also defined on $\Omega_p$.
\end{itemize}
For $T > 0$, the Biot poroelastic plate system is:
\begin{subequations}\label{eq:thin}
    \begin{align} 
        H\rho_{p} \partial_{t}^2  w+H^3D \Delta_{\boldsymbol{\tau}}^2  w +H\gamma^p w+H \alpha^{p} \Delta_{\boldsymbol{\tau}} \overline{z q}
        &=F_{p}   &&\tin (0,T) \times \Gamma, \label{eq:thin1}\\
        \partial_{t}(c_{0}^{p}  q-\alpha^{p} z \Delta_{\boldsymbol{\tau}}\tilde w)+\nabla_{\boldsymbol{n}}u_p 
        &= G_p &&\tin (0,T) \times \Omega_p, \label{eq:thin2}\\
        (\kappa^p)^{-1}u_p 
        &= -\nabla_{\boldsymbol{n}} q &&\tin (0,T) \times \Omega_p. \label{eq:thin3}
    \end{align}
\end{subequations}
Here $\kappa^p > 0$ is a permeability coefficient, $F_p$ and $G_p$ are forcing terms, $\rho_p$ is the plate area density, $c_0^p$ is the storage coefficient, $\alpha^p$ is the Biot-Willis parameter representing the strength of fluid-solid coupling,~$D$ is the bending stiffness of the plate, and $\gamma^p$ is a 2D spring coefficient. 
Moreover, $\bn$ represents the unit vector normal to the plate  perpendicular to $\Gamma$, aligned with the $z$ axis so that $\bn = \bfe_3$, and $\btau_i$, $i = 1,2$, denote the tangential unit vectors, i.e., $\btau_i = \bfe_i$.
In this context, we also remark that $\Delta_{\boldsymbol{\tau}}$ denotes the horizontal Laplacian.
Moreover, the bar notation in \eqref{eq:thin1} represents the vertical average, see also \eqref{eq:qbr & qbar} for a precise definition in the context of the pressure variables.
Since the displacement and pressure are defined on domains of different dimensions, in \eqref{eq:thin2} we use $\tilde w$ to denote the extension of $w$ from $\Gamma$ to~$\Omega_p$, taken to be constant in the normal direction.

The fact that the displacement and fluid pressure are defined on domains of different dimensions is a significant drawback of the Biot poroelastic plate model, both from analytical and numerical perspectives. 
This motivates the development of a fully averaged poroelastic plate model, as proposed in \cite{SCTB:26}.

While several averaging procedures can be applied to the Biot plate model, not all lead to formulations that are energy-consistent with the full Biot model, or readily coupled with the Stokes equations governing free fluid flow. 
The averaged model introduced in \cite{SCTB:26} preserves the energy structure of the full Biot system, satisfying the same type of energy estimate, and allows for a natural coupling with the Stokes equations, as we demonstrate below.
The resulting model is formulated in terms of the displacement of the plate's middle surface $w$, the jump in the fluid pressure $[q]$ across the middle surface $\Gamma$, and the averaged fluid pressure $\bar q$, defined by:
\begin{equation}\label{eq:qbr & qbar}
    [q] = q^+ - q^- \tand \qbar = \frac{1}{H}\int_{-H/2}^{H/2} q \srd z.
\end{equation}
Here $q^+$ and $q^-$ denote the right and left limits of $q$, representing its values on the top and bottom faces of the surface $\Gamma$, respectively. 
For a sketch of the geometric setup, we also refer to \autoref{fig:domain FPSI problem}.
Before we state the fully averaged poroelastic plate model, we introduce the notation for the velocity of the plate's middle surface displacement:
\[
    v = \partial_t w.
\]

The fully averaged plate model from \cite{SCTB:26} reads as follows: find $w, v, [q]$, and $\qbar$ such that
\begin{equation}\label{eq:fully averaged plate model}
    \left\{
    \begin{aligned}
        H \rho_p \del_t v + H^3 D \Delta_{\btau}^2 w + H \gamma^p w + H^2 \frac{\alpha^p}{12} \Delta_{\btau} [q]
        &= F_p, &&\tin (0,T) \times \Gamma,\\
        H \frac{c_0^p}{12} \del_t [q] - H^2 \frac{\alpha^p}{12} \Delta_{\btau} v + \frac{4 \kappa^p}{H} [q] + \frac{6 \kappa^p}{H} \qbar
        &= -u_p^- + \frac{6 \kappa^p}{H} q^+, &&\tin (0,T) \times \Gamma,\\
        H c_0^p \del_t \qbar + \frac{12 \kappa^p}{H} \qbar + \frac{6 \kappa^p}{H} [q]
        &= \frac{12 \kappa^p}{H} q^+, &&\tin (0,T) \times \Gamma,
    \end{aligned}
    \right.
\end{equation}
where $F_p$ is the total load onto the structure, given by the jump in the surface force density acting on the top and bottom faces of the plate:
\[
    F_p = F^+ - F^-,
\]
and $u_p^-$ and $q^+$ are the boundary data prescribed at the bottom and top faces of the plate surface:
\begin{equation*}
    \begin{aligned}
        u_p^- &\text{ is the fluid velocity prescribed at the bottom face;}\\
        q^+ &\text{ is the fluid pressure prescribed at the top face.}
    \end{aligned}
\end{equation*}
Notice that all three equations in \eqref{eq:fully averaged plate model} are defined on the middle surface $\Gamma$ of the poroelastic plate. 
The parameters in the problem and their typical values are presented in \autoref{tab:param} in \autoref{sec:numerics}.

\medskip

\noindent{{\bfseries The time-dependent Stokes equations.}}
The averaged plate model is coupled to the flow of an incompressible, viscous fluid described by the time-dependent Stokes equations:
\begin{equation}\label{Stokes}
    \left\{
    \begin{aligned}
        \rho_f \del_t \bu - \nabla \cdot \bsigma_f(\bu,\pi)
        &= \bF_f, &&\tin (0,T) \times \Omega_f ,\\
        \nabla \cdot \bu
        &= 0, &&\tin (0,T) \times \Omega_f ,
    \end{aligned}
    \right.
\end{equation}
where $\bu$ is fluid velocity, $\pi$ the pressure, $\bF_f$ denotes external force, and $\bsigma_f(\bu,\pi)$ the Cauchy stress tensor. 
We will be assuming that the fluid is Newtonian so that $\bsigma_f(\bu,\pi) = -\pi \bI + 2 \mu_f \bD(\bu)$, where $\mu_f > 0$ denotes the fluid viscosity.
The equations are posed on a cylindrical domain in $\mathbb{R}^3$ given by
\begin{equation*}
    \Omega_f = \{ (x,y,z) : x\in (0,L), \ y \in (0,L), \ z \in (-R, 0) \} = \Gamma \times (-R,0). 
\end{equation*}
In the analysis part of the manuscript, we will often assume unit domains, i.e., $L = R = 1$.

\medskip
\noindent{{\bfseries Coupling conditions.}}
The fluid equations \eqref{Stokes} and the fully averaged poroelastic plate equations~\eqref{eq:fully averaged plate model} are linearly coupled; that is, the coupling conditions are imposed along a fixed interface $\Gamma$, which corresponds to the middle surface of the plate, see \autoref{fig:domain FPSI problem}.
To state the coupling conditions we recall that~$\bn$ denotes the unit normal to $\Gamma$, and $\btau_i$, $i = 1,2$ are two linearly independent, unit tangent vectors to $\Gamma$. 
Since the FSI problem considered here is linearly coupled, the unit normal and tangent vectors can be identified with the standard basis  vectors $\bn = {\bf e}_3, \btau_i =  {\bf e}_i$, for $i = 1,2$.
We impose two sets of coupling conditions: the kinematic and dynamic coupling conditions. 

\medskip
\noindent{{\bfseries Kinematic coupling conditions.}}
The kinematic coupling conditions specify the behavior of fluid velocities at the interface $\Gamma$:
\begin{itemize}
    \item Continuity of normal components of fluid velocities relative to the interface motion:
    \begin{equation}\label{ContVel}
        u_p^- = \left. \bu \right|_{\Gamma} \cdot \bn - v.
    \end{equation}
    \item Slip in the tangential component of fluid velocity at the interface, known as the Beavers--Joseph--Saffman condition:
    \begin{equation}\label{BJS}
        \left.\bsigma_f(\bu,\pi)\right|_{\Gamma} \bn  \cdot \btau_i = - \beta \left. \bu \right|_{\Gamma} \cdot \btau_i, \tfor i =1,2,
    \end{equation}
    where $\beta > 0$ is the friction parameter (slip rate), and we recall the tangent vectors $\btau_i$, i = 1,2, to $\Gamma$.
\end{itemize}

\medskip
\noindent{{\bfseries Dynamic coupling conditions.}}
The dynamic coupling conditions describe balance of forces at the interface $\Gamma$:
\begin{itemize}
    \item Continuity of normal components of normal fluid stress (traction) at $\Gamma$ (also known as ``continuity of pressures''):
    \begin{equation}\label{ContP}
        \left.\bsigma_f(\bu,\pi)\right|_{\Gamma}  \bn  \cdot \bn = - q^- = \qbr - q^+.
    \end{equation}
    \item The force driving the elastodynamics of the plate is given by the jump in the normal components of normal stress across the interface (jump in traction): 
    \begin{equation}\label{Dynamic2}
        F_p =  -\left. \bsigma_f(\bu,\pi)\right|_{\Gamma} \bn  \cdot \bn - q^+.
    \end{equation}
    By taking into account the continuity of pressure condition \eqref{ContP}, one gets:
    \begin{equation}\label{Fp}
        F_p = - (q^+ - q^-) = - [q].
    \end{equation}
\end{itemize}

\begin{figure}[h!]
\centering
    \begin{subfigure}[b]{0.48\textwidth}
        \centering
        \begin{tikzpicture}[>=Stealth, line join=round, line cap=round, scale=0.85, every node/.style={scale=0.85}]
            \def\XD{4.5} \def\YD{2.5} \def\ZD{2.5} \def\ang{40} \def\scale{0.5}
            \newcommand{\xyz}[3]{({#1 + #2*\scale*cos(\ang)}, {#3 + #2*\scale*sin(\ang)})}

            \draw[gray!60, fill=gray!5, opacity=0.5] \xyz{0}{0}{-\ZD} -- \xyz{\XD}{0}{-\ZD} -- \xyz{\XD}{\YD}{-\ZD} -- \xyz{0}{\YD}{-\ZD} -- cycle;
            \draw[gray!60, fill=gray!15, opacity=0.5] \xyz{0}{\YD}{-\ZD} -- \xyz{\XD}{\YD}{-\ZD} -- \xyz{\XD}{\YD}{0} -- \xyz{0}{\YD}{0} -- cycle;
            \draw[gray!60, fill=gray!15, opacity=0.5] \xyz{0}{0}{-\ZD} -- \xyz{0}{\YD}{-\ZD} -- \xyz{0}{\YD}{0} -- \xyz{0}{0}{0} -- cycle;
            \node[black] at \xyz{\XD/2}{\YD/2}{-\ZD/2} {\Large $\Omega_f$};

            \draw[black, fill=gray!40, opacity=0.9, line width=1.5pt] \xyz{0}{0}{0} -- \xyz{\XD}{0}{0} -- \xyz{\XD}{\YD}{0} -- \xyz{0}{\YD}{0} -- cycle;

            \draw[gray!60, fill=gray!20, opacity=0.3] \xyz{0}{0}{-\ZD} -- \xyz{\XD}{0}{-\ZD} -- \xyz{\XD}{0}{0} -- \xyz{0}{0}{0} -- cycle;
            \draw[gray!60, fill=gray!20, opacity=0.3] \xyz{\XD}{0}{-\ZD} -- \xyz{\XD}{\YD}{-\ZD} -- \xyz{\XD}{\YD}{0} -- \xyz{\XD}{0}{0} -- cycle;

            \node (Lplus) at \xyz{0.25*\XD}{0.1}{1.2} {$\Gamma_+$};
            \draw[thick, ->] \xyz{0.25*\XD}{0.1}{1.0} -- \xyz{0.25*\XD}{0.1}{0.05};
            \node (Lminus) at \xyz{0.25*\XD}{0.1}{-1.2} {$\Gamma_-$};
            \draw[thick, ->] \xyz{0.25*\XD}{0.1}{-1.0} -- \xyz{0.25*\XD}{0.1}{-0.05};
            \node (gamma_label) at \xyz{\XD+1.0}{\YD/2}{0.3} {$\Gamma$};
            \draw[thick, ->] (gamma_label) -- \xyz{\XD+0.1}{\YD/2}{0};

            \begin{scope}[shift={(\xyz{-1.2}{0}{-1.5})}]
                \draw[thick, ->] (0,0) -- (0.8,0) node[right, scale=0.7] {$\boldsymbol{\tau}_1$};
                \draw[thick, ->] (0,0) -- (0,0.8) node[above, scale=0.7] {$\boldsymbol{n}$};
                \draw[thick, ->] (0,0) -- ({cos(\ang)*0.4}, {sin(\ang)*0.4}) node[right, scale=0.7] {$\boldsymbol{\tau}_2$};
            \end{scope}
        \end{tikzpicture}
        \caption{FPSI with the fully averaged model.}
    \end{subfigure}
    \hfill
    \begin{subfigure}[b]{0.48\textwidth}
        \centering
        \begin{tikzpicture}[>=Stealth, line join=round, line cap=round, scale=0.85, every node/.style={scale=0.85}]
            \def\XD{4.5} \def\YD{2.5} \def\ZD{2.5} \def\H{0.8} \def\Hhalf{0.4} \def\ang{40} \def\scale{0.5}
            \newcommand{\xyz}[3]{({#1 + #2*\scale*cos(\ang)}, {#3 + #2*\scale*sin(\ang)})}

            \draw[gray!60, fill=gray!5, opacity=0.5] \xyz{0}{0}{-\ZD-\Hhalf} -- \xyz{\XD}{0}{-\ZD-\Hhalf} -- \xyz{\XD}{\YD}{-\ZD-\Hhalf} -- \xyz{0}{\YD}{-\ZD-\Hhalf} -- cycle;
            \draw[gray!60, fill=gray!15, opacity=0.5] \xyz{0}{\YD}{-\ZD-\Hhalf} -- \xyz{\XD}{\YD}{-\ZD-\Hhalf} -- \xyz{\XD}{\YD}{-\Hhalf} -- \xyz{0}{\YD}{-\Hhalf} -- cycle;
            \draw[gray!60, fill=gray!15, opacity=0.5] \xyz{0}{0}{-\ZD-\Hhalf} -- \xyz{0}{\YD}{-\ZD-\Hhalf} -- \xyz{0}{\YD}{-\Hhalf} -- \xyz{0}{0}{-\Hhalf} -- cycle;
            \node[black] at \xyz{\XD/2}{\YD/2}{-\ZD/2-\Hhalf} {\Large $\Omega_f$};

            \draw[black, fill=gray!20, opacity=0.8] \xyz{0}{\YD}{-\Hhalf} -- \xyz{\XD}{\YD}{-\Hhalf} -- \xyz{\XD}{\YD}{\Hhalf} -- \xyz{0}{\YD}{\Hhalf} -- cycle;
            \draw[black, fill=gray!20, opacity=0.8] \xyz{0}{0}{-\Hhalf} -- \xyz{0}{\YD}{-\Hhalf} -- \xyz{0}{\YD}{\Hhalf} -- \xyz{0}{0}{\Hhalf} -- cycle;
            \draw[black, dashed, opacity=0.5] \xyz{0}{0}{0} -- \xyz{\XD}{0}{0} -- \xyz{\XD}{\YD}{0} -- \xyz{0}{0}{0};
            \draw[black, dashed, opacity=0.5] \xyz{\XD}{0}{0} -- \xyz{\XD}{\YD}{0} -- \xyz{0}{\YD}{0};
            \draw[black, fill=gray!40, opacity=0.9, line width=1.2pt] \xyz{0}{0}{\Hhalf} -- \xyz{\XD}{0}{\Hhalf} -- \xyz{\XD}{\YD}{\Hhalf} -- \xyz{0}{\YD}{\Hhalf} -- cycle;
            \draw[black, fill=gray!30, opacity=0.8] \xyz{\XD}{0}{-\Hhalf} -- \xyz{\XD}{\YD}{-\Hhalf} -- \xyz{\XD}{\YD}{\Hhalf} -- \xyz{\XD}{0}{\Hhalf} -- cycle;
            \draw[black, fill=gray!35, opacity=0.8] \xyz{0}{0}{-\Hhalf} -- \xyz{\XD}{0}{-\Hhalf} -- \xyz{\XD}{0}{\Hhalf} -- \xyz{0}{0}{\Hhalf} -- cycle;
            \node[black] at \xyz{0.5*\XD}{0.5*\YD}{0.25} {\Large $\Omega_p$};

            \draw[gray!60, fill=gray!20, opacity=0.3] \xyz{0}{0}{-\ZD-\Hhalf} -- \xyz{\XD}{0}{-\ZD-\Hhalf} -- \xyz{\XD}{0}{-\Hhalf} -- \xyz{0}{0}{-\Hhalf} -- cycle;
            \draw[gray!60, fill=gray!20, opacity=0.3] \xyz{\XD}{0}{-\ZD-\Hhalf} -- \xyz{\XD}{\YD}{-\ZD-\Hhalf} -- \xyz{\XD}{\YD}{-\Hhalf} -- \xyz{\XD}{0}{-\Hhalf} -- cycle;

            \node (Lplus) at \xyz{0.25*\XD}{0.1}{1.6} {$\Gamma_+$};
            \draw[thick, ->] \xyz{0.25*\XD}{0.1}{1.4} -- \xyz{0.25*\XD}{0.1}{\Hhalf+0.05};
            \node (Lminus) at \xyz{0.25*\XD}{0.1}{-1.6} {$\Gamma_-$};
            \draw[thick, ->] \xyz{0.25*\XD}{0.1}{-1.4} -- \xyz{0.25*\XD}{0.1}{-\Hhalf-0.05};
            \node (gamma_label) at \xyz{\XD+1.0}{\YD/2}{0.3} {$\Gamma$};
            \draw[thick, ->] (gamma_label) -- \xyz{\XD+0.1}{\YD/2}{0};

        \end{tikzpicture}
        \caption{FPSI with the half-averaged model.}
    \end{subfigure}

\caption{Comparison between the FPSI problems involving the fully averaged (left) and half-averaged (right) poroelastic plate models. 
In the fully averaged case, the plate is reduced to a codimension-one interface $\Gamma$, whereas in the half-averaged case, the plate retains its physical thickness $H$, defining a volumetric domain $\Omega_p$.}
\label{fig:domain FPSI problem}
\end{figure}
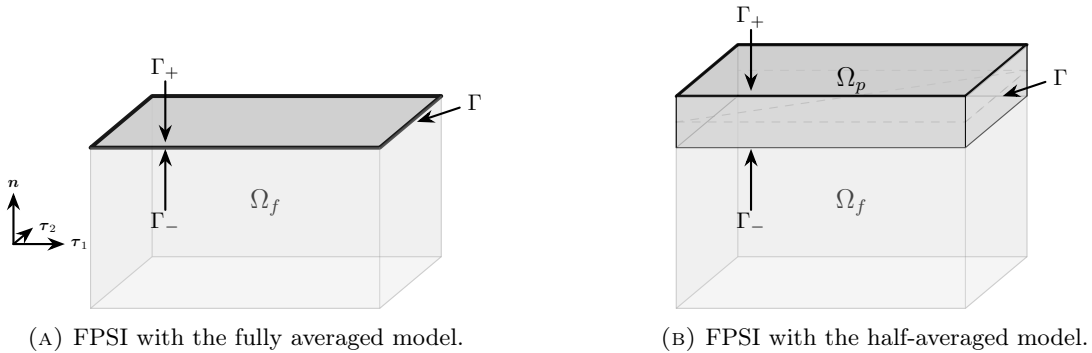

\medskip
\noindent{{\bfseries Initial and boundary conditions.}}
We supplement the coupled problem \eqref{eq:fully averaged plate model}, \eqref{Stokes}, \eqref{ContVel}, \eqref{BJS}, \eqref{ContP}, and \eqref{Fp} with the following initial and boundary conditions: 

\medskip
\noindent{{\bfseries Initial conditions:}}
\begin{equation}\label{init}
w|_{t = 0} = w_0, \ \partial_t w|_{t = 0} = v_0, \ 
[q]|_{t = 0} = [q]_0, \ \bar{q}|_{t = 0} = \bar{q}_0, \
\bu|_{t = 0} = \bu_0.
\end{equation}

\medskip

\noindent{{\bfseries Boundary conditions:}}
\begin{enumerate}
    \item{On $\partial\Gamma$, we assume the clamped plate boundary conditions:}
    \begin{equation}\label{clamped}
        w = \del_{\bnu} w = 0  \ton \del \Gamma, 
    \end{equation}
    where $\bnu$ is the outer normal to $\partial\Gamma$.
    \item On the outer (top) side of the plate $\Gamma^+$, we prescribe the pressure $q^+$ (drained boundary condition):
    \begin{equation}\label{top boundary}
        q|_{\Gamma^+}  = q^+.
    \end{equation}
    \item On the side boundaries of the fluid domain $\Gamma_{f}^s$, where $x = 0,L$ and $y = 0,L$, we prescribe periodic boundary conditions for the fluid velocity and pressure: for all $t \in (0,T)$
    \begin{equation}\label{side}
        \begin{array}{l}
            \bu(t,x,0,z) = \bu (t,x,L,z), \tforall x \in (0,L) \tand \bu(t,0,y,z) = \bu(t,L,y,z) \tforall y \in (0,L),\\
            \pi(t,x,0,z) = \pi (t,x,L,z), \tforall x \in (0,L) \tand \pi(t,0,y,z) = \pi(t,L,y,z) \tforall y \in (0,L).
        \end{array}
    \end{equation}
    \item On the bottom boundary $\Gamma_f^b = \Gamma \times \{-R\}$ of the fluid domain, the no-slip boundary conditions is prescribed:
\begin{equation}\label{bottom}
    \bu = 0, \ton (0,T) \times \Gamma_f^b.
\end{equation}
\end{enumerate}

\medskip

\noindent{{\bfseries The coupled problem after conditions \eqref{ContVel} and \eqref{Fp} have been taken into account.}}
We note that we can enforce the kinematic coupling condition \eqref{ContVel} and the dynamic coupling condition~\eqref{Fp} in a strong sense by using those conditions as the boundary conditions in the differential form~\eqref{eq:fully averaged plate model} of the fully averaged plate model.
After rewriting the plate equations using these coupling conditions, and in the absence of external forces on the fluid, the resulting plate and fluid equations read
\begin{equation}\label{eq:FPSI probl}
    \left\{
    \begin{aligned}
        H \rho_p \del_t v + H^3 D \Delta_{\btau}^2 w + H \gamma^p w + H^2 \frac{\alpha^p}{12} \Delta_{\btau} [q] + \qbr
        &= 0, &&\tin (0,T) \times \Gamma,\\
        H \frac{c_0^p}{12} \del_t [q] - H^2 \frac{\alpha^p}{12} \Delta_{\btau} v + \frac{4 \kappa^p}{H} [q] + \frac{6 \kappa^p}{H} \qbar + \left. \bu \right|_{\Gamma} \cdot \bn - v
        &= \frac{6}{H}\kappa^p q^+, &&\tin (0,T) \times \Gamma,\\
        H c_0^p \del_t \qbar + \frac{12 \kappa^p}{H} \qbar + \frac{6 \kappa^p}{H} [q]
        &= {\frac{12}{H}\kappa^p q^+}, &&\tin (0,T) \times \Gamma,\\
        \rho_f \del_t \bu - \nabla \cdot \bsigma_f(\bu,\pi)
        &= 0, &&\tin (0,T) \times \Omega_f,\\
        \nabla \cdot \bu
        &= 0, &&\tin (0,T) \times \Omega_f.
    \end{aligned}
    \right.
\end{equation}
This problem is supplemented with the traction coupling condition 
\begin{equation}\label{DynCouplingTraction}
    \bsigma_f(\bu,\pi)\bn = (\qbr-q^+) \bn - \beta \sum_{i=1}^2 (\bu|_\Gamma \cdot \btau_i) \btau_i,
\end{equation}
and the initial and boundary conditions \eqref{init}--\eqref{bottom}.

This will be the foundation for the derivation of the energy estimate and weak formulation, discussed in the next section.

\section{Weak solutions}\label{sec:weak sols}

In this section, we discuss the existence of finite-energy weak solutions to the linearly coupled FPSI problem \eqref{eq:fully averaged plate model}--\eqref{bottom}.
This weak formulation will be the basis for the design of the finite element method numerical scheme, discussed in \autoref{sec:numerics}.

\subsection{Formal energy estimate and weak formulation}
\

We begin with a formal energy estimate, which motivates the choice of the finite-energy weak solution spaces.
To derive this estimate and to formulate the notion of weak solutions, we introduce the following notation. 
For a domain $\cO \subset \R^d$, $d = 2,3$, we denote by $((\cdot,\cdot))_{\cO}$ the inner product in $\rL^2(0,T;\rL^2(\cO))$, and by $(\cdot,\cdot)_{\cO}$ the inner product in $\rL^2(\cO)$.

\medskip
{\noindent{\bfseries Energy estimate.}}
Testing the equation for the elastodynamics of the plate \eqref{eq:FPSI probl}$_1$ by $v = w_t$, the equation for the jump in the pressure \eqref{eq:FPSI probl}$_2$ by $\qbr$, the equation for the average of the pressure \eqref{eq:FPSI probl}$_3$ by $\qbar$, and the fluid equation \eqref{eq:FPSI probl}$_4$ by $\bu$, integrating by parts, and making use of the boundary and coupling conditions specified in \eqref{ContVel}--\eqref{bottom}, one obtains:
\begin{equation}\label{EnergyEquality}
    \begin{aligned}
        &\quad \frac{1}{2}\frac{\rd}{\rd t}\biggl(H \rho_p \| w_t \|_{\rL^2(\Gamma)}^2 + H \gamma^p \| w \|_{\rL^2(\Gamma)}^2 + H^3 D \| \Delta_{\btau} w \|_{\rL^2(\Gamma)}^2\\
        &\qquad + c_0^p \Bigl(H \| \qbar \|_{\rL^2(\Gamma)}^2 + \frac{H}{12} \| \qbr \|_{\rL^2(\Gamma)}^2\Bigr) + \rho_f \| \bu \|_{\rL^2(\Omega_f)}^2\biggr)\\
        &\quad + \biggl(\frac{\kappa^p}{H} \| \qbr \|_{\rL^2(\Gamma)}^2 + \frac{12 \kappa^p}{H} \Bigl\| \qbar + \frac{1}{2} \qbr \Bigr\|_{\rL^2(\Gamma)}^2\biggr) + 2 \mu_f \| \bD(\bu) \|_{\rL^2(\Omega_f)}^2 + \beta \| \bu \cdot \btau \|_{\rL^2(\Gamma)}^2\\
        &=  \frac{6}{H} \kappa^p ((q^+,\qbr))_\Gamma 
        + \frac{12}{H}\kappa^p((q^+,\qbar))_\Gamma 
        {- ((q^+,\bu\cdot \bn))_\Gamma.}
    \end{aligned}
\end{equation}
We estimate the right-hand side of \eqref{EnergyEquality} using an absorption argument. 
For the first two terms, Young's inequality yields that
\begin{equation}\label{eq:absorption qbr & qbar}
    \begin{aligned}
        \frac{6}{H} \kappa^p (q^+,\qbr)_\Gamma + \frac{12}{H}\kappa^p(q^+,\qbar)_\Gamma
        &= \frac{12}{H}\kappa^p\left(q^+,\qbar + \frac{1}{2}\qbr\right)_\Gamma\\
        &\le \frac{6 \kappa^p}{H} \| q^+ \|_{\rL^2(\Gamma)}^2 + \frac{6 \kappa^p}{H} \ \left\| \qbar + \frac{1}{2} \qbr \right\|_{\rL^2(\Gamma)}^2.
    \end{aligned}
\end{equation}
The second term can then be absorbed into the left-hand side.

To estimate the term $-(q^+,\bu\cdot \bn)_\Gamma$, we combine Young's inequality with the trace theorem, Poincar\'{e}'s inequality, and Korn's inequality. 
This yields the existence of constants $C_1(\delta)$ and $C_2 > 0$, depending only on the fluid domain $\Omega_f$, such that
\begin{equation}\label{eq:absorption fluid bdry term}
    - (q^+,\bu\cdot \bn)_\Gamma \le C_1(\delta) \cdot \| q^+ \|_{\rL^2(\Gamma)}^2 + \delta \cdot \| \bu\cdot \bn \|_{\rL^2(\Gamma)}^2 \le C_1(\delta) \cdot \| q^+ \|_{\rL^2(\Gamma)}^2 + C_2 \delta \cdot \| \bD(\bu) \|_{\rL^2(\Gamma)}^2.
\end{equation}
Thus, choosing $\delta \coloneqq \frac{\mu_f}{C_2}$ in \eqref{eq:absorption fluid bdry term}, we obtain that
\begin{equation}\label{eq:absorption fluid vel}
     - (q^+,\bu\cdot \bn)_\Gamma \le C_1 \cdot \| q^+ \|_{\rL^2(\Gamma)}^2 + \mu_f \cdot \| \bD(\bu) \|_{\rL^2(\Gamma)}^2.
\end{equation}
By combining the above estimates one obtains the following energy inequality: 
\begin{equation}\label{EnergyInequality}
    \frac{\rd}{\rd t} \left({\mathcal{E}}_k(t) +  {\mathcal{E}}_p(t)\right) + {\mathcal{D}}(t) \le C_\mathrm{data}(t),
\end{equation}
where $C_\mathrm{data}(t) >0$ depends only on the data in the problem, 
and ${\mathcal{E}}_k(t)$, ${\mathcal{E}}_p(t)$, and ${\mathcal{D}}(t)$ are given by:
\begin{itemize}
    \item Kinetic energy:
    \begin{equation*}
        {\mathcal{E}}_k(t) \coloneqq \frac{1}{2}\left(\rho_f \| \bu \|_{\rL^2(\Omega_f)}^2 +  \rho_p H \| w_t \|_{\rL^2(\Gamma)}^2  \right).
    \end{equation*}
    \item Potential energy:
    \begin{equation*}
        {\mathcal{E}}_p(t) \coloneqq \frac{1}{2}\left(\gamma^p H \| w \|_{\rL^2(\Gamma)}^2 +  D H^3 \| \Delta_{\btau} w \|_{\rL^2(\Gamma)}^2 + c_0^p \Bigl(H \| \qbar \|_{\rL^2(\Gamma)}^2 + \frac{H}{12} \| \qbr \|_{\rL^2(\Gamma)}^2\Bigr)  \right).
    \end{equation*}
    \item Dissipation:
    \begin{equation*}
        {\mathcal{D}}(t) \coloneqq \biggl(\frac{\kappa^p}{H} \| \qbr \|_{\rL^2(\Gamma)}^2 + \frac{6 \kappa^p}{H} \Bigl\| \qbar + \frac{1}{2} \qbr \Bigr\|_{\rL^2(\Gamma)}^2\biggr) + \mu_f \| \bD(\bu) \|_{\rL^2(\Omega_f)}^2 + \beta \| \bu \cdot \btau \|_{\rL^2(\Gamma)}^2.
    \end{equation*}
\end{itemize}

\medskip
\noindent{{\bfseries Weak formulation.}}
Based on the above energy estimate, we introduce the following finite-energy weak solution spaces.
As in \autoref{sec:FPSI problem}, we consider $T > 0$ in the sequel.

The solution space for the {\bf{plate displacement}} $w$ is given by
\begin{equation*}
    \cV_w \coloneqq  \rL^\infty(0,T;\rH_0^2(\Gamma)) \cap \rW^{1,\infty}(0,T;\rL^2(\Gamma)),
\end{equation*}
where $\rH_0^2(\Gamma) \coloneqq \overline{\rC_c^\infty(\Gamma)}^{\| \cdot \|_{\rH^2(\Gamma)}}$ captures the clamped boundary conditions.

The solution spaces for the {\bf{jump in the pressure}} $q$, and the {\bf{average pressure}} $\bar{q}$ are given by
\begin{equation*}
    \cV_{\qbr} \coloneqq \rL^\infty(0,T;\rL^2(\Gamma)) \tand \cV_{\qbar} \coloneqq \rL^\infty(0,T;\rL^2(\Gamma)).
\end{equation*}
The solution space for the {\bf{free fluid velocity}} is defined by
\begin{equation*}
    \cV_f \coloneqq \rL^\infty(0,T;H) \cap \rL^2(0,T;V),
\end{equation*}
where the spaces $H$ and $V$ are given by:
\begin{equation*}
    \begin{aligned}
        H 
        &\coloneqq \left\{\bu \in \rL^2(\Omega_f)^3 : \nabla \cdot \bu = 0, \enspace \left.(\bu \cdot \bn)\right|_{z = -R} = 0 \right\} \tand\\
        V
        &\coloneqq \left\{\bu \in \rH^1(\Omega_f)^3 : \nabla \cdot \bu = 0, \enspace \left. \bu \right|_{z = -R} = 0, \enspace \bu \text{ periodic on } \Gamma_f^s\right\}.
    \end{aligned}
\end{equation*}
This leads to the {\bf{weak solution space}}
\begin{equation*}
    \cV_\sol \coloneqq \cV_w 
    \times \cV_{\qbr} \times \cV_{\qbar} \times \cV_f.
\end{equation*}
Accordingly, we define the test space by
\begin{equation*}
    \cV_\test \coloneqq \rC_c^1\bigl([0,T);\cV_w 
    \times \cV_{\qbr} \times \cV_{\qbar} \times \cV_f\bigr).
\end{equation*}

The weak formulation is obtained by testing the first four equations in \eqref{eq:FPSI probl} with the test functions $(\varphi,\sbr,\sbar,\bv)$ corresponding to the variables $(w,\qbr,\qbar,\bu)$.
Integration by parts, together with the use of the coupling and boundary conditions, yields the weak formulation. 

Two integrals require special attention.
First, in equation \eqref{eq:FPSI probl}$_2$, after multiplication by the test function $\sbr$, the second term can be treated by recalling that $v = \partial_t w$ and using $\sbr \in \rH_0^2(\Gamma)$ with compact support on $[0,T)$, to obtain:
\begin{equation*}
    \begin{aligned}
        -\int_0^T\int_\Gamma H^2 \frac{\alpha^p}{12}\Delta_{\btau} v \sbr \srd S \srd t
        &= -H^2 \frac{\alpha^p}{12}\int_0^T\int_\Gamma \del_t (\Delta_{\btau} w \sbr) - \Delta_{\btau} w \del_t \sbr \srd S \srd t\\
        &= H^2 \frac{\alpha^p}{12} ((\Delta_{\btau} w, \del_t \sbr))_\Gamma + H^2 \frac{\alpha^p}{12} \left.(\Delta_{\btau} w,\sbr)_\Gamma \right|_{t = 0}.
    \end{aligned}
\end{equation*}

The second important integral is the fluid Cauchy stress term appearing in equation \eqref{eq:FPSI probl}$_4$:
\begin{equation*}
    - \int_0^T\int_{\Omega_f} \nabla \cdot \bsigma_f(\bu,\pi) \cdot \bv \srd \bx \srd t = \mu_f \int_0^T \int_{\Omega_f} \bD(\bu) : \bD(\bv) \srd \bx \srd t - \int_0^T\int_{\del \Omega_f} \bsigma_f(\bu,\pi) \bn \cdot \bv \srd S \srd t.
\end{equation*}
For the term involving the boundary, we use that $\bv = 0$ at the lower fluid boundary $\Gamma_f^b$, as well as the periodic boundary conditions at the side boundary $\Gamma_f^s$.
We further invoke the Beavers--Joseph--Saffman condition \eqref{BJS}, the continuity of normal velocity components \eqref{ContP} as well as the relation $q^- = q^+ - \qbr$ to obtain:
\begin{equation*}
    \begin{aligned}
        &\quad \int_0^T\int_{\del \Omega_f} \bsigma_f(\bu,\pi) \bn \cdot \bv \srd S \srd t
         =  \int_0^T \int_{\Gamma} \bsigma_f(\bu,\pi) \bn \cdot \bv \srd S \srd t\\
        &= \sum_{i = 1}^2\int_0^T\int_{\Gamma} \bsigma_f(\bu,\pi) \bn \cdot \btau_i (\bv \cdot \btau_i) \srd S \srd t
        + \int_0^T \int_{\Gamma} \bsigma_f(\bu,\pi) \bn \cdot \bn (\bv \cdot \bn) \srd S\srd t\\
        &= - \beta \sum_{i = 1}^2 \int_0^T\int_\Gamma \bu \cdot \btau_i (\bv \cdot \btau_i) \srd S \srd t
        + \int_0^T \int_\Gamma (\qbr - q^+) (\bv \cdot \bn) \srd S\srd t \srd t
        \\
        &= - \beta \int_0^T \int_\Gamma (\bu \cdot \btau) (\bv \cdot \btau) \srd S \srd t
          + \int_0^T \int_\Gamma (\qbr - q^+) (\bv \cdot \bn) \srd S \srd t,
    \end{aligned}
\end{equation*}
where we used $\btau$ to abbreviate the sum of the tangential vectors $\btau_1$ and $\btau_2$.

\begin{defn}\label{def:weak form cont-in-time}
We say that $(w,\qbr,\qbar,\bu) \in \cV_{\sol}$ is a weak solution of \eqref{eq:fully averaged plate model}--\eqref{bottom} if, for all test functions $(\varphi,\sbr,\sbar,\bv) \in \cV_{\test}$, the following holds:
\begin{equation}\label{WeakFormulation}
    \cI_{w}(\varphi) + \cI_{[q]}([s]) + \cI_{\qbar}(\sbar) + \cI_{\bu}(\bv) = \cI_{0}(\varphi,\sbr,\sbar,\bv) + \cI_{\mathrm{data}}([s],\sbar,\bv),
\end{equation}
where:
\begin{equation*}
    \begin{aligned}
        &\cI_{w}(\varphi)\coloneqq
        -H \rho_p ((\del_t w,\del_t \varphi))_\Gamma + H^3 D ((\Delta_{\btau} w,\Delta_{\btau} \varphi))_\Gamma + H \gamma^p ((w,\varphi))_\Gamma + H^2 \frac{\alpha^p}{12}((\qbr,\Delta_{\btau} \varphi))_\Gamma + ((\qbr,\varphi))_\Gamma,\\
        &\cI_{[q]}([s]) \coloneqq - H \frac{c_0^p}{12}((\qbr,\del_t \sbr))_\Gamma + H^2 \frac{\alpha^p}{12}((\Delta_{\btau} w,\del_t \sbr))_\Gamma + \frac{4 \kappa^p}{H}((\qbr,\sbr))_\Gamma + \frac{6 \kappa^p}{H}((\qbar,\sbr))_\Gamma\\
        & \qquad\qquad \ \ + ((\left.\bu\right|_\Gamma \cdot \bn,\sbr))_\Gamma + ((w,\del_t\sbr))_\Gamma,\\
        &\cI_{\qbar}(\sbar) \coloneqq -H c_0^p ((\qbar,\del_t \sbar))_\Gamma + \frac{12 \kappa^p}{H}((\qbar,\sbar))_\Gamma + \frac{6 \kappa^p}{H}((\qbr,\sbar))_\Gamma,\\
        &\cI_{\bu}(\bv) \coloneqq - \rho_f ((\bu,\del_t \bv))_{\Omega_f} + 2 \mu_f ((\bD(\bu),\bD(\bv)))_{\Omega_f} + \beta ((\bu \cdot \btau,\bv \cdot \btau))_\Gamma - ((\qbr,\bv \cdot \bn))_\Gamma,\\
        &\cI_{0}(\varphi,\sbr,\sbar,\bv) \coloneqq H \rho_p \left.(v,\varphi)_\Gamma\right|_{t=0} + H \frac{c_0^p}{12}\left.(\qbr,\sbr)_\Gamma\right|_{t = 0} - H^2 \frac{\alpha^p}{12}\left.(\Delta_{\btau} w,\sbr)_\Gamma\right|_{t=0} - \left.(w,\sbr)_\Gamma\right|_{t=0}\\
        &\qquad\qquad \qquad \ + H c_0^p \left.(\qbar,\sbar)_\Gamma\right|_{t=0} + \rho_f \left.(\bu,\bv)_{\Omega_f}\right|_{t=0}, \tand\\
        &\cI_{\mathrm{data}}([s],\sbar,\bv) \coloneqq 6 \kappa^p ((q^+,\sbr))_\Gamma + \frac{12}{H}\kappa^p((q^+,\sbar))_\Gamma - ((q^+,\bv\cdot \bn))_\Gamma,
    \end{aligned}
\end{equation*}
with
\begin{equation*}
    w(0,\cdot) = w_0, \enspace \del_t w(0,\cdot) = v_0, \enspace \qbr(0,\cdot) = \qbr_0, \enspace \qbar(0,\cdot) = \qbar_0, \tand \bu(0,\cdot) = \bu_0.
\end{equation*}
\end{defn}

To prove the existence of weak solutions, we follow an approach similar to that in \cite{BCMW:21}. We construct approximate solutions via a time semidiscretization of the problem and show that the corresponding approximate weak formulation admits a unique solution after an appropriate rescaling of the test functions. 

We then establish uniform \emph{a priori} energy estimates for the approximate solutions, which allow us to extract subsequences converging to a weak solution of the coupled problem in suitable weak and weak* topologies. We outline the main ideas below, omitting details that closely follow those in \cite{BCMW:21}.

\subsection{Construction of approximate solutions}
\

We semidiscretize the problem in time by partitioning the interval $[0,T]$ into $N$ subintervals of equal length $\Delta t = \frac{T}{N} > 0$, with $N \in \mathbb{N}$.
For a fixed $N$, we define a sequence of approximate solutions $(w_N^n,\qbr_N^n,\qbar_N^n,\bu_N^n)$ for $n = 0,\dots,N$.

To simplify notation, we omit the subscript ``$N$'' in what follows.
We introduce discrete time derivatives based on the implicit Euler scheme, namely,
\begin{equation}\label{DiscreteTimeDer}
    \Dot{w}^{n+1} = \frac{w^{n+1}-w^n}{\Delta t}.
\end{equation}
The solution space for the discretized problem is
\begin{equation*}
    \cV_\sd = \rH_0^2(\Gamma) 
    \times \rL^2(\Gamma) \times \rL^2(\Gamma) \times V, 
\end{equation*}
and we equip this space with the norm
\begin{equation*}
    \| (w,\qbr,\qbar,\bu) \|_{\sd}^2 \coloneqq \| \Delta w \|_{\rL^2(\Gamma)}^2 + \| \qbr \|_{\rL^2(\Gamma)}^2 + \| \qbar \|_{\rL^2(\Gamma)}^2 + \| \bu \|_{\rH^1(\Omega_f)}^2.
\end{equation*}

The following lemma can be obtained in a straightforward way (assuming the clamped boundary conditions for $w$).

\begin{lem}
The space $(\cV_\sd,\|\cdot\|_{\sd})$ is a Hilbert space, and the norm $\|\cdot\|_{\sd}$ is equivalent to the standard norm on $\rH_0^2(\Gamma) \times \rL^2(\Gamma) \times \rL^2(\Gamma) \times V$.
\end{lem}

At each time step $t^{n} = n [\Delta t]$, for $n = 1,\dots,N$, we ``solve'' the weak formulation of an approximate problem, and look for a solution that is obtained in an appropriate limit of approximate solutions as $N \to \infty$. 

The approximate problem is obtained by semidiscretizing the time derivatives as in \eqref{DiscreteTimeDer} to obtain the following weak formulation:
\begin{equation}\label{eq:first weak form semidiscr form}
    \begin{aligned}
        &\quad H \rho_p (\Dot{w}^{n+1} - \Dot{w}^n,\varphi)_\Gamma + H^3 D [\Delta t] (\Delta_{\btau} w^{n+1},\Delta_{\btau} \varphi)_\Gamma + H \gamma^p [\Delta t] (w^{n+1},\varphi)_\Gamma\\
        &\quad + H^2 \frac{\alpha^p}{12} [\Delta t] (\qbr^{n+1},\Delta_{\btau} \varphi)_\Gamma + [\Delta t] (\qbr^{n+1},\varphi)_\Gamma + H \frac{c_0^p}{12}(\qbr^{n+1} - \qbr^n,\sbr)_\Gamma\\
        &\quad - H^2 \frac{\alpha^p}{12} [\Delta t] (\Delta_{\btau} \Dot{w}^{n+1},\sbr)_\Gamma + \frac{4 \kappa^p}{H} [\Delta t] (\qbr^{n+1},\sbr)_\Gamma + \frac{6 \kappa^p}{H} [\Delta t](\qbar^{n+1},\sbr)_\Gamma\\
        &\quad + [\Delta t] (\left.\bu^{n+1}\right|_\Gamma \cdot \bn,\sbr)_\Gamma - [\Delta t](\Dot{w}^{n+1},\sbr)_\Gamma + H c_0^p (\qbar^{n+1} - \qbar^n,\sbar)_\Gamma+ \frac{12 \kappa^p}{H}[\Delta t] (\qbar^{n+1},\sbar)_\Gamma\\
        &\quad + \frac{6 \kappa^p}{H}[\Delta t] (\qbr^{n+1},\sbar)_\Gamma + \rho_f (\bu^{n+1} - \bu^n,\bv)_{\Omega_f} + 2 \mu_f [\Delta t](\bD(\bu^{n+1}),\bD(\bv))_{\Omega_f}\\
        &\quad + \beta [\Delta t](\bu^{n+1} \cdot \btau,\bv \cdot \bv)_\Gamma - [\Delta t](\qbr^{n+1},\bv \cdot \bn)_\Gamma\\
        &= \frac{6 \kappa^p}{H} [\Delta t]((q^+)^{n+1},\sbr)_\Gamma + \frac{12 \kappa^p}{H} [\Delta t]((q^+)^{n+1},\sbar)_\Gamma - [\Delta t]((q^+)^{n+1},\bv\cdot \bn)_\Gamma, 
    \end{aligned}
\end{equation} 
for all $(\varphi,\sbr,\sbar,\bv) \in \cV_\sd$.
However, as in \cite{BCMW:21}, due to the mixed scaling of the problem where a hyperbolic elastodynamics equation is coupled to the remaining nonhyperbolic system of PDEs, the resulting weak formulation does not have a coercive structure. 
As in \cite{BCMW:21}, we get around this difficulty by rescaling the test functions for the nonhyperbolic part of the problem, namely, for the jump in the pressure, the average of the pressure, and the fluid velocity:
\begin{equation*}
    \sbr \mapsto [\Delta t] \sbr, \enspace \sbar \mapsto [\Delta t] \sbar \tand \bv \mapsto [\Delta t] \bv.
\end{equation*}
This modification does not affect the admissibility of the test space.
After explicitly writing out the finite difference approximations of the time derivatives, the resulting weak formulation can be written as follows.

\begin{defn}\label{def:weak form rescaled}
We say that $(w^{n+1},\qbr^{n+1},\qbar^{n+1},\bu^{n+1}) \in \cV_\sd$ is a weak solution to the semidiscretized linear problem \eqref{eq:fully averaged plate model}--\eqref{bottom} if for every test function $(\varphi,[\Delta t]\sbr,[\Delta t]\sbar,[\Delta t]\bv) \in \cV_\sd$, the following holds 
\begin{equation}\label{eq:weak form rescaled}
    \begin{aligned}
        &\quad \cI_{w^{n+1}}(\varphi)  + \cI_{[q]^{n+1}}([\Delta t][s])  + \cI_{\qbar^{n+1}}([\Delta t]\sbar)  + \cI_{\bu^{n+1}}([\Delta t]\bv) \\ 
        &=\cI_{n}(\varphi, [\Delta t]\sbr,[\Delta t]\sbar,[\Delta t] \bv) + \cI_{\mathrm{data}}([\Delta t][s],[\Delta t]\sbar,[\Delta t]\bv),
    \end{aligned}
\end{equation}
where
\begin{equation}\label{eq:Integrals1}
    \begin{aligned}
        &\cI_{w^{n+1}}(\varphi)\coloneqq
         H \rho_p (w^{n+1},\varphi)_\Gamma + H^3 D [\Delta t]^2 (\Delta_{\btau} w^{n+1},\Delta_{\btau} \varphi)_\Gamma 
        + H \gamma^p [\Delta t]^2 (w^{n+1},\varphi)_\Gamma\\
        &\qquad \qquad+ [\Delta t]^2 (\qbr^{n+1},\varphi)_\Gamma,\\
        &\cI_{[q]^{n+1}}([\Delta t][s]) \coloneqq H \frac{c_0^p}{12} [\Delta t] (\qbr^{n+1},[\Delta t]\sbr)_\Gamma - H^2 \frac{\alpha^p}{12} [\Delta t] (\Delta_{\btau} w^{n+1},[\Delta t]\sbr)_\Gamma\\
        &\qquad \qquad + \frac{4 \kappa^p}{H} [\Delta t]^2 (\qbr^{n+1},[\Delta t]\sbr)_\Gamma +\frac{6 \kappa^p}{H} [\Delta t]^2(\qbar^{n+1},[\Delta t]\sbr)_\Gamma\\
        &\qquad \qquad + [\Delta t]^2 (\left.\bu^{n+1}\right|_\Gamma \cdot \bn,[\Delta t]\sbr)_\Gamma - [\Delta t] (w^{n+1},[\Delta t]\sbr)_\Gamma,\\
        &\cI_{\qbar^{n+1}}([\Delta t]\sbar) \coloneqq
         H c_0^p [\Delta t] (\qbar^{n+1},[\Delta t]\sbar)_\Gamma + \frac{12 \kappa^p}{H} [\Delta t]^2 (\qbar^{n+1},[\Delta t]\sbar)_\Gamma + \frac{6 \kappa^p}{H} [\Delta t]^2 (\qbr^{n+1},[\Delta t]\sbar)_\Gamma,\\
        &\cI_{\bu^{n+1}}([\Delta t]\bv) \coloneqq
         \rho_f [\Delta t] (\bu^{n+1},[\Delta t]\bv)_{\Omega_f}
        + 2 \mu_f [\Delta t]^2 (\bD(\bu^{n+1}),\bD([\Delta t]\bv))_{\Omega_f}\\
        &\qquad \qquad + \beta [\Delta t]^2 (\bu^{n+1} \cdot \btau,[\Delta t]\bv \cdot \btau)_\Gamma  - [\Delta t]^2 (\qbr^{n+1},[\Delta t] \bv \cdot \bn)_\Gamma,
    \end{aligned}
\end{equation}
and
\begin{equation}\label{eq:Integrals2}
    \begin{aligned}
        &\cI_{n}(\varphi, [\Delta t]\sbr,[\Delta t]\sbar,[\Delta t] \bv)\coloneqq H \rho_p (w^{n},\varphi)_\Gamma + H \rho_p [\Delta t] (\Dot{w}^{n},\varphi)_\Gamma + H \frac{c_0^p}{12} [\Delta t] (\qbr^{n},[\Delta t]\sbr)_\Gamma\\
        &\qquad \qquad  + H^2 \frac{\alpha^p}{12} [\Delta t] (\Delta_{\btau} w^{n},[\Delta t]\sbr)_\Gamma
        - [\Delta t] (w^{n},[\Delta t]\sbr)_\Gamma + H c_0^p [\Delta t] (\qbar^{n},[\Delta t]\sbar)_\Gamma\\
        &\qquad \qquad + \rho_f [\Delta t] (\bu^{n},[\Delta t]\bv)_{\Omega_f}, \tand\\
        &\cI_{\mathrm{data}}([\Delta t]\sbr,[\Delta t]\sbar,[\Delta t] \bv) \coloneqq \frac{6}{H} \kappa^p [\Delta t]^2 ((q^+)^{n+1},[\Delta t]\sbr)_\Gamma+ \frac{12}{H} \kappa^p [\Delta t]^2 ((q^+)^{n+1},[\Delta t]\sbar)_\Gamma\\
        &\qquad \qquad -  [\Delta t]^2 ((q^+)^{n+1},[\Delta t]\bv \cdot \bn)_\Gamma,
    \end{aligned}
\end{equation}
with the initial data specified in \eqref{init}.
\end{defn}

We first show that a weak solution of the semidiscretized problem specified in \autoref{def:weak form rescaled}, exists. 
To this end, we introduce a bilinear form $a_{n+1}(\cdot,\cdot) \colon \cV_\sd \times \cV_\sd \to \R$ associated with the left-hand side of~\eqref{eq:weak form rescaled}, as well as a linear functional $F_n(\cdot) \colon \cV_\sd \to \R$ representing the right-hand side. 
We then verify that the hypotheses of the Lax--Milgram lemma are satisfied, which yields the existence of a weak solution.

Let $a_{n+1}$ be the bilinear form defined by:
\begin{equation*}
    a_{n+1}((w^{n+1},\qbr^{n+1},\qbar^{n+1},\bu^{n+1}),(\varphi,\sbr,\sbar,\bv)) \coloneqq \cI_{w^{n+1}}(\varphi) + \cI_{[q]^{n+1}}([\Delta t][s]) + \cI_{\qbar^{n+1}}([\Delta t]\sbar)  + \cI_{\bu^{n+1}}([\Delta t]\bv),
\end{equation*}
and let $F_n(\varphi,\sbr,\sbar,\bv)$ be defined as follows:
\begin{equation*}
    F_n(\varphi,\sbr,\sbar,\bv) \coloneqq \cI_{n}(\varphi, [\Delta t]\sbr,[\Delta t]\sbar,[\Delta t] \bv) + \cI_{\mathrm{data}}([\Delta t][s],[\Delta t]\sbar,[\Delta t]\bv),
\end{equation*}
where $\cI_{w^{n+1}}, \cI_{[q]^{n+1}}, \cI_{\qbar^{n+1}}$, $\cI_{\bu^{n+1}}$, $\cI_{n}$, and $\cI_{\mathrm{data}}$ have been introduced in \eqref{eq:Integrals1} and \eqref{eq:Integrals2}.

\begin{lem}
The bilinear form $a_{n+1}(\cdot,\cdot) \colon \cV_\sd \times \cV_\sd \to \R$ is continuous and coercive, and for the linear form $F_n(\cdot) \colon \cV_\sd \to \R$, it holds that $F_n \in \cV_\sd'$.    
\end{lem}

\begin{proof}
In order to show the coercivity of the bilinear form $a_{n+1}(\cdot,\cdot)$ on $\cV_\sd \times \cV_\sd$, we take $(\varphi,\sbr,\sbar,\bv) = (w^{n+1},\qbr^{n+1},\qbar^{n+1},\bu^{n+1})$.
Using the resulting cancellations together with Korn's and Poincar\'e's inequalities for the fluid velocity, we find that
\begin{equation*}
    \begin{aligned}
        &\quad a_{n+1}((w^{n+1},\qbr^{n+1},\qbar^{n+1},\bu^{n+1}),(w^{n+1},\qbr^{n+1},\qbar^{n+1},\bu^{n+1}))\\
        &= H \rho_p \| w^{n+1} \|_{\rL^2(\Gamma)}^2 + H^3 D [\Delta t]^2 \| \Delta_{\btau} w^{n+1} \|_{\rL^2(\Gamma)}^2 + H \gamma^p [\Delta t]^2 \| w^{n+1} \|_{\rL^2(\Gamma)}^2 + H \frac{c_0^p}{12} [\Delta t]^2 \| \qbr^{n+1} \|_{\rL^2(\Gamma)}^2\\
        &\quad +\frac{\kappa^p}{H} [\Delta t]^3 \| \qbr^{n+1} \|_{\rL^2(\Gamma)}^2 + \frac{12 \kappa^p}{H} [\Delta t]^3 \left\| \qbar^{n+1} + \frac{1}{2} \qbr^{n+1} \right\|_{\rL^2(\Gamma)}^2 + H c_0^p [\Delta t]^2 \| \qbar^{n+1} \|_{\rL^2(\Gamma)}^2\\
        &\quad + \rho_f [\Delta t]^2 \| \bu^{n+1} \|_{\rL^2(\Omega_f)}^2 + 2 \mu_f [\Delta t]^3 \| \bD(\bu^{n+1}) \|_{\rL^2(\Omega_f}^2 + \beta [\Delta t]^3 \| \bu^{n+1} \cdot \btau \|_{\rL^2(\Gamma)}^2\\
        &\gtrsim \min(1,[\Delta t]^2) \cdot \| w^{n+1} \|_{\rH^2(\Gamma)}^2 + \min([\Delta t]^2,[\Delta t]^3) \cdot \| \qbr^{n+1} \|_{\rL^2(\Gamma)}^2\\
        &\quad + [\Delta t]^2 \cdot \| \qbar^{n+1} \|_{\rL^2(\Gamma)}^2 + \min([\Delta t]^2,[\Delta t]^3) \cdot \| \bu^{n+1} \|_{\rH^1(\Omega_f)}^2\\
        &\gtrsim \min(1,[\Delta t]^2,[\Delta t]^3) \cdot \| (w^{n+1},\qbr^{n+1},\qbar^{n+1},\bu^{n+1}) \|_{\sd}^2.
    \end{aligned}
\end{equation*}
The continuity of $a_n(\cdot,\cdot)$ readily follows.
With regard to the continuity of $F_n(\cdot)$, we estimate
\begin{equation*}
    \begin{aligned}
        \| F_n \|_{\cV_\sd'}
        &\lesssim \| w^n \|_{\rL^2(\Gamma)} + [\Delta t] \| \Dot{w}^n \|_ {\rL^2(\Gamma)} + [\Delta t]^3 \cdot \| (q^+)^{n+1} \|_{\rL^2(\Gamma)}\\
        &\quad + [\Delta t]^2 \bigl(\| w^n \|_{\rH^2(\Gamma)} + \| \qbr^n \|_{\rL^2(\Gamma)}^2 + \| \qbar^n \|_{\rL^2(\Gamma)} + \| \bu^n \|_{\rL^2(\Omega_f)}\bigr),
    \end{aligned}
\end{equation*}
finishing the proof.
\end{proof}

The existence of a unique weak solution to the semidiscretized problem follows 
immediately from the Lax--Milgram lemma applied to the problem
\begin{equation}\label{ApproximateProblem}
    a_{n+1}((w^{n+1},\qbr^{n+1},\qbar^{n+1},\bu^{n+1}),(\varphi,\sbr,\sbar,\bv)) = F_n((\varphi,\sbr,\sbar,\bv)), \tforall (\varphi,\sbr,\sbar,\bv) \in \cV_{\sd}.
\end{equation}

\begin{lem}\label{prop:ex of weak sols semidiscr probl}
The rescaled, semidiscretized linear problem introduced in \autoref{def:weak form rescaled} has a unique weak solution.
\end{lem}

For each fixed $\Delta t > 0$, the approximate sequence, defined as the solution of \eqref{ApproximateProblem}, is given at the discrete time levels $t^n$, $n=0,\dots,N$. 
At each time level, these approximations depend only on the spatial variables $(x,y,z)$.

To obtain approximate functions defined on the entire time interval $(0,T)$, we extend these discrete approximations by piecewise constant extrapolation in time. 
This yields the approximate functions $(w^{[N]},\qbr^{[N]},\qbar^{[N]},\bu^{[N]})$ defined for every $(t,x,y,z)$:
\begin{equation}\label{eq:approx sols}
    (w^{[N]},\qbr^{[N]},\qbar^{[N]},\bu^{[N]}) = (w_N^n,\qbr_N^n,\qbar_N^n,\bu_N^n) \tfor t \in (t^{n-1},t^n] \tand n = 1,\dots,N.
\end{equation}
Additionally, for a piecewise constant function $f^{[N]}$, we define the time derivative $\Dot{f}^{[N]}$ via the backward difference quotient
\begin{equation*}
    \Dot{f}^{[N]}(t) = \frac{f_N^n - f_N^{n-1}}{[\Delta t]}, \tfor t \in (t^{n-1},t^n].
\end{equation*}
Our goal is to show that there exists a subsequence, still denoted by
$(w^{[N]},\qbr^{[N]},\qbar^{[N]},\bu^{[N]})$, that converges, as $N \to \infty$ (i.e., $\Delta t \to 0$), to a weak solution of the original coupled problem,  namely, to a function that satisfies the weak formulation given in \autoref{def:weak form cont-in-time}.

To this end, we observe that for each $N \in \N$ (or, equivalently, for each $\Delta t > 0$), the approximate functions $(w^{[N]},\qbr^{[N]},\qbar^{[N]},\bu^{[N]})$ satisfy the following weak formulation:

\begin{prop}\label{prop:weak form approx sols}
The approximate solutions $(w^{[N]},\qbr^{[N]},\qbar^{[N]},\bu^{[N]})$ defined by  \eqref{eq:approx sols}, satisfy the following weak formulation: 
\begin{equation*}
    \begin{aligned}
        &\quad H \rho_p ((\Ddot{w}^{[N]},\varphi))_\Gamma + H^3 D ((\Delta_{\btau} w^{[N]},\Delta_{\btau} \varphi))_\Gamma + H \gamma^p ((w^{[N]},\varphi))_\Gamma + H^2 \frac{\alpha^p}{12} ((\qbr^{[N]},\Delta_{\btau} \varphi))_\Gamma\\
        &\quad + ((\qbr^{[N]},\varphi))_\Gamma + H \frac{c_0^p}{12} ((\Dot{\qbr}^{[N]},\sbr))_\Gamma - H^2 \frac{\alpha^p}{12} ((\Delta_{\btau} \Dot{w}^{[N]},\sbr))_\Gamma + \frac{4 \kappa^p}{H} ((\qbr^{[N]},\sbr))_\Gamma +\frac{6 \kappa^p}{H}((\qbar^{[N]},\sbr))_\Gamma\\
        &\quad + ((\left.\bu^{[N]}\right|_\Gamma \cdot \bn,\sbr))_\Gamma - ((\Dot{w}^{[N]},\sbr))_\Gamma + H c_0^p ((\Dot{\qbar}^{[N]},\sbar))_\Gamma + \frac{12 \kappa^p}{H} ((\qbar^{[N]},\sbar))_\Gamma + \frac{6 \kappa^p}{H} ((\qbr^{[N]},\sbar))_\Gamma\\
        &\quad + \rho_f ((\Dot{\bu}^{[N]},\bv))_{\Omega_f} + 2 \mu_f ((\bD(\bu^{[N]}),\bD(\bv)))_{\Omega_f} + \beta ((\bu^{[N]} \cdot \btau,\bv \cdot \btau))_\Gamma - ((\qbr^{[N]},\bv \cdot \bn))_\Gamma\\
        &= \frac{6}{H} \kappa^p (((q^+)^{[N]},\sbr))_\Gamma + \frac{12}{H} \kappa^p (((q^+)^{[N]},\sbar))_\Gamma - (((q^+)^{[N]},\bv\cdot \bn))_\Gamma, \tforall (\varphi,\sbr,\sbar,\bv) \in \cV_\test.
    \end{aligned}
\end{equation*}
\end{prop}

This provides the foundation for passing to the limit as $N \to \infty$ along convergent subsequences and showing that the limiting functions satisfy the weak formulation of the original coupled problem.

In the next section, we establish that the sequence $(w^{[N]},\qbr^{[N]},\qbar^{[N]},\bu^{[N]})$, $N \in \mathbb{N}$, is uniformly bounded (independently of $N$ or $\Delta t$) in appropriate norms, which will allow us to extract convergent subsequences.

\subsection{Proof of existence of weak solutions}
\

We start by showing the following uniform a priori estimates for the sequence of approximate solutions $(w^{[N]},\qbr^{[N]},\qbar^{[N]},\bu^{[N]})$, $N \in \mathbb{N}$.

\begin{lem}\label{lem:a priori ests - 1}
There is a constant $C > 0$ independent of $N$, such that the following a priori estimate holds for approximate solutions $(w^{[N]},\qbr^{[N]},\qbar^{[N]},\bu^{[N]})$ defined in \eqref{eq:approx sols}:
\begin{equation*}
    \begin{aligned}
        &H \rho_p \| \Dot{w}^{[N]} \|_{\rL^\infty(0,T;\rL^2(\Gamma))}^2 + H \frac{c_0^p}{12} \| \qbr^{[N]} \|_{\rL^\infty(0,T;\rL^2(\Gamma))}^2 + H c_0^p \| \qbar^{[N]} \|_{\rL^\infty(0,T;\rL^2(\Gamma))}^2\\
        &+ \rho_f \| \bu^{[N]} \|_{\rL^\infty(0,T;\rL^2(\Omega_f))}^2 + H^3 D \| \Delta_{\btau} w^{[N]} \|_{\rL^\infty(0,T;\rL^2(\Gamma))}^2 + H \gamma^p \| w^{[N]} \|_{\rL^\infty(0,T;\rL^2(\Gamma))}^2 \le C,\\
        &\frac{\kappa^p}{H} \| \qbr^{[N]} \|_{\rL^2(0,T;\rL^2(\Gamma))}^2 + \frac{6 \kappa^p}{H} \left\| \qbar^{[N]} + \frac{1}{2} \qbr^{[N]} \right\|_{\rL^2(0,T;\rL^2(\Gamma))}^2\\
        &\enspace + \mu_f \| \bD(\bu^{[N]}) \|_{\rL^2(0,T;\rL^2(\Omega_f))}^2 + \beta \| \bu^{[N]} \cdot \btau \|_{\rL^2(0,T;\rL^2(\Gamma))}^2 \le C.
    \end{aligned}
\end{equation*}
\end{lem}

\begin{proof}
We test the semidiscretized weak formulation \eqref{eq:first weak form semidiscr form} by the test function $(\Dot{w}^{n+1},\qbr^{n+1},\qbar^{n+1},\bu^{n+1})$.
Making use of the resulting cancellations, invoking the difference quotient $\Dot{w}^{n+1}$, and using absorption arguments similar to those in \eqref{eq:absorption qbr & qbar} and \eqref{eq:absorption fluid vel} in the continuous case, we find that
\begin{equation*}
    \begin{aligned}
        &\quad H \rho_p \| \Dot{w}^{n+1} \|_{\rL^2(\Gamma)}^2 + H^3 D \| \Delta_{\btau} w^{n+1} \|_{\rL^2(\Gamma)}^2 + H \gamma^p \| w^{n+1} \|_{\rL^2(\Gamma)}^2 + H \frac{c_0^p}{12} \| \qbr^{n+1} \|_{\rL^2(\Gamma)}^2 + H c_0^p \| \qbar^{n+1} \|_{\rL^2(\Gamma)}^2\\
        &\quad + \rho_f \| \bu^{n+1} \|_{\rL^2(\Omega_f)}^2 + [\Delta t]\Bigl(\frac{\kappa^p}{H} \| \qbr^{n+1} \|_{\rL^2(\Gamma)}^2 + \frac{6 \kappa^p}{H} \left\| \qbar^{n+1} + \frac{1}{2}\qbr^{n+1}\right\|_{\rL^2(\Gamma)}^2\\
        &\qquad + \mu_f \| \bD(\bu^{n+1}) \|_{\rL^2(\Omega_f)}^2 + \beta \| \bu^{n+1} \cdot \btau \|_{\rL^2(\Gamma)}^2\Bigr)\\
        &\lesssim H \rho_p \| \Dot{w}^n \|_{\rL^2(\Gamma)}^2 + H^3 D \| \Delta_{\btau} w^n \|_{\rL^2(\Gamma)}^2 + H \gamma^p \| w^n \|_{\rL^2(\Gamma)}^2 + H \frac{c_0^p}{12} \| \qbr^n \|_{\rL^2(\Gamma)}^2 + H c_0^p \| \qbar^n \|_{\rL^2(\Gamma)}^2\\
        &\quad + \rho_f \| \bu^n \|_{\rL^2(\Omega_f)}^2 + [\Delta t] \Bigl(\frac{6}{H} \kappa^p + C_1\Bigr) \| (q^+)^{n+1} \|_{\rL^2(\Gamma)}^2,
    \end{aligned}
\end{equation*}
where $C_1$ is the constant appearing in \eqref{eq:absorption fluid vel}. 
Adding up all the terms, we get
\begin{equation*}
    \begin{aligned}
        & \quad H \rho_p \| \Dot{w}^{N+1} \|_{\rL^2(\Gamma)}^2 + H^3 D \| \Delta_{\btau} w^{N+1} \|_{\rL^2(\Gamma)}^2 + H \gamma^p \| w^{N+1} \|_{\rL^2(\Gamma)}^2 + H \frac{c_0^p}{12} \| \qbr^{N+1} \|_{\rL^2(\Gamma)}^2 + H c_0^p \| \qbar^{N+1} \|_{\rL^2(\Gamma)}^2\\
        &\quad + \rho_f \| \bu^{N+1} \|_{\rL^2(\Omega_f)}^2 + [\Delta t] \sum_{n=0}^N \Bigl[\frac{\kappa^p}{H} \| \qbr^{n+1} \|_{\rL^2(\Gamma)}^2 + \frac{6 \kappa^p}{H} \left\| \qbar^{n+1} + \frac{1}{2}\qbr^{n+1}\right\|_{\rL^2(\Gamma)}^2\\
        &\qquad + \mu_f \| \bD(\bu^{n+1}) \|_{\rL^2(\Omega_f)}^2 + \beta \| \bu^{n+1} \cdot \btau \|_{\rL^2(\Gamma)}^2\Bigr]\\
        &\lesssim H \rho_p \| \Dot{w}^0 \|_{\rL^2(\Gamma)}^2 + H^3 D \| \Delta_{\btau} w^0 \|_{\rL^2(\Gamma)}^2 + H \gamma^p \| w^0 \|_{\rL^2(\Gamma)}^2 + H \frac{c_0^p}{12} \| \qbr^0 \|_{\rL^2(\Gamma)}^2 + H c_0^p \| \qbar^0 \|_{\rL^2(\Gamma)}^2\\
        &\quad + \rho_f \| \bu^0 \|_{\rL^2(\Omega_f)}^2 + \Bigl(\frac{6}{H} \kappa^p + C_1\Bigr)[\Delta t] \sum_{n=0}^N \| (q^+)^{n+1} \|_{\rL^2(\Gamma)}^2.
    \end{aligned}
\end{equation*}
The a priori estimates stated in the lemma then follow immediately.
\end{proof}

Before we can pass to the limit $N \to \infty$ we need a dual estimate of $\Ddot{w}^{[N]}$ in order to deal with the term~$\dot{w}^{[N]}$. 
Additionally, we need a dual estimates of the fluid content $(H \frac{c_0^p}{12} \Dot{\qbr}^{[N]} - H^2 \frac{\alpha^p}{12} \Delta_{\btau} \Dot{w}^{[N]})$ appearing in the second equations for the jump in the pressure. 
Indeed, the following estimates hold:

\begin{lem}\label{lem:a priori ests - 2}
There exists a constant $C > 0$ independent of $N$ such that
\begin{equation*}
    \left\| H \rho_p \Ddot{w}^{[N]} \right\|_{\rH_0^2(\Gamma)'} \le C \tand \left\| H \frac{c_0^p}{12} \Dot{\qbr}^{[N]} - H^2 \frac{\alpha^p}{12} \Delta_{\btau} \Dot{w}^{[N]} \right\|_{\rL^2(\Gamma)'} \le C.
\end{equation*}
\end{lem}

\begin{proof}
In order to obtain the first estimate, we insert the test function $(\varphi,0,0,0)$ into \eqref{eq:first weak form semidiscr form} and find that
\begin{equation*}
    \begin{aligned}
        &\quad H \rho_p (\Dot{w}^{n+1} - \Dot{w}^n,\varphi)_\Gamma + H^3 D [\Delta t] (\Delta_{\btau} w^{n+1},\Delta_{\btau} \varphi)_\Gamma\\
        &\quad + H \gamma^p [\Delta t] (w^{n+1},\varphi)_\Gamma + H^2 \frac{\alpha^p}{12} [\Delta t] (\qbr^{n+1},\Delta_{\btau} \varphi)_\Gamma + [\Delta t] (\qbr^{n+1},\varphi)_\Gamma = 0.
    \end{aligned}
\end{equation*}
Thus, invoking the definition of $\Ddot{w}^{n+1}$, and dividing by $[\Delta t]$, we obtain
\begin{equation*}
    H \rho_p (\Ddot{w}^{n+1},\varphi)_\Gamma = -H^3 D (\Delta_{\btau} w^{n+1},\Delta_{\btau} \varphi)_\Gamma - H \gamma^p (w^{n+1},\varphi)_\Gamma - H^2 \frac{\alpha^p}{12} (\qbr^{n+1},\Delta_{\btau} \varphi)_\Gamma - (\qbr^{n+1},\varphi)_\Gamma.
\end{equation*}
Integrating with respect to time from $t^n = n [\Delta t]$ to $t^{n+1} = (n+1) [\Delta t]$, summing over $n = 0,\dots,N-1$, and making use of $\varphi \in \rH_0^2(\Gamma)$, together with the a priori estimates from \autoref{lem:a priori ests - 1}, we conclude that there exists a constant $C>0$ independent of $N$ such that 
$$\left\| H \rho_p \Ddot{w}^{[N]} \right\|_{\rH_0^2(\Gamma)'} \le C.$$

For the second estimate, we use the test function $(0,\sbr,0,0)$ in \eqref{eq:first weak form semidiscr form} to obtain
\begin{equation*}
    \begin{aligned}
        &\quad H \frac{c_0^p}{12}(\qbr^{n+1} - \qbr^n,\sbr)_\Gamma - H^2 \frac{\alpha^p}{12} [\Delta t] (\Delta_{\btau} \Dot{w}^{n+1},\sbr)_\Gamma + \frac{4 \kappa^p}{H} [\Delta t] (\qbr^{n+1},\sbr)_\Gamma\\
        &\quad + \frac{6 \kappa^p}{H} [\Delta t](\qbar^{n+1},\sbr)_\Gamma + [\Delta t] (\left.\bu^{n+1}\right|_\Gamma \cdot \bn,\sbr)_\Gamma - [\Delta t](\Dot{w}^{n+1},\sbr)_\Gamma = 6 \kappa^p [\Delta t]((q^+)^{n+1},\sbr)_\Gamma.
    \end{aligned}
\end{equation*}
Inserting the difference quotient for $\Dot{\qbr}^{n+1}$, and dividing the resulting equation by $[\Delta t]$, we infer that
\begin{equation*}
    \begin{aligned}
        &\quad \left(H \frac{c_0^p}{12}\Dot{\qbr}^{n+1} - H^2 \frac{\alpha^p}{12} \Delta_{\btau} \Dot{w}^{n+1},\sbr\right)\\
        &= -\frac{4 \kappa^p}{H} (\qbr^{n+1},\sbr)_\Gamma - \frac{6 \kappa^p}{H}(\qbar^{n+1},\sbr)_\Gamma - (\left.\bu^{n+1}\right|_\Gamma \cdot \bn,\sbr)_\Gamma +(\Dot{w}^{n+1},\sbr)_\Gamma + 6 \kappa^p ((q^+)^{n+1},\sbr)_\Gamma.
    \end{aligned}
\end{equation*}
Proceeding in the same way as above, and noting that the uniform estimates especially imply a bound on $\left.\bu \right|_\Gamma$ in $\rL^2(\Gamma)$ thanks to mapping properties of the trace, we complete the proof of the second estimate.
\end{proof}

These uniform estimates allow us to deduce the existence of subsequences of approximate solutions that converge weakly and weakly* in appropriate spaces. 
In particular, we have the following lemma.

\begin{lem}\label{convergence}
There exist convergent subsequences, still denoted by $((w^{[N]},\qbr^{[N]},\qbar^{[N]},\bu^{[N]}))_{N \in \N}$, such that the following convergence results hold:
\begin{equation*}
    \begin{aligned}
        (w^{[N]},\qbr^{[N]},\qbar^{[N]},\bu^{[N]}) 
        &\rightharpoonup (w,\qbr,\qbar,\bu) &&\text{ weakly in } \cV_\sol,\\
        H \rho_p \Ddot{w}^{[N]} 
        &\overset{*}{\rightharpoonup} H \rho_p \Ddot{w} &&\text{weakly* in } \rL^\infty(0,T;\rH_0^2(\Gamma)) \tand\\
        H \frac{c_0^p}{12} \Dot{\qbr}^{[N]} - H^2 \frac{\alpha^p}{12} \Delta_{\btau} \Dot{w}^{[N]}
        &\overset{*}{\rightharpoonup} H \frac{c_0^p}{12} \Dot{\qbr} - H^2 \frac{\alpha^p}{12} \Delta_{\btau} \Dot{w} &&\text{weakly* in } \rL^\infty(0,T;\rL^2(\Gamma)).
    \end{aligned}
\end{equation*}
\end{lem}
We remark that the weak* convergence of the fluid component follows after using Korn's and Poincar\'e's inequalities.

It is now straightforward to show that the limits of convergent approximate subsequences satisfy the weak formulation of the original problem by taking the limit as $N\to\infty$ in the approximate weak formulation stated in \autoref{prop:weak form approx sols}.
More precisely, we have the following main result on the existence of weak solutions to our coupled FPSI problem.

\begin{thm}[{\bf{Existence of weak solutions}}]\label{thm:ex of weak solutions}
Let $T > 0$, and assume that the initial data satisfy $w_0 \in \rH_0^2(\Gamma)$, $v_0 \in \rL^2(\Gamma)$, $\qbr_0 \in \rL^2(\Gamma)$, $\qbar_0 \in \rL^2(\Gamma)$, and $\bu_0 \in \rL^2(\Omega_f)^3$.
Then there exists a weak solution to \eqref{eq:fully averaged plate model}--\eqref{bottom} on $(0,T)$ in the sense of \autoref{def:weak form cont-in-time}.
\end{thm}

\begin{rem}
We conclude this section with a brief remark on the uniqueness of weak solutions. 
Even in the linear setting, uniqueness is not immediate.
The main difficulty is that the energy estimates derived above for solutions constructed by our method do not necessarily extend to an arbitrary weak solution.

However, as in \cite[Sec.~5.3.2]{BCMW:21}, uniqueness can be established under an additional assumption. 
More precisely, if a weak solution satisfies the following energy identity:
\[
    H \rho_p (\partial_{t}^2 w,\partial_t w)_\Gamma + H^3 D (\Delta^2_{\boldsymbol{\tau}} w, \partial_t w)_\Gamma = \frac{1}{2}\frac{\rd}{\rd t} \left(H \rho_p \| \partial_t w \|^2_{L^2(\Gamma)} + H^3 D \| \Delta_{\boldsymbol{\tau}} w \|^2_{L^2(\Gamma)}\right),
\]
then uniqueness of weak solutions to the linear problem follows.
\end{rem}

In the next section, we establish the existence and uniqueness of strong solutions to \eqref{eq:fully averaged plate model}--\eqref{bottom}, after introducing a regularization consisting of a viscoelastic term in the plate elastodynamics equation and an elliptic regularization in the pressure jump equation.

\section{Strong solutions to a regularized problem}\label{sec:strong sols}

In this section, we prove the existence of a unique strong solution to a regularized version of the coupled FPSI problem \eqref{eq:fully averaged plate model}--\eqref{bottom}.
The regularization is achieved by incorporating viscoelastic terms into the poroelastic Kirchhoff plate equation \eqref{eq:fully averaged plate model}$_1$ and applying elliptic regularization to the pressure jump equation \eqref{eq:fully averaged plate model}$_2$.

The origin of these terms is twofold. 
First, the biharmonic term \( \varepsilon_1 \Delta^2 \partial_t w \) represents intrinsic damping within the solid matrix -- while classical poroelasticity models assume a purely elastic skeleton, real materials exhibit internal friction. 
By incorporating a Kelvin--Voigt damping mechanism at the three-dimensional level, this biharmonic contribution arises naturally in the asymptotic dimension reduction to a plate model, thereby capturing energy dissipation within the solid phase.

Second, the elliptic regularization in the fluid equation $\eps_2\Delta_{\btau} \qbr$ reflects the anisotropic nature of fluid filtration: tangential flow is typically of much smaller magnitude than transversal flow. 
Although this effect vanishes in the formal asymptotic limit leading to the reduced model, we retain the regularizing term at the level of the analysis. 
Its inclusion ensures mathematical stability of the system and provides the additional regularity needed to establish well-posedness and derive robust analytical results.

To simplify notation in this section we will be assuming that all the coefficients in the model are equal to $1$:
\begin{equation*}
    H = \rho_p = D = \alpha^p = c_0^p = \kappa^p \equiv 1,
\end{equation*}
as well as the domain parameters:
\begin{equation*}
    L = R = 1,
\end{equation*}
while we keep $\gamma^p$ in front of the dissipative term $w$ unchanged, leaving open the possibility of taking this coefficient to be equal to zero in future works. 

Written as a first-order system in time, and using $\eps_1>0$ as well as $\eps_2>0$ to denote the coefficients in front of the regularizing terms, the regularized FPSI problem reads:
\begin{equation}\label{eq:FPSI Cauchy problem}
    \left\{
    \begin{aligned}
        \del_t w
        &= v, &&\tin (0,T) \times \Gamma,\\
        \del_t v + \Delta_{\btau}^2 w + \gamma^p w + \eps_1\Delta_{\btau}^2 v + \frac{1}{12} \Delta_{\btau} \qbr + \qbr
        &= 0, &&\tin (0,T) \times \Gamma,\\
        \del_t \qbr - \Delta_{\btau} v -12 v - \eps_2\Delta_{\btau} \qbr + 48 \qbr + 72 \qbar + \left.12\bu \right|_{\Gamma} \cdot \bn
        &= 72 q^+, &&\tin (0,T) \times \Gamma,\\
        \del_t \qbar  + 6 \qbr + 12 \qbar 
        &= 12 q^+, &&\tin (0,T) \times \Gamma,\\
        \del_t \bu - \nabla \cdot \bsigma_f(\bu,\pi)
        &= 0, &&\tin (0,T) \times \Omega_f,\\
        \nabla \cdot \bu 
        &= 0, &&\tin (0,T) \times \Omega_f.
    \end{aligned}
    \right.
\end{equation}
The problem is supplemented with the coupling conditions
\eqref{ContVel}--\eqref{Fp}, the initial conditions
\eqref{init}, and the following boundary conditions:
\begin{itemize}
    \item {Periodic boundary conditions}:
    \begin{equation}\label{NewBC}
        \begin{aligned}
            (w,v,\qbr,\qbar)(x,0)
            &= (w,v,\qbr,\qbar)(x,1), &&\tfor x \in (0,1),\\
            (w,v,\qbr,\qbar)(0,y)
            &= (w,v,\qbr,\qbar)(1,y), &&\tfor y \in (0,1),\\
            (\bu,\pi)(x,0,z)
            &= (\bu,\pi)(x,1,z), &&\tfor x \in (0,1), \ z \in (-1,0),\\
            (\bu,\pi)(0,y,z)
            &= (\bu,\pi)(1,y,z), &&\tfor y \in (0,1), \ z \in (-1,0).
        \end{aligned}
    \end{equation}
    \item {Prescribed fluid pressure on the top surface of the plate:}
    \begin{equation*}
        q|_{\Gamma^+} = q^+.
    \end{equation*}
    \item {No-slip condition at the bottom of the fluid domain:}
    \begin{equation*}
        \bu|_{\Gamma_f^b} = 0.
    \end{equation*}
\end{itemize}
We note that the two coupling conditions \eqref{BJS} and \eqref{Dynamic2} (along with \eqref{Fp}) can be written together as
\begin{equation}\label{NewCoupling}
    \bsigma_f(\bu,\pi) \bn = (\qbr - q^+) \bn - \beta \sum_{i=1}^2 (\bu \cdot \btau_i) \btau_i.
\end{equation}
This is the form that will be used throughout the rest of this section.

For the reader's convenience, we recall the initial conditions specified in \eqref{init}:
\begin{equation}\label{init2}
    w|_{t = 0} = w_0, \ \partial_t w|_{t = 0} = v_0, \ [q]|_{t = 0} = [q]_0, \ \bar{q}|_{t = 0} = \bar{q}_0, \ \bu|_{t = 0} = \bu_0.
\end{equation}

We prove here the global-in-time existence of a unique strong solution by employing an operator-theoretic framework applied to the first-order evolution system
\begin{equation}\label{A system}
    \begin{aligned}
        u'(t) - Au(t)
        &= f(t), \tfor t \in (0,T),\\
        u(0)
        &= u_0,
    \end{aligned}
\end{equation}
and by proving sectoriality of the operator $A$. 
On Hilbert spaces $\rX$, this property is sufficient to ensure maximal $\rL^p$-regularity for the Cauchy problem \eqref{A system}.
Maximal $\rL^p$-regularity will guarantee existence of a unique solution 
\begin{equation*}
    u \in \rW^{1,p}(0,T;\rX) \cap \rL^p(0,T;\rD(A)),
\end{equation*}
for $f \in \rL^p(0,T;\rX)$, $1 < p < \infty$, where $\rD(A)\subset \rX$ denotes the domain of the operator $A$. 

One reason for adopting the sectoriality and maximal $\rL^p$-regularity framework to establish the existence and uniqueness of solutions to the regularized problem is that it provides a natural setting for extensions to nonlinear problems and the use of fixed-point arguments. 
A primary direction for future work is to extend the methods developed here to FPSI problems with nonlinear coupling.

Notice that we imposed \emph{periodic boundary conditions} on the plate variables $w$, $v$, $\qbr$, and $\qbar$, instead of the \emph{clamped boundary conditions} for the plate. 
This choice was made to avoid technical difficulties in the analysis of strong solutions. 
The use of clamped or Neumann-type boundary conditions on Lipschitz domains such as $\Gamma = (0,1)^2$ introduces significant technical difficulties; see, for example, \cite{Tol:18} for a discussion of elliptic systems on Lipschitz domains.

Therefore, in this section, we will be working with the following three additional assumptions:
\begin{enumerate}
    \item {\em Viscoelasticity} of the poroelastic plate, described by the term with the coefficient $\varepsilon_1$.
    \item {\em Elliptic regularization} of the jump in the pressure evolution, described by the term with the coefficient $\varepsilon_2$. 
    \item {\em Periodic boundary conditions} for $w$, $v$, $\qbr$, and $\qbar$.
\end{enumerate}
The viscoelastic damping term is introduced to render the plate elastodynamics problem fully parabolic. 
The study of the corresponding inviscid limit $\varepsilon_1 \to 0$ remains an interesting open problem. 
In a related setting of fluid--beam interaction without viscous damping, local-in-time existence of strong solutions has been established recently in \cite{GHL:19}.

With regard to the pressure jump $\qbr$, a viscous regularization is incorporated to avoid the difficulties associated with an off-diagonal dominant operator matrix. 
The analysis of such operator matrices is generally delicate; see, for example, the monograph \cite{Tre:08} for further details.

The aim of this section is to show the existence of a unique global-in-time strong solution to the Cauchy problem \eqref{eq:FPSI Cauchy problem}--\eqref{init2}.

\medskip
\noindent{{\bfseries Outline of the proof strategy.}}
Our goal is to reformulate the regularized coupled FPSI problem \eqref{eq:FPSI Cauchy problem}--\eqref{init2} as an abstract Cauchy problem and to establish its global strong well-posedness by proving that the associated spatial operator is sectorial and generates an analytic semigroup. 
This approach allows us to exploit maximal $\rL^p$-regularity and semigroup theory in a systematic way.

A central difficulty in this analysis stems from the structure of the coupled system. 
The problem is not parabolic in a straightforward sense: 
the elastodynamics of the plate, the pressure variables, and the fluid velocity are coupled through interface conditions, and several of these coupling terms appear as lower-order boundary contributions. 
As a result, the associated operator matrix is not diagonally dominant in an obvious way, and it does not have a diagonal domain, i.e., the domain cannot be written as the product of the domains of the jump in the pressure and the fluid velocity.
Thus, classical perturbation results cannot be applied directly.

To overcome this, the proof is organized into several steps, each addressing a specific structural difficulty.

{\bf Step 1.} We first assume that the boundary condition for the pressure prescribed on the top surface of the plate, $q^+$, is homogeneous, i.e., $q^+=0$. The inhomogeneous data $q^+$ will be handled at the end of the proof by a lifting argument.

{\bf Step 2.} For the problem with homogeneous boundary condition $q^+=0$, we begin by isolating the subproblem involving the Stokes fluid flow coupled to the pressure jump equation, which captures the most delicate part of the coupling. This subsystem is studied on the ground space
\[
    \rL^2(\Gamma) \times \rL_\sigma^2(\Omega_f),
\]
where $\rL_\sigma^2(\Omega_f)$ denotes the space of solenoidal vector fields in $\rL^2(\Omega_f)$.
The space $\rL_\sigma^2(\Omega_f)$ will be introduced in~\eqref{Projection}.
We first consider an interface condition that only captures the highest-order boundary term on the left-hand side of \eqref{NewCoupling}, while the right-hand side of \eqref{NewCoupling} is set equal to zero in a first step.
In this simplified setting, we prove sectoriality of the operator, which provides a baseline result for the fluid component of the fluid--structure interaction problem.

{\bf Step 3.} In the next step, we return to the full coupling at the interface for the fluid--pressure jump subsystem, now imposing the coupling conditions \eqref{NewCoupling} with $q^+ = 0$. 
At this stage, the main difficulty becomes apparent: the coupling enters through boundary terms involving traces of the fluid velocity and stresses. 
This gives rise to a coupled problem where the boundary terms are of mixed orders, and the domain of the associated operator is not diagonal.
Consequently, the coupling cannot be treated as a standard bounded perturbation of the previously analyzed operator.

To overcome this issue, we exploit the fractional scale.
This allows us to represent the operator to the full problem with $q^+ = 0$ as a perturbation of the extrapolated operator with homogeneous boundary conditions.
At this stage, the fact that the trace of the fluid velocity is a bounded operator from a suitable interpolation space into $\rL^2(\Gamma)$ also comes into play.

For the aforementioned representation of the operator, we also invoke a lifting operator associated with the stationary Stokes problem with inhomogeneous Neumann boundary conditions at the interface $\Gamma$. 
This operator maps boundary data into divergence-free velocity fields in the fluid domain and allows us to represent the interface coupling as a volume term in the extended (extrapolated) space. 
In this way, the boundary contributions are ``lifted'' into the bulk, where they can be handled using the operator framework.

Combining these two ingredients, we can reinterpret the coupling conditions as lower-order perturbations of the fluid operator, acting between appropriate fractional spaces. 
This reformulation places the problem within the scope of the abstract perturbation results recalled in \autoref{sec:appendix}, and allows us to retain sectoriality of the underlying operator.

{\bf Step 4.} Having established the sectoriality of the fluid--pressure jump subsystem, we next extend the analysis to incorporate the averaged pressure variable $\qbar$. 
At the level of the equations, $\qbar$ is not independent, but is coupled to the pressure jump $\qbr$ through a first-order (in time) relation without spatial derivatives. 
In particular, this coupling is of lower order compared to the principal operators governing the fluid and $\qbr$, but the resulting operator matrix does not directly have block structure.

To handle this, we introduce a suitable similarity transformation that combines $\qbar$ and $\qbr$ in such a way that the dominant part of the operator becomes block-diagonal, while the remaining coupling terms appear as bounded (or relatively bounded) perturbations. 
This transformation is invertible and acts continuously on the underlying state space, so it preserves key properties such as closedness, spectrum, and sectoriality of the operator.

In the transformed variables, the evolution equation for $\qbar$ takes the form of a lower-order ODE-type component, which does not contribute to the principal part of the operator. 
As a result, the extended system can be viewed as a perturbation of the already analyzed fluid--pressure jump subsystem. 
The coupling between $\qbar$ and the remaining variables now enters only through bounded operators, and therefore fits naturally into the perturbation framework mentioned earlier.

Applying the abstract perturbation results, we conclude that the inclusion of $\qbar$ does not affect sectoriality (up to a shift), and hence the extended operator governing the fluid, $\qbr$, and $\qbar$ remains sectorial with the same spectral angle.

{\bf Step 5.} 
In a further step, we incorporate the viscoelastic plate dynamics and combine it with the previously analyzed fluid subsystem into a unified block operator matrix formulation. 
At this point, the system consists of two components of different analytical character: 
a parabolic (Stokes-type) fluid subsystem coupled with pressure variables, and a fourth-order (in space) viscoelastic plate equation with additional damping terms. 
These components interact through interface conditions involving traces of the fluid velocity and stresses, as well as the plate displacement and its time derivative.

To capture this structure, we introduce a product space of the form
\[
    \rX = \rX_{\mathrm{plate}} \times \rX_{\mathrm{fluid}},
\]
where $\rX_{\mathrm{plate}}$ encodes the displacement and velocity variables of the plate, and $\rX_{\mathrm{fluid}}$ corresponds to the fluid--pressure subsystem analyzed in the previous steps. 
The full system can then be written as an abstract evolution equation governed by an operator matrix of the form
\[
    \cA = 
    \begin{pmatrix}
        A_{\mathrm{plate}} & B \\
        C & A_{\mathrm{fluid}}
    \end{pmatrix},
\]
where:
\begin{itemize}
    \item $A_{\mathrm{plate}}$ represents the viscoelastic plate operator, including the biharmonic term and damping contributions, and is sectorial on $\rX_{\mathrm{plate}}$,
    \item $A_{\mathrm{fluid}}$ is the sectorial operator associated with the fluid--pressure subsystem,
    \item $B$ and $C$ encode the coupling of the elastodynamics of the plate and the jump in the pressure.
\end{itemize}

The key point is that the diagonal operators $A_{\mathrm{plate}}$ and $A_{\mathrm{fluid}}$ capture the principal dynamics and are both sectorial, while the off-diagonal operators $B$ and $C$ are of lower order. 
More precisely, due to the ellipticity of the plate operator, the coupling operators map between fractional domain spaces and satisfy estimates of the form required in \autoref{def:diag dom block op matrix}. 
In particular, $B$ and $C$ are relatively bounded with respect to the diagonal operators, with sufficiently small relative bounds.

As a consequence, the full operator $\cA$ is diagonally dominant in the sense of \autoref{def:diag dom block op matrix}. 
We can therefore invoke the abstract results for such operator matrices to conclude that $\cA$ is sectorial on~$\rX$, possibly after adding a sufficiently large shift. 
This step is crucial, as it allows us to transfer the sectoriality of the uncoupled subsystems to the fully coupled fluid--structure interaction problem.

{\bf Step 6.} 
Finally, we remove the shift introduced in the previous step and establish exponential stability by analyzing the spectrum of the full coupled operator. 
More precisely, after having shown that the shifted operator $\cA + \omega$ is sectorial, we study the location of the spectrum of $-\cA$ and prove that it lies strictly in the open left half-plane, i.e.,
\[
    \sigma(-\cA) \subset \{ \lambda \in \C : \Rep \lambda < 0 \}.
\]
This step relies on a careful analysis of both the point spectrum and the essential spectrum. 
The dissipative structure of the system, encoded in the energy identities, ensures that eigenvalues cannot lie on the imaginary axis, while a successive reduction argument inspired by \cite{BMR:26} in the context of a multilayered FSI problem allows us to control the essential spectrum and exclude its intersection with the imaginary axis.

As a consequence, the spectral bound of $-\cA$ is strictly negative. 
Combining this with the sectoriality of $\cA$, we conclude that $-\cA$ generates an analytic semigroup that is not only bounded but exponentially stable, i.e., there exist constants $M \ge 1$ and $\mu > 0$ such that
\[
    \| e^{-t\cA} \|_{\cL(\rX)} \le M e^{-\mu t}, \qquad t \ge 0.
\]

This exponential stability allows us to remove the shift $\omega$ introduced earlier, since the semigroup generated by $-\cA$ already exhibits decay. 
In particular, the original (unshifted) operator $\cA$ inherits the desired generation and stability properties.

With these results at hand, the regularized FPSI system can be interpreted as an abstract parabolic evolution equation on the state space $\rX$. 
The sectoriality of $\cA$, together with the Hilbert space setting, implies maximal $\rL^p$-regularity. 
We can therefore apply the general theory of parabolic evolution equations to obtain existence, uniqueness, and continuous dependence of strong solutions with optimal regularity. 
This completes the proof of well-posedness for the regularized coupled problem with homogeneous boundary data $q^+ = 0$.

{\bf Step 7.}
At the end of the section, we return to the original problem with inhomogeneous boundary data $q^+$. 
The presence of this term prevents a direct application of the abstract framework developed above, since the coupling conditions are no longer homogeneous and therefore do not fit into the operator-theoretic formulation of the Cauchy problem.

To overcome this difficulty, we employ a lifting argument. 
More precisely, we decompose the solution as
\[
    (w,v,\qbr,\qbar,\bu) = (w,v,\qbr,\qbar,\widetilde{\bu}) + (0,0,0,0,\bu_*),
\]
where the lifting function $\bu_*$ is chosen to solve a Stokes-type problem with boundary data corresponding to $q^+$. 
In particular, $\bu_*$ satisfies the same divergence-free and boundary conditions as the fluid velocity, but incorporates the inhomogeneous interface condition through a prescribed normal stress.

By construction, $\bu_*$ absorbs the boundary contribution associated with $q^+$, and the remaining unknown $(w,v,\qbr,\qbar,\widetilde{\bu})$ satisfies a system with homogeneous coupling conditions. 
However, this comes at the expense of introducing additional source terms in the equations for $\qbr$ and $\qbar$, which now depend on $q^+$ and the trace of $\bu_*$. 
These terms act as forcing terms in the abstract evolution equation.

The key point is that the lifting problem for $\bu_*$ is well-posed in the appropriate maximal regularity class, provided that $q^+$ satisfies suitable regularity assumptions. 
This follows from the optimal regularity theory for the Stokes problem with inhomogeneous boundary data, combined with a novel perturbation argument for the optimal data property. 
As a result, the lifting $\bu_*$ inherits the same regularity (and, if applicable, exponential decay) as the boundary data $q^+$.

Once this decomposition is established, the remaining system for $(w,v,\qbr,\qbar,\widetilde{\bu})$ can be treated using the homogeneous theory developed in the previous steps. 
In particular, the operator governing the homogeneous part is sectorial and generates an analytic semigroup, and the additional forcing terms satisfy the required regularity assumptions.

Combining the well-posedness of the homogeneous problem with the regularity properties of the lifting, we obtain the existence, uniqueness, and continuous dependence of solutions to the original problem with inhomogeneous boundary data. 
Moreover, the solution decays exponentially when assuming that the data decays exponentially.
In fact, the solution decays exponentially at the same rate as the data.
This completes the proof.

\medskip
\noindent{{\bfseries Notation.}}
Before specifying the functional framework, we introduce notation that will be used throughout this section. 
In particular, we make precise the notion of periodicity in the horizontal variables.

Let $m \in \N_0$. 
A smooth function $f \colon \oOmega_f \to \R$ is said to be \emph{periodic of order $m$} (with respect to the variables $x$ and $y$) if
\begin{equation*}
    \frac{\del^\alpha f}{\del x^\alpha}(0,y,z) = \frac{\del^\alpha f}{\del x^\alpha}(1,y,z) \tand
    \frac{\del^\alpha f}{\del y^\alpha}(x,0,z) = \frac{\del^\alpha f}{\del y^\alpha}(x,1,z),
\end{equation*}
for all multi-indices $\alpha$ with $0 \le |\alpha| \le m$. 
The notion of periodicity of order $m$ on $\del \Gamma$ is defined analogously for functions $f \colon \oGamma \to \R$.

We then introduce the spaces of smooth periodic functions
\begin{equation*}
    \begin{aligned}
        \rC_\per^\infty(\oOmega_f) 
        &\coloneqq \left\{ f \in \rC^\infty(\oOmega_f) : \text{$f$ is periodic of arbitrary order on $\Gamma_f^l$} \right\}, \\
        \rC_\per^\infty(\oGamma) 
        &\coloneqq \left\{ f \in \rC^\infty(\oGamma) : \text{$f$ is periodic of arbitrary order on $\del\Gamma$} \right\}.
    \end{aligned}
\end{equation*}

For $s \ge 0$, we define the periodic Sobolev spaces as the closures
\begin{equation}\label{eq:Sobolev spaces with per bc}
    \rH_\per^s(\Omega_f) \coloneqq \overline{\rC_\per^\infty(\oOmega_f)}^{\| \cdot \|_{\rH^s(\Omega_f)}}, \qquad \rH_\per^s(\Gamma) \coloneqq \overline{\rC_\per^\infty(\oGamma)}^{\| \cdot \|_{\rH^s(\Gamma)}}.
\end{equation}
In particular, $\rH_\per^0(\Omega_f)$ and $\rH_\per^0(\Gamma)$ are identified with $\rL^2(\Omega_f)$ and $\rL^2(\Gamma)$, respectively.

For $s \ge 0$ and $p \in (1,\infty)$, the corresponding periodic Besov spaces $\rB_{2p,\per}^s(\Omega_f)$ and $\rB_{2p,\per}^s(\Gamma)$ are defined analogously.

We now present the details of the analysis where we first assume the homogeneous boundary condition $q^+ = 0$.
We start with the fluid--pressure jump subproblem.

\subsection{The coupled  fluid -- pressure jump subproblem with $q^+ = 0$}
\ 

The aim of this subsection is to analyze the coupled system consisting of the Stokes fluid and the pore pressure variable $\qbr$, which will serve as a fundamental building block for the analysis of the full FSI problem.

Since this analysis plays a central role in the overall argument, we briefly outline the strategy. 
The construction proceeds in several steps, each addressing a different aspect of the coupled operator.

We start by introducing the so-called \emph{maximal} operator matrix $\bA_{f,\max}^{\rI}$, which captures the differential structure of the problem in the interior of the fluid domain, together with the boundary conditions on the non-interface part of the boundary. 
The superscript ``$\rI$'' indicates the first operator matrix capturing the fluid subproblem in the two-part formulation, where the second part will include the average pressure contribution. 
At this stage, the interface conditions are not yet fully imposed. 
A key ingredient here is the maximal Stokes operator $A_{\max}$, which encodes the fluid dynamics without incorporating the coupling at the interface.

We then introduce the operators that will incorporate the interface coupling by separating the highest-order and lower-order contributions in the boundary conditions. 
To this end, we introduce a boundary operator $L$ that encodes the principal (highest-order) part of the interface condition. 
More precisely, $L$ acts on the trace of the fluid variables and enforces the leading-order relation between the normal stress $\bsigma_f(\bu,\pi)\bn$ and the pressure jump $\qbr$, see \eqref{NewCoupling}.
This operator captures the dominant coupling mechanism at the interface and determines the principal part of the boundary dynamics.

The remaining terms in the interface condition \eqref{NewCoupling}, which involve lower-order contributions such as tangential components, are collected into a perturbation operator $\Phi$. 
By construction, $\Phi$ is of lower order relative to $L$ in the sense that it acts between weaker (fractional) spaces and satisfies appropriate boundedness or relative boundedness estimates. 
This decomposition is crucial, as it allows us to treat the full coupling as a perturbation of a problem with simpler (and analytically more tractable) boundary conditions.

We then turn to the analysis of sectoriality of the associated operator by first going back to the reference problem obtained by imposing homogeneous interface condition. 
More precisely, we consider the Stokes system subject to homogeneous Neumann boundary conditions at the interface $\Gamma$, namely,
\[
    \bsigma_f(\bu,\pi)\bn = 0 \ton \Gamma.
\]
This corresponds to removing the leading-order coupling at the interface and yields a problem that is purely fluid-dynamic, with no exchange of momentum across the interface.

The associated operator, denoted by $\bA_{f,0}^{\rI}$, plays the role of a reference operator. 
It captures the intrinsic parabolic structure of the Stokes system under homogeneous boundary conditions and is well understood analytically. 
In particular, we prove in \autoref{lem:sect of op with hom bc} that $\bA_{f,0}^{\rI}$ is sectorial on the underlying Hilbert space. 
This result provides the cornerstone for the subsequent analysis: once sectoriality is established for the reference operator, the full coupled operator can be treated as a perturbation via the operators $L$ and $\Phi$ introduced above.

Indeed, we next return to the Stokes problem with prescribed interface data, corresponding to the coupling condition \eqref{NewCoupling} with $q^+ = 0$, which can be written in general form as
\[
    \bsigma_f(\bu,\pi)\bn = \psi \ton \Gamma,
\]
for a given $\psi \in \rH_\per^{\frac{1}{2}}(\Gamma)^3$. 
The ultimate goal is to show that the operator corresponding to the problem with inhomogeneous coupling condition, denoted by $\bA_f^I$,  is sectorial.

This is done by constructing a \emph{lifting operator} that associates to each admissible boundary datum $\psi$ a divergence-free velocity field $\bu$ (and corresponding pressure $p$) solving the Stokes system in the fluid domain. 
More precisely, we aim to define a linear operator
\[
    \mathcal{L} : \psi \mapsto (\bu,\pi),
\]
which maps boundary data into solutions of the Stokes problem with the prescribed normal stress on $\Gamma$, while satisfying the remaining boundary conditions (such as periodicity and no-slip on $\Gamma_f^b$).

A key point is that this lifting operator enjoys optimal mapping properties: under the assumed regularity $\psi \in \rH_\per^{\frac{1}{2}}(\Gamma)^3$, the corresponding solution $(\bu,\pi)$ belongs to the natural maximal regularity class associated with the Stokes system. 
In particular, the map $\mathcal{L}$ is continuous between the relevant trace space on $\Gamma$ and the corresponding bulk spaces for $\bu$ and $\pi$.

This construction allows us to reinterpret the boundary condition as a volume contribution and is crucial for the subsequent analysis.
Combined with the homogeneous theory developed in the previous step, it enables us to treat the full coupled problem with inhomogeneous coupling condition, described by the operator $\bA_f^I$,  as a perturbation of the reference operator. 
In particular, the mapping properties of the lifting operator ensure that the resulting perturbation describing $\bA_f^I$ satisfies the assumptions required by the abstract perturbation results to conclude that $\bA_f^I$  is sectorial up to a shift (see \autoref{prop:sect of fluid & qbr}).

Details are presented next.

\medskip
\noindent{{\bfseries The introduction of the operators $\bA_{f,\max}^{\rI}$, $L$ and $Q$.}}
Let us start by introducing the setting for the maximal operator $\bA_{f,\max}^{\rI}$ capturing the incompressible free fluid flow in $\Omega_f$ coupled to the equation for the pressure jump across $\Gamma$. 
As mentioned earlier, the superscript ``$\rI$'' indicates the first operator matrix capturing the fluid subproblem in the two-part formulation, where the second part will include the average pressure contribution. 

Before proceeding, we introduce the space of solenoidal vector fields by recalling the Helmholtz projection $\bP$ onto the space of divergence-free vector fields on $\Omega_f$:
\begin{equation}\label{Projection}
    \bP \colon \rL^2(\Omega_f)^3 \to \rL_\sigma^2(\Omega_f), \twhere 
    \rL_\sigma^2(\Omega_f) \coloneqq \overline{\{\bu \in \rC_\per^\infty(\Omega_f)^3 : \mdiv \bu = 0 \tin \Omega_f\}}^{\| \cdot \|_{\rL^2(\Omega_f)}}.
\end{equation}
Here, the subscript ``$\sigma$'' is used to denote the subset of solenoidal vector fields.
Indeed, for the functions in~$\rL^2(\Omega_f)^3$, the Helmholtz decomposition exists for any open set $\Omega_f \subset \R^d$, see, e.g., \cite[Rem.~1]{HS:18}.
For the general case of the Helmholtz decomposition in $\rL^p(\Omega)$, assuming cylindrical domains $\Omega \subset \R^d$, see~\cite{Nau:15}.
The basic underlying Hilbert space in which the coupled fluid--pressure jump problem is formulated, i.e., the ground space, is then given by
\begin{equation}\label{eq:ground space pressure jump-fluid}
    \rX_{f,0}^{\rI} \coloneqq \rL^2(\Gamma) \times \rL_\sigma^2(\Omega_f),
\end{equation}
where $\rL_\sigma^2(\Omega_f)$ is defined in \eqref{Projection}.

To define the maximal operator $\bA_{f,\max}^{\rI}$ we must introduce two operators, one associated with the pressure jump subproblem defined on $\Gamma$ and the other one associated with the fluid subproblem defined on $\Omega_f$. 

For the operator associated with the pressure jump subproblem, we invoke the tangential Laplacian term $\Delta_{\btau}\qbr$ in equation \eqref{eq:FPSI Cauchy problem}$_3$, and introduce an $\rL^2$-realization of the tangential Laplacian operator with periodic boundary conditions, denoted by $\Delta_s$, as follows:
\begin{equation}\label{eq:Laplacian}
    \Delta_s f \coloneqq \Delta_{\btau} f, \tfor f \in \rD(\Delta_s) \coloneqq \rH_\per^2(\Gamma).
\end{equation}

For the operator associated with the free fluid subproblem (Stokes problem), we consider the viscous (Laplace) term in the Stokes equation \eqref{eq:FPSI Cauchy problem}$_5$ together with the divergence-free constraint \eqref{eq:FPSI Cauchy problem}$_6$. 
This motivates the introduction of the maximal Stokes operator $A_{\max}$, defined as the realization of the Laplacian on divergence-free vector fields subject to the non-interface boundary conditions.
(The term ``maximal'' is used here in a slightly nonstandard sense: 
it indicates that all equations and boundary conditions are incorporated, except for the coupling conditions imposed at the interface.)

Namely, we define the domain of the maximal Stokes operator $A_{\max}$ to be
\begin{equation*}
    \rD(A_{\max}) = \bigl\{\bu \in \rH_\per^2(\Omega_f)^3 \cap \rL_\sigma^2(\Omega_f) : \bu = 0 \ton \Gamma_f^b\bigr\},
\end{equation*}
and the operator is defined by
\begin{equation}\label{eq:max Stokes op}
    A_{\max} \bu \coloneqq \bP \Delta \bu, \tfor \bu \in \rD(A_{\max}).
\end{equation}

We are now in a position to define the maximal operator $\bA_{f,\max}^{\rI}$ associated with the fluid--pressure jump coupled problem. 

The domain of the operator $\bA_{f,\max}^{\rI}$ is given by
\begin{equation*}
    \rD(\bA_{f,\max}^{\rI}) \coloneqq \rH_\per^2(\Gamma) \times \rD(A_{\max})
\end{equation*}
and based on the full problem descriptions \eqref{eq:FPSI Cauchy problem}, the operator itself is defined by
\begin{equation}\label{bAmax}
    \bA_{f,\max}^{\rI} \coloneqq \begin{pmatrix}
        -\eps_2 \Delta_s + 48 & 0\\
        0 & -A_{\max}
    \end{pmatrix}
    = 
    \begin{pmatrix}
        -\eps_2 \Delta_s + 48 & 0\\
        0 & -\bP \Delta
    \end{pmatrix}.
\end{equation}

Next, we incorporate the coupling condition \eqref{NewCoupling}, restated here with $q^+ = 0$:
\begin{equation}\label{CoupleWithQ+=0}
    \bsigma_f(\bu,\pi) \bn = \qbr \bn - \beta \sum_{i=1}^2 (\bu \cdot \btau_i) \btau_i.
\end{equation}

We deal with this coupling condition in two steps.
First, we will consider the homogeneous condition $\bsigma_f(\bu,\pi) \bn = 0$ and show that the associated operator $\bA_{f,0}^{\rI}$ defined below in \eqref{OperatorA0}, is sectorial.
Then, we will use perturbation arguments to incorporate the right-hand side of \eqref{CoupleWithQ+=0}.

For this purpose, we introduce the boundary operator $L$, and the perturbation operator $\Phi$ as follows:
\begin{equation}\label{BoundaryOperators}
    L \binom{\qbr}{\bu} \coloneqq \binom{0}{\bsigma_f(\bu,\pi)\bn}, \qquad \Phi \binom{\qbr}{\bu} \coloneqq \binom{0}{\qbr \bn - \beta \sum_{i=1}^2 (\bu|_\Gamma \cdot \btau_i)\btau_i},
\end{equation}
where
\[
    L, \Phi \colon \rD(\bA_{f,\max}^{\rI}) \subset \rX_{f,0}^{\rI} \to \rX_{f,0}^{\rI},
\]
and $\rX_{f,0}^{\rI}$ is defined in \eqref{eq:ground space pressure jump-fluid}.

In this definition,  the first components of $L \binom{\qbr}{\bu}$ and $\Phi \binom{\qbr}{\bu}$ are set equal to zero.
This is merely a notational device and does not imply that the pressure jump $\qbr$ vanishes on the interface. 
Rather, the pressure jump enters the coupling through the dynamic interface condition \eqref{NewCoupling}, where it contributes to the normal stress balance.
Since $\qbr$ is defined only on the interface $\Gamma$, the first components in the above definitions play the role of auxiliary (or ``dummy'') entries, allowing us to embed the boundary operators into the product space $\rX_{f,0}^{\rI}$ and thereby fit the abstract operator-theoretic framework.

\medskip
\noindent{{\bfseries Homogeneous coupling condition.}}
We now introduce the operator $\bA_{f,0}^{\rI}$ to describe the fluid--pressure jump coupled problem subject to the {{homogeneous}} coupling condition $\bsigma_f(\bu,\pi)\bn = 0$ at the interface $\Gamma$:
\begin{equation}\label{OperatorA0}
    \bA_{f,0}^{\rI}
    \coloneqq \begin{pmatrix}
        -\eps_2 \Delta_s + 48 & 0\\
        0 & -A_0
    \end{pmatrix},
\end{equation}
where $\Delta_s$ is defined in \eqref{eq:Laplacian}, and
the Stokes operator $A_0$ is defined by:
\begin{equation}\label{HomoStokes}
    \begin{aligned}
        A_0 \bu 
        &\coloneqq \bP \Delta \bu,\\
        \rD(A_0) 
        &\coloneqq \left\{\bu \in \rH_\per^2(\Omega_f)^3 \cap \rL_\sigma^2(\Omega_f) : \bu = 0 \ton \Gamma_f^b \tand  \bsigma_f(\bu,\pi) \bn = 0 \ton \Gamma\right\}.
    \end{aligned}
\end{equation}

The domain of operator $\bA_{f,0}^{\rI}$ is given by:
\begin{equation}
    \begin{aligned}
        \rD(\bA_{f,0}^{\rI}) \coloneqq \rD(\Delta_s) \times \rD(A_0)
        &= \left\{\binom{\qbr}{\bu}\in \rH_\per^2(\Gamma) \times \rD(A_{\max}) : \bsigma_f(\bu,\pi)\bn = 0 \ton \Gamma \right\}\\
        &= \left\{\binom{\qbr}{\bu} \in \rD(\bA_{f,\max}^{\rI}) : L \binom{\qbr}{\bu} = 0\right\}.
    \end{aligned}
\end{equation}
We now show that the operator $\bA_{f,0}^{\rI}$ is sectorial on the ground space $\rX_{f,0}^{\rI}$ with the spectral angle less than~$\frac{\pi}{2}$, and that $0$ is in the resolvent set $\rho(\bA_{f,0}^{\rI})$. The spectral angle less than $\frac{\pi}{2}$ implies that $\bA_{f,0}^{\rI}$ generates an analytic semigroup.

More precisely, the following result holds true.

\begin{lem}\label{lem:sect of op with hom bc}
The operator $\bA_{f,0}^{\rI}$ is sectorial on $\rX_{f,0}^{\rI}$ with spectral angle $\phi_{\bA_{f,0}^{\rI}} < \frac{\pi}{2}$, and $0 \in \rho(\bA_{f,0}^{\rI})$.
\end{lem}

\begin{proof}
We summarize the main steps of the proof, which relies on classical results for elliptic operators.
First, by standard theory for elliptic problems on the cube $(0,\pi)^n$ (see, e.g., \cite[Ch.~7]{Nau:12}), which extends to the present setting by a rescaling argument, the operator $-\eps_2 \Delta_s + 48$ is sectorial on $\rL^2(\Gamma)$ with spectral angle $\phi_{-\eps_2 \Delta_s + 48} < \frac{\pi}{2}$.

Next, the sectoriality of the Stokes operator $A_0$ adapted to the present setting can be established along the lines of \cite[Sec.~7.3]{PS:16}; see also \cite{BKP:13}. 
This follows by reducing the problem to the case of an infinite layer, and subsequently to the half-space, via a suitable extension and cutoff argument. 
{ More precisely, in \cite{BKP:13}, the case of the domain being a half-space is especially treated.
The half-space case can then be reduced to the situation of an infinite layer of the form $\Omega = \R^2 \times (-1,0)$ upon invoking the same localization procedure as for bounded domains, see \cite{BKP:13}.
We also refer to \cite[Ch.~8]{DHP:03} for localization procedures from a half-space to bounded domains for general parabolic boundary value problems.
The case of $\Omega_f = (0,1)^2 \times (-1,0)$ with horizontal periodicity can then be reduced to the case of an infinite layer $\R^2 \times (-1,0)$ by extending the horizontally periodic functions to the full space, see also \cite[p.~1082]{HK:16} for more details on such an extension procedure.}
Due to the homogeneous Dirichlet boundary conditions at the bottom boundary $\Gamma_f^b$ of the fluid domain, the Stokes operator $A_0$ is invertible.

Finally, since $\bA_{f,0}^{\rI}$ has a diagonal structure, the claimed result follows immediately.
\end{proof}

The sectoriality of $\bA_{f,0}^{\rI}$ allows us to invoke the associated analytic semigroup and fractional power spaces.
Indeed, it gives rise to a natural scale of fractional domain spaces between $\rX_{f,0}^{\rI}$ as well as $\rX_{f,1}^{\rI} \coloneqq \rD(\bA_{f,0}^{\rI})$, where we have denoted 
\begin{equation*}
    \rX_{f,1}^{\rI} \coloneqq \rD(\bA_{f,0}^{\rI}) = \rD(\Delta_s) \times \rD(A_0).
\end{equation*}
These intermediate spaces play a crucial role in the analysis of the coupled problem, as they provide the correct regularity framework for traces, interface conditions, and perturbation arguments.
In particular, they are essential for capturing the coupled fluid-pressure jump problem with inhomogeneous interface conditions using perturbation of $\bA_{f,0}^{\rI}$ arguments, as we shall see below. 

In the following lemma, we characterize the complex and real interpolation spaces associated with the pair $(\rX_{f,0}^{\rI},\rX_{f,1}^{\rI})$. 
Before we state the lemma, we recall the Besov spaces $\rB_{2p,\per}^{2 \theta}(\Omega_f)$ and $\rB_{2p,\per}^{2 \theta}(\Gamma)$ that have been defined after \eqref{eq:Sobolev spaces with per bc}.
Moreover, we introduce the Besov space $\rB_{2p,\per,\Gamma_{f}^b}^{2 \theta}(\Omega_f)$ consisting of all the functions in $\rB_{2p,\per}^{2 \theta}(\Omega_f)$ that satisfy a homogeneous Dirichlet boundary condition at the bottom boundary $\Gamma_{f}^b$ of the fluid domain $\Omega_f$.
Similarly, the subscript ``$\Gamma_f^b$'' in $\rH_{\per,\Gamma_f^b}^{2 \theta}(\Omega_f)$ indicates homogeneous Dirichlet boundary conditions.

\begin{lem}\label{lem:interpol spaces}
Let $[\cdot,\cdot]_\theta$ and $(\cdot,\cdot)_{\theta,p}$ denote the complex and real interpolation functors, respectively.
Then, for all $\theta \in (\nicefrac{1}{4},\nicefrac{3}{4})$ and $p \in (1,\infty)$, we have
\begin{equation*}
    \rX_{f,\theta}^{\rI} \coloneqq [\rX_{f,0}^{\rI},\rX_{f,1}^{\rI}]_\theta = \rH_\per^{2 \theta}(\Gamma) \times \rH_{\per,\Gamma_f^b}^{2 \theta}(\Omega_f),
\end{equation*}
and
\begin{equation*}
    \rX_{f,\theta,p}^{\rI} \coloneqq (\rX_{f,0}^{\rI},\rX_{f,1}^{\rI})_{\theta,p} = \rB_{2p,\per}^{2 \theta}(\Gamma) \times \rB_{2p,\per,\Gamma_f^b}^{2 \theta}(\Omega_f)^3 \cap \rL_\sigma^2(\Omega_f).
\end{equation*}
\end{lem}

\begin{proof}
The result follows from standard interpolation theory; see, for instance, \cite{Tri:78}. 
For the interpolation of the divergence-free subspace $\rL_\sigma^2(\Omega_f)$, we also refer to \cite{Ama:00}.
\end{proof}

\medskip
\noindent{{\bfseries Inhomogeneous coupling condition and operator $\bA_f^I$.}}
In order to deal with the {{inhomogeneous}} coupling condition 
\begin{equation*}
    \bsigma_f(\bu,\pi) \bn = \qbr \bn - \beta \sum_{i=1}^2 (\bu \cdot \btau_i) \btau_i \ton \Gamma,
\end{equation*}
we introduce a new operator $\bA_f^{\rI}$ to be the same as the original operator $\bA_{f,\max}^{\rI}$ except for the interface condition, which is now inhomogeneous:
\begin{equation*}
    \bA_f^{\rI} \binom{\qbr}{\bu} = \bA_{f,\max}^{\rI} \binom{\qbr}{\bu},
\end{equation*}
where $\bA_{f,\max}^{\rI}$ is given in \eqref{bAmax}, but now the domain of $\bA_f^{\rI}$ includes the boundary conditions specified via the operators $L$ and $\Phi$ defined in \eqref{BoundaryOperators}:
\begin{equation*}
    \rD(\bA_f^{\rI}) = \left\{\binom{\qbr}{\bu} \in \rH_\per^2(\Gamma) \times \rH_{\per,\Gamma_f^b}^2(\Omega_f)^3 \cap \rL_\sigma^2(\Omega_f) : L \binom{\qbr}{\bu} = \Phi \binom{\qbr}{\bu}\right\}.
\end{equation*}

Our objective is to treat the fully coupled operator $\bA_f^{\rI}$ as a perturbation of the operator $\bA_{f,0}^{\rI}$, which corresponds to the problem with homogeneous interface conditions. 
This allows us to exploit the sectoriality of $\bA_{f,0}^{\rI}$ and apply abstract perturbation results to deduce well-posedness and regularity properties of the coupled problem.

A key difficulty in this approach is that the coupling conditions act at the interface and therefore naturally give rise to boundary terms, rather than operators defined on the state space $\rX_{f,0}^{\rI}$. 
To overcome this, we reformulate the coupling conditions as an operator acting on $\rX_{f,0}^{\rI}$. 
This is achieved in two steps: first, by lifting boundary data into the interior, and second, by interpreting the resulting operator within an appropriate extrapolation framework.

To this end, we consider the stationary Stokes problem: 
for $\lambda \in \C$ and $\psi \in \rH_\per^{\frac{1}{2}}(\Gamma)^3$, find $\bu$ such that
\begin{equation}\label{eq:Stokes lifting problem}
    \left\{
    \begin{aligned}
        (\lambda - A_{\max}) \bu 
        &= 0,\\
        \left.\bsigma_f(\bu,\pi)\bn \right|_\Gamma 
        &= \psi, &&\ton \Gamma,\\
        \bu 
        &= 0, &&\ton \Gamma_f^b.
    \end{aligned}
    \right.
\end{equation}
By classical results (see, e.g., \cite[Sec.~7.3]{PS:16}), this problem admits a unique solution 
\[
    \bu \in \rH_{\per,\Gamma_f^b}^2(\Omega_f)^3 \cap \rL_\sigma^2(\Omega_f)
\]
for all $\lambda \in \{z \in \C : \Rep z \ge 0\}$, and the result carries over to the present cylindrical setting.

We thus define the lifting operator
\[
    L_0 \colon \{0\} \times \rH_\per^{\frac{1}{2}}(\Gamma)^3 \to \{0\} \times \rH_{\per,\Gamma_f^b}^2(\Omega_f)^3 \cap \rL_\sigma^2(\Omega_f)
\]
by
\[
    L_0 \binom{0}{\psi} \coloneqq \binom{0}{\bu},
\]
where $\bu$ is the unique solution to \eqref{eq:Stokes lifting problem}. 
This operator is well-defined and bounded.

\medskip
\noindent{{\bfseries Reformulation as a perturbation.}}
Our goal is to show that the coupled operator $\bA_f^{\rI}$ can be represented as a perturbation of the operator $\bA_{f,0}^{\rI}$, which corresponds to homogeneous interface conditions. 

For this purpose we will introduce two operators, $Q_0$ and $Q$. 
The operator $Q_0$ will represent the effect of the coupling condition lifted into the interior via $L_0$, see \eqref{Q0Q}$_1$, while $Q$ will allow us to interpret the boundary-induced operator $L_0$ as a lower-order perturbation operator -- this will be achieved via the operator $(\bA_{f,0}^{\rI})_{-1}$, which is the extrapolation of $\bA_{f,0}^{\rI}$ to the lower-regularity space $\rX_{f,-1}^{\rI}$, see \eqref{Q0Q}$_2$.

We briefly recall that extrapolation spaces associated with $\bA_{f,0}^{\rI}$ extend the scale of interpolation spaces beyond $\rX_{f,0}^{\rI}$ to lower-regularity spaces $\rX_{f,-1}^{\rI}, \rX_{f,-\frac{1}{2}}^{\rI}, \ldots$. 
The extrapolated operator $(\bA_{f,0}^{\rI})_{-1}$ then acts continuously from $\rX_{f,0}^{\rI}$ into $\rX_{f,-1}^{\rI}$, allowing us to interpret our boundary-induced operator $L_0$ as a lower-order perturbation of $\bA_{f,0}^{\rI}$. 
This mechanism is essential for applying abstract perturbation results.

Therefore, for the boundary operator $\Phi$, we introduce
\begin{equation}\label{Q0Q}
    Q_0 \coloneqq L_0 \Phi, \qquad Q \coloneqq -(\bA_{f,0}^{\rI})_{-1} Q_0 = -(\bA_{f,0}^{\rI})_{-1} L_0 \Phi,
\end{equation}
where $(\bA_{f,0}^{\rI})_{-1}$ denotes the extrapolation of $\bA_{f,0}^{\rI}$ to the space $\rX_{f,-1}^{\rI}$.

\begin{lem}\label{lem:mapping props of Q}
Let $\gamma \in (0,\tfrac{3}{4})$. Then
\[
    Q_0 \in \cL(\rX_{f,\frac{1}{2}}^{\rI}, \rX_{f,\gamma}^{\rI}), \qquad Q \in \cL(\rX_{f,\frac{1}{2}}^{\rI}, \rX_{f,\gamma-1}^{\rI}).
\]
\end{lem}

\begin{proof}
From \autoref{lem:interpol spaces}, we recall that $\rX_{f,\frac{1}{2}}^{\rI} = [\rX_{f,0}^{\rI},\rX_{f,1}^{\rI}]_{{\frac{1}{2}}} = \rH_\per^{1}(\Gamma) \times \rH_{\per,\Gamma_f^b}^{1}(\Omega_f)$.
The definition of~$\Phi$, together with classical trace results (see \cite{Tri:78}), yields
\[
    \Phi \in \cL\bigl(\rX_{f,\frac{1}{2}}^{\rI}, \{0\} \times \rH_\per^{\frac{1}{2}}(\Gamma)^3\bigr).
\]
On the other hand, by construction of the lifting operator,
\[
    L_0 \in \cL\bigl(\{0\} \times \rH_\per^{\frac{1}{2}}(\Gamma)^3,\, \{0\} \times \rH_{\per,\Gamma_f^b}^2(\Omega_f)^3 \cap \rL_\sigma^2(\Omega_f)\bigr).
\]
Using the characterization of the space $\rX_{f,\gamma}^{\rI}$ provided by \autoref{lem:interpol spaces}, we have
\[
    \{0\} \times \rH_{\per,\Gamma_f^b}^2(\Omega_f)^3 \cap \rL_\sigma^2(\Omega_f) \hookrightarrow \rX_{f,\gamma}^{\rI} \quad \text{for } \gamma < \tfrac{3}{4}.
\]
Thus, $Q_0 = L_0 \Phi$ maps $\rX_{f,\frac{1}{2}}^{\rI}$ into $\rX_{f,\gamma}^{\rI}$. 
The first assertion follows by the closed graph theorem.
The second assertion is a direct consequence of the definition of $(\bA_{f,0}^{\rI})_{-1}$.
\end{proof}

The diagram below illustrates the interpolation--extrapolation scale associated with $\bA_{f,0}^{\rI}$ and the construction of the perturbation operator. 
The operator $\Phi$ maps intermediate states in $\rX_{f,\frac{1}{2}}^{\rI}$ to boundary data, which is lifted into the interior via $L_0$, and subsequently mapped into a lower-regularity space by the extrapolated operator $(\bA_{f,0}^{\rI})_{-1}$. 
This yields the perturbation operator $Q = -(\bA_{f,0}^{\rI})_{-1} L_0 \Phi$, which acts as a lower-order term in the interpolation scale.

\begin{figure}[h!]
    \centering
    \begin{small}
        \begin{tikzcd}[column sep=large, row sep=large]
            \rX_{f,1}^{\rI} \arrow[r, hook] & \rX_{f,\frac{1}{2}}^{\rI} \arrow[r, hook] \arrow[d, "\Phi"] & \rX_{f,0}^{\rI} \arrow[r, hook] & \rX_{f,-1}^{\rI} \\
            & \{0\} \times \rH_\per^{\frac{1}{2}}(\Gamma)^3 \arrow[d, "L_0"] & & \\
            & \rX_{f,1}^{\rI} \arrow[d, "(\bA_{f,0}^{\rI})_{-1}"] & & \\
             & \rX_{f,\gamma-1}^{\rI} & &
        \end{tikzcd}
    \end{small}
    \caption{Interpolation--extrapolation scale associated with $\bA_{f,0}^{\rI}$ and the construction of the perturbation operator $Q = -(\bA_{f,0}^{\rI})_{-1} L_0 \Phi$.}
\end{figure}

\medskip

We now show that the operator $\bA_f^{\rI}$ can be realized as a perturbation of a fractional power of $\bA_{f,0}^{\rI}$.

\begin{lem}\label{lem:representation of the coupled op}
The operator $\bA_f^{\rI}$ admits the representation
\[
    \bA_f^{\rI} = \left.((\bA_{f,0}^{\rI})_{-\frac{1}{2}} + Q)\right|_{\rX_{f,0}^{\rI}},
\]
that is,
\[
    \bA_f^{\rI} \bv = ((\bA_{f,0}^{\rI})_{-\frac{1}{2}} + Q)\bv, \tfor \bv \in \rD(\bA_f^{\rI}) = \{ \bv \in \rX_{f,0}^{\rI} : ((\bA_{f,0}^{\rI})_{-\frac{1}{2}} + Q)\bv \in \rX_{f,0}^{\rI} \}.
\]
\end{lem}

\begin{proof}
Let $\widetilde{\bA_f^{\rI}}$ denote the operator on the right-hand side. 
We first show that $\rD(\bA_f^{\rI}) = \rD(\widetilde{\bA_f^{\rI}})$.

Let $\bv \in \rD(\bA_f^{\rI})$, so that $\bv \in \rD(\bA_{f,\max}^{\rI})$ and $L\bv = \Phi\bv$. 
Using the identity $L L_0 \Phi = \Phi$, we obtain
\[
    L(1 - Q_0)\bv = L\bv - L L_0 \Phi \bv = 0,
\]
hence $(1 - Q_0)\bv \in \ker(L) = \rX_{f,1}^{\rI}$. 
Therefore,
\[
    (\bA_{f,0}^{\rI})_{-1}(1 - Q_0)\bv \in \rX_{f,0}^{\rI}.
\]
Since $\bv \in \rD(\bA_f^{\rI}) \subset \rD(\bA_{f,\max}^{\rI}) \hookrightarrow \rX_{f,\frac{1}{2}}^{\rI}$, we have 
\begin{equation*}
    (\bA_{f,0}^{\rI})_{-\frac{1}{2}}\bv = (\bA_{f,0}^{\rI})_{-1}\bv, \tfor \bv \in \rX_{f,\frac{1}{2}}^{\rI}.
\end{equation*} 
Thus, recalling the definition of $Q = - (\bA_{f,0}^{\rI})_{-1}Q_0$, we obtain
\begin{equation*}
    (\bA_{f,0}^{\rI})_{-\frac{1}{2}} \bv + Q \bv = (\bA_{f,0}^{\rI})_{-1} \bv - (\bA_{f,0}^{\rI})_{-1} Q_0 \bv = (\bA_{f,0}^{\rI})_{-1}(1-Q_0)\bv \in \rX_{f,0}^{\rI},
\end{equation*}
where we additionally used that $L_0  \Phi \bv \in \rX_{f,\gamma}^{\rI}$ for all $\gamma < \frac{3}{4}$.
This shows that $\bv \in \rD(\widetilde{\bA_f^{\rI}})$.

Conversely, let $\bv \in \rD(\widetilde{\bA_f^{\rI}})$, which means $\bv \in \rX_{f,\frac{1}{2}}^{\rI}$ with $(\bA_{f,0}^{\rI})_{-\frac{1}{2}} \bv + Q \bv \in \rX_{f,0}^{\rI}$.
In a similar way as above, we can show that
\begin{equation*}
    (\bA_{f,0}^{\rI})_{-1}(1-Q_0)\bv = (\bA_{f,0}^{\rI})_{-\frac{1}{2}} \bv + Q \bv\in \rX_{f,0}^{\rI}.
\end{equation*}
Recalling that
\begin{equation*}
    \bv = (1-Q_0)\bv + Q_0 \bv \in \rX_{f,1}^{\rI} \oplus \ker(\bA_{f,\max}^{\rI}),
\end{equation*}
we obtain $(1-Q_0) \bv \in \rX_{f,1}^{\rI} = \ker(L)$ by construction.
Thus, we may apply $L$ to $\bv$, and make use of the fact that $(1-Q_0)\bv \in \ker(L)$ and $L L_0 \Phi = \Phi$, to obtain
\begin{equation*}
    L \bv = L(1-Q_0)\bv + L L_0 \Phi \bv = \Phi \bv.
\end{equation*}
This show that $\bv \in \rD(\bA_f^{\rI})$ and therefore, $\rD(\bA_f^{\rI}) = \rD(\widetilde{\bA_f^{\rI}})$.

Finally, we show that $\widetilde{\bA_f^{\rI}} = \bA_f^{\rI}$ by employing that $Q_0 \bv \in \ker(\bA_{f,\max}^{\rI})$, and $\bv \in \rD(\bA_f^{\rI}) \subset \rX_{f,\frac{1}{2}}^{\rI}$ as well as $\rD(\bA_f^{\rI}) \subset \rD(\bA_{f,\max}^{\rI})$, to obtain
\begin{equation*}
    \widetilde{\bA_f^{\rI}} \bv = ((\bA_{f,0}^{\rI})_{-\frac{1}{2}} + Q)\bv = ((\bA_{f,0}^{\rI})_{-\frac{1}{2}} - (\bA_{f,0}^{\rI})_{-1}Q_0)\bv = (\bA_{f,0}^{\rI})_{-\frac{1}{2}} \bv= \bA_f^{\rI} \bv. \qedhere
\end{equation*}
\end{proof}

Using the perturbation representation of the coupled Stokes--pressure jump operator with inhomogeneous coupling condition discussed above, we now show that this operator is sectorial up to a shift.

\begin{lem}\label{prop:sect of fluid & qbr}
There exists $\omega \ge 0$ such that $\bA_f^{\rI} + \omega$ is sectorial on $\rX_{f,0}^{\rI}$ with spectral angle $\phi_{\bA_f^{\rI} + \omega} < \frac{\pi}{2}$.
\end{lem}

\begin{proof}
From \autoref{lem:representation of the coupled op}, we recall the representation of the operator 
\begin{equation*}
    \bA_f^{\rI} = \left.((\bA_{f,0}^{\rI})_{-\frac{1}{2}} + Q)\right|_{\rX_{f,0}^{\rI}}.
\end{equation*}
On the other hand, from \autoref{lem:mapping props of Q}, we recall that $Q \in \cL(\rX_{f,\frac{1}{2}}^{\rI},\rX_{f,\gamma-1}^{\rI})$ for $\gamma < \frac{3}{4}$.
Thus, in the setting of \autoref{lem:abstr pert result} in the appendix, choosing $\alpha = \gamma = \frac{1}{2}$, and $\beta \in (\frac{1}{2},\frac{3}{4})$, we deduce from this lemma together with the result for the operator with homogeneous boundary conditions, \autoref{lem:sect of op with hom bc}, the existence of a sufficiently large $\omega \ge 0$ such that $\bA_f^{\rI} + \omega$ is sectorial on $\rX_{f,0}^{\rI}$ with spectral angle $\phi_{\bA_f^{\rI} + \omega} < \frac{\pi}{2}$.
\end{proof}

\subsection{The addition of the average pressure $\bar{q}$ equation and the operator $\bA_f^{\rII}$}
\ 

To complete the operator description of the fluid component of the entire coupled problem, what is left is to incorporate the average pressure variable $\qbar$.

The introduction of $\qbar$ leads to a coupling structure that is not immediately compatible with the operator~$\bA_f^{\rI}$ studied above. 
In particular, the extended system does not directly have block structure. 
This prevents a direct application of the sectoriality result for $\bA_f^{\rI}$. 
To overcome this, we perform a similarity transformation, which amounts to a change of variables designed to reorganize the operator matrix into a form that separates the leading (sectorial) part involving the fluid velocity and the jump in the pressure from the average of the pressure.

This reformulation is crucial: 
it allows us to transfer the sectoriality of $\bA_f^{\rI}$ to the enlarged system by treating the additional terms involving $\qbar$ as a lower-order perturbation. 
The resulting formulation can then be analyzed using perturbation arguments for sectorial operators.

More precisely, based on the full problem formulation \eqref{eq:FPSI Cauchy problem}, we extend the fluid subsystem by incorporating the average pressure variable $\qbar$. 
This leads to the operator matrix $\bA_f^{\rII}$ given by
\begin{equation}\label{eq:op matrix A_q}
    \begin{aligned}
        \bA_f^{\rII} 
        &= \begin{pmatrix}
            -\eps_2 \Delta_s + 48 & 72 & 12\, \tr_\Gamma(\cdot)\,\bn\\
            6 & 12 & 0\\
            0 & 0 & -\bP \Delta
        \end{pmatrix}, \text{ with domain} \\
        \rD(\bA_f^{\rII})
        &= \left\{
        \begin{pmatrix}
            \qbr\\ \qbar\\ \bu
        \end{pmatrix} \in \rH_\per^2(\Gamma) \times \rL^2(\Gamma) \times \rH_{\per,\Gamma_f^b}^2(\Omega_f)^3 \cap \rL_\sigma^2(\Omega_f) :\; L \binom{\qbr}{\bu} = \Phi \binom{\qbr}{\bu}
        \right\}.
    \end{aligned}
\end{equation}

The operator $\bA_f^{\rII}$ is considered on the ground space
\[
    \rX_{f,0}^{\rII} \coloneqq \rL^2(\Gamma) \times \rL^2(\Gamma) \times \rL_\sigma^2(\Omega_f),
\]
which represents the natural energy space for the extended fluid variables $(\qbr,\qbar,\bu)$.
In particular, $\rX_{f,0}^{\rII}$ serves as the base space in which sectoriality and semigroup properties of $\bA_f^{\rII}$ will be analyzed, while the domain $\rD(\bA_f^{\rII})$ encodes the corresponding higher regularity and coupling conditions.

To reveal the underlying block structure of $\bA_f^{\rII}$, we perform a similarity transformation that interchanges the roles of $\qbr$ and $\qbar$:
\begin{equation*}
    S = S^{-1} = \begin{pmatrix}
        0 & \Id & 0\\
        \Id & 0 & 0\\
        0 & 0 & \Id
    \end{pmatrix}.
\end{equation*}
Since sectoriality and the associated spectral angle are invariant under similarity transformations (see, e.g., \cite[Prop.~1.3(vi)]{DHP:03}), it suffices to study the transformed operator
\[
    \widetilde{\bA_f^{\rII}} \coloneqq S \bA_f^{\rII} S^{-1}.
\]

A direct computation shows that $\widetilde{\bA_f^{\rII}}$ takes the form
\begin{equation*}
    \begin{aligned}
        \widetilde{\bA_f^{\rII}} 
        &= \begin{pmatrix}
            12 & 6 & 0\\
            72 & -\eps_2 \Delta_s + 48 & 12 \tr_\Gamma(\cdot) \bn\\
            0 & 0 & -\bP \Delta
        \end{pmatrix}
        = \begin{pmatrix}
            12 & \B^{\rI}\\
            \C^{\rI} & \bA_f^{\rI}
        \end{pmatrix} + \widetilde{\bA_f^\rp},
    \end{aligned}
\end{equation*}
with domain
\[
    \rD(\widetilde{\bA_f^{\rII}}) = \rL^2(\Gamma) \times \rD(\bA_f^{\rI}),
\]
where
\[
    \B^{\rI} = \begin{pmatrix} 6 & 0 \end{pmatrix}, \qquad  \C^{\rI} = \binom{72}{0}, \tand \widetilde{\bA_f^\rp} 
    = \begin{pmatrix}
        0 & 0 & 0\\
        0 & 0 & 12 \tr_\Gamma(\cdot)\bn\\
        0 & 0 & 0
    \end{pmatrix}.
\]

The operator $\widetilde{\bA_f^{\rII}}$ acts on the ground space
\[
    \widetilde{\rX_{f,0}^{\rII}} = \rL^2(\Gamma) \times \rX_{f,0}^{\rI},
\]
where we recall the space $\rX_{f,0}^{\rI}$ from \eqref{eq:ground space pressure jump-fluid}.
In this representation, the principal part of the operator is given by the block matrix $\begin{pmatrix} 12 & \B^{\rI} \\ \C^{\rI} & \bA_f^{\rI} \end{pmatrix}$, where $\bA_f^{\rI}$ is the sectorial operator identified in the previous section.
The remaining term $\widetilde{\bA_f^\rp}$ is of lower order and will be treated as a perturbation.

This decomposition makes the structure of $\widetilde{\bA_f^{\rII}}$ amenable to perturbation arguments and allows us to transfer the sectoriality of $\bA_f^{\rI}$ to the full operator matrix. 
The sectoriality of $\widetilde{\bA_f^{\rII}}$, and hence of $\bA_f^{\rII}$, is established in the lemma below.
In the rest of the manuscript, we will use $\cS(\rX)$ to denote all sectorial operators $A \colon \rD(A) \subset \rX \to \rX$ on a Banach space~$\rX$.

\begin{lem}\label{lem:sect of bA_q}
There exists $\omega \ge 0$ such that $\widetilde{\bA_f^{\rII}} + \omega \in \cS(\widetilde{\rX_{f,0}^{\rII}})$ with spectral angle $\phi_{\widetilde{\bA_f^{\rII}} + \omega} < \frac{\pi}{2}$.
Consequently, $\bA_f^{\rII} + \omega \in \cS(\rX_{f,0}^{\rII})$ with
\[
    \phi_{\bA_f^{\rII} + \omega} = \phi_{\widetilde{\bA_f^{\rII}} + \omega} < \frac{\pi}{2}.
\]
\end{lem}

\begin{proof}
We decompose the operator $\widetilde{\bA_f^{\rII}}$ as
\[
    \widetilde{\bA_f^{\rII}}= \widetilde{\bA_f^\mrm} + \widetilde{\bA_f^\rp}, \twhere \widetilde{\bA_f^\mrm} \coloneqq 
    \begin{pmatrix}
        12 & \B^{\rI}\\
        \C^{\rI} & \bA_f^{\rI}
    \end{pmatrix}.
\]
Here, the superscript ``$\mrm$'' stands for ``main'', while the superscript ``$\rp$'' stands for ``perturbation''.
We first show that $\widetilde{\bA_f^\mrm}+\omega$ is sectorial on $\widetilde{\rX_{f,0}^{\rII}}$ for sufficiently large $\omega \ge 0$, and then treat $\widetilde{\bA_f^\rp}$ as a relatively bounded perturbation.

To this end, we invoke \autoref{lem:sect of diag dom op matrices}(a). The operator $12\Id$ is sectorial on $\rL^2(\Gamma)$ with spectral angle $0$, while by \autoref{prop:sect of fluid & qbr}, there exists $\omega_1 \ge 0$ such that $\bA_f^{\rI}+\omega_1$ is sectorial on $\rX_{f,0}^{\rI}$, where the latter space has been introduced in \eqref{eq:ground space pressure jump-fluid}, with spectral angle strictly less than $\frac{\pi}{2}$. 
Moreover,
\[
    \B^{\rI} \in \cL(\rX_{f,0}^{\rI},\rL^2(\Gamma)) \tand \C^{\rI} \in \cL(\rL^2(\Gamma),\rX_{f,0}^{\rI}).
\]
Since $12\Id$ is bounded, we have
\[
    \rD\bigl((12\Id)^\alpha\bigr)=\rL^2(\Gamma) \tforall \alpha \in (0,1),
\]
and therefore
\[
    \C^{\rI} \in \cL\bigl(\rD((12\Id)^\alpha),\rX_{f,0}^{\rI}\bigr) \tforall \alpha \in (0,1).
\]
Hence all assumptions of \autoref{lem:sect of diag dom op matrices}(a) are satisfied. 
Consequently, for every
\[
    \psi \in \Bigl(\max\{0,\phi_{\bA+\omega_1}\},\frac{\pi}{2}\Bigr)
\]
there exists $\omega_2 \ge 0$ such that, for all $\omega \ge \omega_2$, the operator
\[
    \widetilde{\bA_f^\mrm}+\omega =
    \begin{pmatrix}
        12 & \B^{\rI}\\
        \C^{\rI} & \bA_f^{\rI}
    \end{pmatrix} +\omega
\]
is sectorial on $\widetilde{\rX_{f,0}^{\rII}}$ with spectral angle bounded by $\psi$, and hence strictly less than $\frac{\pi}{2}$.

Next, we treat $\widetilde{\bA_f^\rp}$ as a perturbation of $\widetilde{\bA_f^\mrm}$. 
Let
\[
    \by=(\qbar,\qbr,\bu)\in \rD(\widetilde{\bA_f^{\rII}})=\rD(\widetilde{\bA_f^\mrm}).
\]
By the mapping properties of the trace, for $\delta > 0$ small, we first obtain
\[
    \|\widetilde{\bA_f^\rp}\by\|_{\widetilde{\rX_{f,0}^{\rII}}} = \|12\,\tr_\Gamma(\bu)\cdot \bn\|_{\rL^2(\Gamma)} \le C\|\bu\|_{\rH_\per^{\frac12+\delta}(\Omega_f)},
\]
interpolation as well as Young's inequality then yield
\[
    \|\bu\|_{\rH_\per^{\frac12+\delta}(\Omega_f)}
    \le C\|\bu\|_{\rL_\sigma^2(\Omega_f)}^\beta \|\bu\|_{\rH_\per^2(\Omega_f)\cap\rL_\sigma^2(\Omega_f)}^{1-\beta}
    \le C(\eps)\|\bu\|_{\rL_\sigma^2(\Omega_f)}+\eps \|\bu\|_{\rH_\per^2(\Omega_f)\cap\rL_\sigma^2(\Omega_f)}.
\]
Since $\rD(\widetilde{\bA_f^\mrm})=\rL^2(\Gamma)\times \rD(\bA_f^{\rI})$, the graph norm of $\widetilde{\bA_f^\mrm}$ controls the norm of the fluid component in
\[
    \rH_\per^2(\Omega_f)^3\cap \rL_\sigma^2(\Omega_f).
\]
Therefore, for every $\eps > 0$, there exists $C(\eps) > 0$ such that
\begin{equation*}
    \|\widetilde{\bA_f^\rp}\by\|_{\widetilde{\rX_{f,0}^{\rII}}} 
    \le C(\eps) \|\by\|_{\widetilde{\rX_{f,0}^{\rII}}} + \eps \|\widetilde{\bA_f^\mrm} \by\|_{\widetilde{\rX_{f,0}^{\rII}}}.
\end{equation*}
Thus, $\widetilde{\bA_f^\rp}$ is relatively $\widetilde{\bA_f^\mrm}$-bounded with relative bound zero. 
In particular, choosing $\eps>0$ sufficiently small, the assumptions of \autoref{lem:rel bdd pert} are satisfied.

Applying \autoref{lem:rel bdd pert} to $A=\widetilde{\bA_f^\mrm}+\omega$ and $B=\widetilde{\bA_f^\rp}$, we conclude that, after possibly increasing $\omega$, the operator
\[
    \widetilde{\bA_f^{\rII}}+\omega=\widetilde{\bA_f^\mrm}+\widetilde{\bA_f^\rp}+\omega
\]
is sectorial on $\widetilde{\rX_{f,0}^{\rII}}$ with spectral angle less than or equal to that of $\widetilde{\bA_f^\mrm}+\omega$, and hence still strictly less than $\frac{\pi}{2}$. This proves the first assertion.

Finally, the second assertion of the lemma follows from the fact that 
\[
    \widetilde{\bA_f^{\rII}} = S \bA_f^{\rII} S^{-1},
\]
and the fact that sectoriality together with the spectral angle are both preserved under similarity transformations.
\end{proof}

\subsection{The addition of the plate subproblem and the resulting full FPSI operator matrix $\cA$}\label{ssec:sect of the full op}
\ 

In this section, we establish the sectoriality of the operator matrix associated with the full regularized FPSI problem and, in addition, investigate the exponential stability of the semigroup it generates. 
These properties are essential for obtaining well-posedness and long-time behavior of solutions to the coupled system.

We introduce the operator matrix $\cA$ corresponding to the full problem, which includes the elastodynamics of the plate, the fluid velocity $\bu$, and both pressure variables -- the jump $\qbr$ and the average $\qbar$. 
The goal is to show that this fully coupled operator inherits the sectorial structure from its individual components.

To this end, we proceed in two steps.
First, we isolate the operator associated with the elastodynamics of the plate and show that it is sectorial. 
The viscoelastic plate subproblem is parabolic, and it should be compatible with the parabolic framework required for the coupled analysis.

Next, we combine this result with the sectoriality of the operator matrix describing the fluid subsystem, which now includes $\bu$, $\qbr$, and $\qbar$.
Using a perturbation argument, we show that the full operator matrix~$\cA$ is sectorial up to a shift. 
This reflects the fact that coupling terms can alter the spectral position but do not destroy the underlying analytic semigroup structure.

Finally, to remove the shift and show sectoriality of $\cA$, and to also establish exponential stability of the semigroup generated by $\cA$, we carry out a spectral analysis of the operator. 
To this end, we investigate the point spectrum, and we reduce the study of the essential spectrum to the situation of a diagonal operator matrix. 
This reduction allows us to isolate the dominant spectral contributions and to conclude stability of the full coupled system.

\medskip
\noindent{{\bfseries The plate elastodynamics problem.}}
We start by introducing the operator matrix $A_p$ associated with the viscoelastic plate dynamics.
This operator captures the evolution of the plate displacement and velocity and will serve as the structural component in the coupled system. 
Establishing its analytic properties is a key step toward understanding the full fluid--structure interaction problem.

To this end, we consider the $\rL^2$-realization of the bi-Laplacian subject to periodic boundary conditions. 
More precisely, we define
\begin{equation}\label{eq:bi-Laplacian}
    \Delta_s^2 f \coloneqq \Delta_{\btau}^2 f, 
    \tfor 
    f \in \rD(\Delta_s^2) \coloneqq \rH_\per^4(\Gamma).
\end{equation}

On the ground space
\[
    \rX_0^p = \rD(\Delta_s^2) \times \rL^2(\Gamma) = \rH_\per^4(\Gamma) \times \rL^2(\Gamma),
\]
which represents the natural energy space for the plate displacement and velocity, we define the operator matrix $A_p$ by
\begin{equation}\label{eq:plate op matrix}
    A_p \coloneqq 
    \begin{pmatrix}
        0 & -1\\
        \Delta_s^2 + \gamma^p & \eps_1 \Delta_s^2
    \end{pmatrix}, 
    \qquad 
    \rD(A_p) = \rH_\per^4(\Gamma) \times \rH_\per^4(\Gamma).
\end{equation}

\begin{lem}\label{lem:sect of A_p}
The operator $A_p$ is sectorial on $\rX_0^p$ with spectral angle $\phi_{A_p} < \frac{\pi}{2}$, and $0 \in \rho(A_p)$.
\end{lem}

\begin{proof}
The proof relies on viewing $A_p$ as a diagonally dominant block operator matrix and applying \autoref{lem:sect of diag dom op matrices}(b) presented in the appendix. 
To this end, we first verify that $A_p$ satisfies the required diagonal dominance condition.

We note that the operator $\Delta_s^2 + \gamma^p \colon \rD(\Delta_s^2) \to \rL^2(\Gamma)$ is bounded. 
Moreover, for every $f \in \rD(\Delta_s^2)$, since the graph norm of $\Delta_s^2 + \gamma^p$ is isomorphic with $\rH_\per^4(\Gamma)$, we have
\begin{equation*}
    \| f \|_{\rD(\Delta_s^2)} \le C\bigl( \| (\Delta_s^2 + \gamma^p) f \|_{\rL^2(\Gamma)} + \| f \|_{\rL^2(\Gamma)} \bigr).
\end{equation*}

In fact, this verifies diagonal dominance in the sense of \autoref{def:diag dom block op matrix} for the present setting. 
More precisely, we consider the spaces and operators
\[
    \rX_1 = \rH_\per^4(\Gamma), \quad \rX_2 = \rL^2(\Gamma), \quad A = 0, \quad B = -1, \quad C = \Delta_s^2 + \gamma^p, \quad D = \eps_1 \Delta_s^2,
\]
with domains $\rD(A) = \rD(B) = \rD(C) = \rD(D) = \rH_\per^4(\Gamma)$.

The estimate derived above shows that for every $f \in \rH_\per^4(\Gamma)$,
\begin{equation*}
    \begin{aligned}
        \| B f \|_{\rX_1} = \| -f \|_{\rH_\per^4(\Gamma)} 
        &\le c \Bigl( \| (\Delta_s^2 + \gamma^p) f \|_{\rL^2(\Gamma)} + \| f \|_{\rL^2(\Gamma)} \Bigr) \\
        &\le c \Bigl( \| \Delta_s^2 f \|_{\rL^2(\Gamma)} + \| f \|_{\rL^2(\Gamma)} \Bigr) = c_D \, \| D f \|_{\rX_2} + C_0 \, \| f \|_{\rX_2}.
    \end{aligned}
\end{equation*}
This establishes the required estimate for the operator $B$ in the definition of diagonal dominance.

Since $A = 0$, the corresponding estimate for the operator $C$ reduces to verifying that
\[
    C = \Delta_s^2 + \gamma^p \colon \rH_\per^4(\Gamma) \to \rL^2(\Gamma)
\]
is bounded, which follows immediately from the definition of the bi-Laplacian.
Hence, all conditions of \autoref{def:diag dom block op matrix} are satisfied, and the operator matrix under consideration is diagonally dominant.

Next, by classical elliptic theory on cylindrical domains (see, e.g., \cite[Ch.~7]{Nau:12}), the operator $\Delta_s^2 + \omega$ is sectorial on $\rL^2(\Gamma)$ for every $\omega > 0$, with spectral angle strictly less than $\frac{\pi}{2}$. 
Therefore, an application of \autoref{lem:sect of diag dom op matrices}(b) yields the existence of $\omega' \ge 0$ such that $A_p + \omega'$ is sectorial on $\rX_0^p$ with spectral angle less than $\frac{\pi}{2}$.

To remove the shift, it remains to show that $A_p$ is invertible. 
This follows from the explicit representation of the inverse
\begin{equation*}
    A_p^{-1} = 
    \begin{pmatrix}
        (\Delta_s^2 + \gamma^p)^{-1} \eps_1 \Delta_s^2 & (\Delta_s^2 + \gamma^p)^{-1}\\
        -1 & 0
    \end{pmatrix},
\end{equation*}
which is well-defined and bounded on $\rX_0^p$. 
Consequently, $0 \in \rho(A_p)$, and the sectoriality of $A_p$ (without shift) follows.
\end{proof}

\medskip
\noindent{{\bfseries The full fluid--structure coupled problem and operator $\cA$.}}
We are now in a position to combine the plate dynamics and the fluid subsystem into a single operator formulation of the fully coupled problem. 
The goal is to represent the entire fluid--structure interaction system as a block operator matrix, in which the previously analyzed components appear as principal parts and the coupling terms enter as perturbations.

To this end, we define the full operator matrix $\cA$ by
\begin{equation}\label{eq:op FPSI problem}
    \begin{aligned}
        \cA 
        &= \begin{pmatrix}
            0 & -1 & 0 & 0 & 0\\
            \Delta_s^2 + \gamma^p & \eps_1 \Delta_s^2 & \frac{1}{12} \Delta_s + 1 & 0 & 0\\
            0 & -\Delta_s - 12 & -\eps_2 \Delta_s + 48 & 72 & 12 \tr_\Gamma(\cdot)\bn\\
            0 & 0 & 6 & 12 & 0\\
            0 & 0 & 0 & 0 & -\bP \Delta
        \end{pmatrix}
        = \begin{pmatrix}
            A_p & \B^{\rII}\\
            \C^{\rII} & \bA_f^{\rII}
        \end{pmatrix}, \twith\\
        \rD(\cA) 
        &= \rD(A_p) \times \rD(\bA_f^{\rII}),
    \end{aligned}
\end{equation}
where the coupling operators are given by
\begin{equation}\label{BandC}
    \B^{\rII} 
    = \begin{pmatrix}
        0 & 0 & 0\\
        \frac{1}{12} \Delta_s + 1 & 0 & 0
    \end{pmatrix}, \qquad
    C^{\rII} 
    = \begin{pmatrix}
        0 & -\Delta_s - 12\\
        0 & 0\\
        0 & 0
    \end{pmatrix}.
\end{equation}
The operator $\cA$ acts on the ground space
\begin{equation}\label{eq:ground space}
    \rX_0 = \rD(\Delta_s^2) \times \rL^2(\Gamma) \times \rL^2(\Gamma) \times \rL^2(\Gamma) \times \rL_\sigma^2(\Omega_f).
\end{equation}

This block representation separates the structural part $A_p$ and the fluid part $\bA_f^{\rII}$ from the coupling operators $\B^{\rII}$ and $\C^{\rII}$.
It is precisely this structure that allows us to treat the full system using perturbation arguments.

We now establish the sectoriality of $\cA$ up to a shift.

\begin{prop}\label{prop:sect of the op matrix}
There exists $\omega \ge 0$ such that $\cA + \omega \in \cS(\rX_0)$ with spectral angle $\phi_{\cA + \omega} < \frac{\pi}{2}$.
\end{prop}

\begin{proof}
The proof is based on \autoref{lem:sect of diag dom op matrices}(a), and proceeds by verifying that $\cA$ is a diagonally dominant block operator matrix.

First, by \autoref{lem:sect of A_p} and \autoref{lem:sect of bA_q}, the operators $A_p$ and $\bA_f^{\rII}$ are closed, densely defined, and sectorial on their respective spaces.

Next, we observe that
\[
    \rD(\bA_f^{\rII}) \subset \rD(\B^{\rII}) \tand \rD(A_p) \subset \rD(\C^{\rII}).
\]
Moreover, for the fractional power domain $\rD(A_p^{\nicefrac{1}{2}}) = \rH_\per^4(\Gamma) \times \rH_\per^2(\Gamma)$, the mapping properties of $\C^{\rII}$ imply that
\[
    \C^{\rII} \in \cL\bigl(\rD(A_p^{\nicefrac{1}{2}}), \rX_{f,0}^{\rII}\bigr).
\]
On the other hand, for $(\qbr,\qbar,\bu) \in \rD(\bA_f^{\rII})$, we have
\begin{equation*}
    \left\| \B^{\rII} 
    \begin{pmatrix}
        \qbr\\ \qbar\\ \bu
    \end{pmatrix} \right\|_{\rX_0^p}
    \le C \left(\left\| \bA_f^{\rII} 
    \begin{pmatrix}
        \qbr\\ \qbar\\ \bu
    \end{pmatrix} \right\|_{\rX_{f,0}^{\rII}} +
    \left\|
    \begin{pmatrix}
        \qbr\\ \qbar\\ \bu
    \end{pmatrix} \right\|_{\rX_{f,0}^{\rII}}\right),
\end{equation*}
which shows that $\B^{\rII}$ is relatively bounded with respect to $\bA_f^{\rII}$.

Thus, all assumptions of \autoref{lem:sect of diag dom op matrices}(a) are satisfied. 
Since both diagonal blocks $A_p$ and $\bA_f^{\rII}$ are sectorial with spectral angle strictly less than $\frac{\pi}{2}$, we conclude that there exists $\omega \ge 0$ such that $\cA + \omega$ is sectorial on $\rX_0$ with spectral angle $\phi_{\cA+\omega} < \frac{\pi}{2}$.
\end{proof}

The above result shows that the full fluid--structure interaction system (currently with boundary condition $q^+ = 0$) can be realized as a sectorial operator (up to a shift) on the natural energy space $\rX_0$.
In particular, $\cA + \omega$ generates an analytic semigroup on $\rX_0$, which provides a robust framework for studying the well-posedness of the coupled problem.

From a structural point of view, this result demonstrates that the coupling between the viscoelastic plate and the fluid subsystem does not destroy the parabolic character of the underlying dynamics. 
Instead, the interaction terms enter as lower-order perturbations of two sectorial operators, allowing the full system to inherit their analytic properties.

As a consequence, one obtains existence, uniqueness, and continuous dependence of solutions on initial data within the semigroup framework, which will be discussed in \autoref{sec:WellPosed}. 
Moreover, the sectoriality of $\cA + \omega$ will play a crucial role in establishing further qualitative properties of the system, such as long-time behavior.

Indeed, we show next that the semigroup $\mre^{-t\cA}$ generated by $-\cA$ is exponentially stable, i.e., there exist $\mu_0 > 0$ and $C > 0$ such that $\| \mre^{-t \cA} \bv \|_{\rX_0} \le C \mre^{-\mu_0 t} \cdot \| \bv \|_{\rX_0}$ for all $\bv \in \rX_0$.

\medskip
\noindent{{\bfseries Spectral properties of the coupled fluid--structure operator $\cA$.}}
We now turn to the spectral analysis of the coupled operator $\cA$. 
Understanding the structure of the spectrum $\sigma(\cA)$ is essential for characterizing the qualitative behavior of solutions, in particular their stability and long-time dynamics.

A key difficulty in this analysis stems from the fact that in our problem the embedding
\[
    \rD(\cA) \hookrightarrow \rX_0
\]
is not compact. 

The lack of compactness comes from the average pressure variable $\qbar$  as well as the lack or smoothing of the plate displacement $w$ -- only the plate velocity $v$ experiences regularization by the addition of the viscoelastic term in plate elastodynamics. 
Indeed, while the plate {velocity}, the pressure jump $\qbr$, and the fluid velocity $\bu$ gain spatial regularity in the operator domain, the variable $\qbar$ remains only in $\rL^2(\Gamma)$, and the the plate displacement $w$ remains in $\rH_\per^4(\Gamma)$.

Thus, the embedding $\rD(\cA)\hookrightarrow \rX_0$ especially contains the identity embedding $\rL^2(\Gamma)\hookrightarrow \rL^2(\Gamma)$ in the $\qbar$-component, {{and the identity embedding $\rH_\per^4(\Gamma)\hookrightarrow \rH_\per^4(\Gamma)$,}} which are not compact.

Consequently, the resolvent of $\cA$ is not expected to be compact, and the spectrum need not consist solely of isolated eigenvalues of finite multiplicity -- we must therefore account for the possible presence of essential spectrum.

To deal with this difficulty, we adopt a strategy that separates the analysis of the point spectrum and the essential spectrum. 
This approach is inspired by the methodology developed in \cite[Prop.~5.17]{BMR:26} for a multilayered FSI problem with a damped thick structure. 
Adapting this framework to the present setting allows us to isolate the contributions of the coupling and to systematically analyze the spectral components of $\cA$.

\begin{prop}\label{prop:spectral analysis}
For the operator $\cA$, it holds that
\[
    \sigma(-\cA) \subset \{\lambda \in \C : \Rep \lambda < 0\},
\]
that is, the spectral bound
\begin{equation}\label{SB}
    s(-\cA) \coloneqq \sup\{\Rep \lambda : \lambda \in \sigma(-\cA)\}
\end{equation}
satisfies $s(-\cA) < 0$.
\end{prop}

\begin{proof}
The goal is to show that the spectrum of $-\cA$ lies strictly in the open left half-plane, which implies stability of the associated semigroup.

Since the resolvent of $\cA$ is not compact, the spectrum need not consist solely of eigenvalues. Therefore, we decompose the spectrum into its essential and point parts 
\[
    \sigma(-\cA) \subset \sigma_{\mathrm{ess}}(-\cA) \cup \sigma_{\mathrm{p}}(-\cA).
\]
We begin with the point spectrum.

\medskip

\noindent
\emph{\underline{Step 1}: Analysis of the point spectrum $\sigma_{\mathrm{p}}(-\cA)$.}
Let $\bv = (w,v,\qbr,\qbar,\bu)$ be an eigenvector corresponding to $\lambda \in \sigma_{\mathrm{p}}(-\cA)$, i.e.,
\[
    (\lambda + \cA)\bv = 0.
\]
This leads to the eigenvalue system
\begin{equation}\label{eq:eigenvalue probl}
    \left\{
    \begin{aligned}
        \lambda w - v 
        &= 0, &&\tin \Gamma,\\
        \lambda v + (\Delta_s^2 + \gamma^p) w + \eps_1 \Delta_s^2 v + \Bigl(\tfrac{1}{12}\Delta_s + 1\Bigr)\qbr
        &= 0, &&\tin \Gamma,\\
        \lambda \qbr - (\Delta_s + 12) v - \eps_2 \Delta_s \qbr + 48 \qbr + 72 \qbar + 12 \tr_\Gamma \bu \cdot \bn 
        &= 0, &&\tin \Gamma,\\
        \lambda \qbar + 6 \qbr + 12 \qbar 
        &= 0, &&\tin \Gamma,\\
        \lambda \bu - \mdiv \bsigma_f(\bu,\pi) 
        &= 0, &&\tin \Omega_f,
    \end{aligned}
    \right.
\end{equation}
supplemented with the coupling condition
\begin{equation*}
    \bsigma_f(\bu,\pi) \bn = \qbr \bn - \beta \sum_{i=1}^2 (\bu \cdot \btau_i) \btau_i, \ton \Gamma,
\end{equation*}
and the corresponding boundary conditions:
\begin{equation}
    \begin{aligned}
       w, v, \qbr, \qbar 
        \enspace &\text{periodic,} &&\ton \del\Gamma,\\
        \bu, \pi 
        \enspace &\text{periodic,} &&\ton \Gamma_f^l,\\
        \bu
        &= 0, &&\ton \Gamma_f^b.
    \end{aligned}
\end{equation}

To extract information about $\lambda$, we perform an energy estimate.
Testing the equations by the complex conjugates $v^*$, $\frac{1}{12}\qbr^*$, $\qbar^*$, and $\bu^*$, respectively, using the relation $v = \lambda w$, integrating by parts, and employing the boundary and coupling conditions, we obtain the identity
\begin{equation}\label{eq:tested eigenvalue probl}
    \begin{aligned}
        &\lambda \int_\Gamma |v|^2 \srd S + \lambda^* \int_\Gamma \bigl(|\Delta_s w|^2 + \gamma^p |w|^2\bigr) \srd S + \eps_1 \int_\Gamma |\Delta_s v|^2 \srd S \\
        &\quad + \tfrac{\lambda}{12} \int_\Gamma |\qbr|^2 \srd S + \tfrac{\eps_2}{12} \int_\Gamma |\nabla \qbr|^2 \srd S + \lambda \int_\Gamma |\qbar|^2 \srd S \\
        &\quad + \int_\Gamma |\qbr|^2 \srd S + 12 \int_\Gamma \Bigl|\qbar + \tfrac{1}{2}\qbr\Bigr|^2 \srd S + \lambda \int_{\Omega_f} |\bu|^2 \srd \bx + \int_{\Omega_f} |\bD(\bu)|^2 \srd \bx \\
        &\quad + \beta \int_\Gamma |\bu \cdot \btau|^2 \srd S = 0.
    \end{aligned}
\end{equation}

Taking real parts in \eqref{eq:tested eigenvalue probl}, and observing that all terms except those involving $\lambda$ are nonnegative, we deduce
\[
    \Rep \lambda \le 0.
\]
Thus,
\[
    \sigma_{\mathrm{p}}(-\cA) \subset \{\lambda \in \C : \Rep \lambda \le 0\}.
\]

\medskip

\noindent
\emph{\underline{Step 2}: Strict negativity of the real part.}
We now show that no eigenvalue satisfies $\Rep \lambda = 0$. 
Assume, for contradiction, that $\lambda \in \sigma_{\mathrm{p}}(-\cA)$ with $\Rep \lambda = 0$.

Taking real parts in \eqref{eq:tested eigenvalue probl}, we observe that all terms not multiplied by $\lambda$ are nonnegative. 
Hence, the identity implies that all dissipative contributions must vanish. In particular, we obtain
\[
    \Delta_s v = 0, \quad \nabla \qbr = 0, \quad \qbr = 0, \quad \qbar + \tfrac{1}{2}\qbr = 0, \quad \bu \cdot \btau = 0 \ton \Gamma,
\]
and
\[
    \bD(\bu) = 0 \tin \Omega_f.
\]
By the boundary conditions, this implies
\[
    \bu = 0 \tin \Omega_f, \qquad \qbar = 0 \ton \Gamma.
\]

Substituting these relations into the third equation of \eqref{eq:eigenvalue probl}, we deduce that $v = 0$. 
Using the first equation, this yields $\lambda w = 0$, and hence $v = 0$ implies either $\lambda = 0$ or $w = 0$.

We now treat the case $\lambda = 0$. 
In this case, testing the second equation in \eqref{eq:eigenvalue probl} by $w^*$ and integrating by parts, we obtain
\[
    \int_\Gamma |\Delta_s w|^2 \, dS + \gamma^p \int_\Gamma |w|^2 \, dS = 0,
\]
which implies $w = 0$. 
Consequently, all components of $\bv = (v,\qbr,\qbar,\bu)$ vanish, and thus $\bv = 0$.

Therefore, there are no nontrivial eigenvectors corresponding to $\Rep \lambda = 0$, and we conclude
\[
    \sigma_{\mathrm{p}}(-\cA) \subset \{\lambda \in \C : \Rep \lambda < 0\}.
\]

\medskip

\noindent
\emph{\underline{Step 3}: Analysis of the essential spectrum.}
It remains to investigate the essential spectrum $\sigma_{\mathrm{ess}}(-\cA)$. 
To this end, we exploit the block structure
\[
    \cA = 
    \begin{pmatrix}
        A_p & \B^{\rII}\\
        \C^{\rII} & \bA_f^{\rII}
    \end{pmatrix},
\]
where $\C^{\rII}$ and $\B^{\rII}$ are defined in \eqref{BandC}.

We first observe that, due to the Kelvin--Voigt damping in the plate component and the Rellich--Kondrachov theorem, the operator $\C^{\rII}$ is a relatively compact perturbation. 
Consequently, it does not affect the essential spectrum (see \cite[Thm.~5.35]{Kat:95}), and therefore
\[
    \sigma_{\mathrm{ess}}(-\cA) 
    = \sigma_{\mathrm{ess}}\!\left(
    -\begin{pmatrix}
        A_p & \B^{\rII}\\
        0 & \bA_f^{\rII}
    \end{pmatrix}
    \right).
\]

Next, we note the factorization
\[
    \lambda - 
    \begin{pmatrix}
        A_p & \B^{\rII}\\
        0 & \bA_f^{\rII}
    \end{pmatrix}
    =
    \begin{pmatrix}
        \lambda - A_p & 0\\
        0 & \lambda - \bA_f^{\rII}
    \end{pmatrix}
    \begin{pmatrix}
        \Id & -R(\lambda,A_p)\B^{\rII}\\
        0 & \Id
    \end{pmatrix},
\]
which implies
\[
    \sigma\!\left(
    \begin{pmatrix}
        A_p & \B^{\rII}\\
        0 & \bA_f^{\rII}
    \end{pmatrix}
    \right)
    = \sigma\bigl(\diag(A_p,\bA_f^{\rII})\bigr).
\]
In particular, we note here that $\B^{\rII}$ from \eqref{BandC} satisfies $\B^{\rII} \colon \rD(\bA_f^{\rII}) \to \rX_0^p$, since the $\qbr$-component in~$\rD(\bA_f^{\rII})$ especially enjoys $\rH^2$-regularity required for the Laplacian in $\B^{\rII}$.
As a result, the factorization is well-defined.
Hence,
\[
    \sigma_{\mathrm{ess}}(-\cA) \subset \sigma\bigl(-\diag(A_p,\bA_f^{\rII})\bigr).
\]
By \autoref{lem:sect of A_p}, we already know that $s(-A_p) < 0$. 
It therefore remains to analyze the spectrum of $\bA_f^{\rII}$.

\medskip

To simplify the structure, we invoke the similarity transform introduced earlier and consider the equivalent operator
\[
    \widetilde{\bA_f^{\rII}} 
    = \begin{pmatrix}
        12 & 6 & 0\\
        72 & -\eps_2 \Delta_s + 48 & 12 \tr_\Gamma(\cdot)\bn\\
        0 & 0 & -\bP \Delta
    \end{pmatrix}.
\]
Similarity transforms preserve the spectrum, so it suffices to study $\widetilde{\bA_f^{\rII}}$.

Proceeding as in Step~1, an energy estimate for the associated eigenvalue problem shows that
\[
    \sigma_{\mathrm{p}}(-\widetilde{\bA_f^{\rII}}) \subset \{\lambda \in \C : \Rep \lambda < 0\}.
\]

To handle the essential spectrum, we introduce the auxiliary operator
\[
    \widehat{\bA_f^{\rI}}
    = \bA_f^{\rI} + \begin{pmatrix}
        0 & 12 \tr_\Gamma(\cdot)\bn\\
        0 & 0
    \end{pmatrix},
\]
whose domain coincides with that of $\bA_f^{\rI}$. 
By elliptic regularity and the compact embedding of $\rD(\bA_f^{\rI})$ into $\rX_{f,0}^{\rI}$, this operator has compact resolvent.

Applying again an energy argument to the corresponding eigenvalue problem yields
\[
    \sigma(-\widehat{\bA_f^{\rI}}) \subset \{\lambda \in \C : \Rep \lambda < 0\}.
\]

Since the remaining coupling terms (such as $6\,\Id$) are relatively compact perturbations, we conclude that they do not affect the essential spectrum. 
Consequently,
\[
    \sigma_{\mathrm{ess}}(-\widetilde{\bA_f^{\rII}}) \subset \sigma\bigl(-\diag(12\,\Id,\widehat{\bA_f^{\rI}})\bigr) \subset \{\lambda \in \C : \Rep \lambda < 0\}.
\]

Combining these results, we obtain
\[
    \sigma(-\bA_f^{\rII}) = \sigma(-\widetilde{\bA_f^{\rII}}) \subset \{\lambda \in \C : \Rep \lambda < 0\}.
\]

Finally, recalling that
\[
    \sigma_{\mathrm{ess}}(-\cA) \subset \sigma\bigl(-\diag(A_p,\bA_f^{\rII})\bigr),
\]
and using $s(-A_p) < 0$, we conclude
\[
    \sigma_{\mathrm{ess}}(-\cA) \subset \{\lambda \in \C : \Rep \lambda < 0\}.
\]

\medskip

Combining this with the result from Steps~1--2 for the point spectrum, we arrive at
\[
    \sigma(-\cA) = \sigma_{\mathrm{ess}}(-\cA)\cup\sigma_{\mathrm{p}}(-\cA) \subset \{\lambda \in \C : \Rep \lambda < 0\},
\]
which completes the proof of the proposition.
\end{proof}

The following corollary establishes the sectoriality of the operator matrix $\cA$ without shift and, as a consequence, the exponential stability of the associated semigroup. 
The corollary is a consequence of \autoref{prop:sect of the op matrix} combined with the spectral analysis carried out in \autoref{prop:spectral analysis}.

\begin{cor} [{\bf{Sectoriality and exponential stability of $\cA$}}]\label{cor:exp decay}
For the operator $\cA$ defined in \eqref{eq:op FPSI problem}, it holds that $\cA \in \cS(\rX_0)$ with spectral angle $\phi_{\cA} < \frac{\pi}{2}$. 
Moreover, the semigroup $(\mre^{-t\cA})_{t \ge 0}$ generated by $-\cA$ is exponentially stable. 
That is, there exist constants $\mu_0 > 0$ and $C > 0$ such that
\[
    \| \mre^{-t \cA} \bv \|_{\rX_0} \le C \mre^{-\mu_0 t} \| \bv \|_{\rX_0} \tforall \bv \in \rX_0 \tand t \ge 0.
\]
\end{cor}

\subsection{Application to the Cauchy problem  with inhomogeneous condition $q^+$ and main result}\label{sec:WellPosed}
\

In this section, we establish the well-posedness of the Cauchy problem
\eqref{eq:FPSI Cauchy problem}--\eqref{init2}  with inhomogeneous boundary condition $q^+$ and derive the main result.

A key difficulty arises from the presence of the inhomogeneous boundary term $q^+ \bn$ in the coupling condition \eqref{NewCoupling}. 
This term prevents a direct application of the operator-theoretic framework developed in above, which relies on the homogeneous condition $q^+ = 0$.

To overcome this obstacle, we decompose the solution into two parts: one that accounts for the inhomogeneous boundary data, and a remainder that satisfies the coupling condition, but with homogeneous boundary data. 

\medskip

\noindent
\textbf{Decomposition of the solution.}
More precisely, we seek a solution of the form
\[
    (w,v,\qbr,\qbar,\bu) = (w,v,\qbr,\qbar,\widetilde{\bu}) + (0,0,0,0,\bu_*),
\]
where $\widetilde{\bu}$ corresponds to the solution with $q^+ = 0$, and the function $\bu_*$ is chosen to absorb the inhomogeneous boundary contribution $q^+ \ne 0$.

\medskip

{\bfseries The problem with $\bu_*$.}
We define $\bu_*$ to be the solution to the Stokes-type problem incorporating the inhomogeneous boundary condition $q^+ \bn$ as follows:
\begin{equation}\label{eq:lifting probl}
    \left\{
    \begin{aligned}
        \del_t \bu_* - \nabla \cdot \bsigma_f(\bu_*,\pi_*) 
        &= 0, && \tin \R_+ \times \Omega_f,\\
        \nabla \cdot \bu_* 
        &= 0, && \tin \R_+ \times \Omega_f,\\
        \bsigma_f(\bu_*,\pi_*) \bn 
        &= -q^+ \bn - \beta \sum_{i=1}^2 (\bu_* \cdot \btau_i)\btau_i, && \ton \R_+ \times \Gamma,\\
        \bu_* 
        &= 0, && \ton \R_+ \times \Gamma_f^b,\\
        \bu_*(0) 
        &= 0, && \tin \Omega_f,
    \end{aligned}
    \right.
\end{equation}
subject to the periodic boundary conditions \eqref{NewBC}.

\medskip
{\bfseries The problem with $\widetilde{\bu}$.}
A straightforward computation shows that if $(w,v,\qbr,\qbar,\bu)$ solves the original Cauchy problem \eqref{eq:FPSI Cauchy problem}--\eqref{init2}, then the remaining part $(w,v,\qbr,\qbar,\widetilde{\bu})$ of the solution satisfies 
\begin{equation}\label{eq:lifted probl}
    \left\{
    \begin{aligned}
        \del_t w 
        &= v, && \tin \R_+ \times \Gamma,\\
        \del_t v + \Delta_{\btau}^2 w + \gamma^p w + \eps_1\Delta_{\btau}^2 v + \tfrac{1}{12} \Delta_{\btau} \qbr + \qbr 
        &= 0, && \tin \R_+ \times \Gamma,\\
        \del_t \qbr - \Delta_{\btau} v - 12 v - \eps_2\Delta_{\btau} \qbr + 48 \qbr + 72 \qbar + 12 \widetilde{\bu}|_{\Gamma}\cdot \bn 
        &= 72 q^+ - 12 \bu_*|_{\Gamma}\cdot \bn, && \tin \R_+ \times \Gamma,\\
        \del_t \qbar + 6 \qbr + 12 \qbar 
        &= 12 q^+, && \tin \R_+ \times \Gamma,\\
        \del_t \widetilde{\bu} - \nabla \cdot \bsigma_f(\widetilde{\bu},\widetilde{\pi}) 
        &= 0, && \tin \R_+ \times \Omega_f,\\
        \nabla \cdot \widetilde{\bu} 
        &= 0, && \tin \R_+ \times \Omega_f,
\end{aligned}
\right.
\end{equation}
subject to the same periodic boundary conditions \eqref{NewBC}.

The interface condition now takes the homogeneous form:
\begin{equation}\label{eq:hom bc}
    \begin{aligned}
        \bsigma_f(\Tilde{\bu},\Tilde{\pi}) \bn 
        &= \bsigma_f(\bu,\pi) \bn - \bsigma_f(\bu_*,\pi_*) \bn\\
        &=  (\qbr - q^+) \bn - \beta \sum_{i=1}^2 (\bu \cdot \btau_i) \btau_i - \left(-q^+ \bn - \beta \sum_{i=1}^2 (\bu_* \cdot \btau_i) \btau_i\right)\\
        &= \qbr \bn - \beta \sum_{i=1}^2 (\Tilde{\bu} \cdot \btau_i) \btau_i.
    \end{aligned}
\end{equation}
This problem for $\tilde\bu$ can now be analyzed using the theory developed above for problems with homogeneous boundary conditions. 

\medskip
\noindent{{\bfseries Analysis of the lifting problem \eqref{eq:lifting probl} for $\bu^*$.}}
What is left is to analyze the problem \eqref{eq:lifting probl} for $\bu^*$.
For this purpose, we consider the following general Stokes problem with data $\bff$, $\bg$, and $\bu_0$, which will be specified later in \autoref{lem:opt data for lifting}:
\begin{equation}\label{eq:inhom Stokes probl}
    \left\{
    \begin{aligned}
        \del_t \bu - \nabla \cdot \bsigma(\bu,\pi)
        &= \bff, &&\tin \R_+ \times \Omega_f,\\
        \nabla \cdot \bu 
        &= 0, &&\tin \R_+ \times \Omega_f,\\
        \bsigma_f(\bu,\pi) \bn + \beta \sum_{i=1}^2 (\bu \cdot \btau_i) \btau_i
        &= \bg, &&\ton \R_+ \times \Gamma,\\
        \bu
        &= 0, &&\ton \R_+ \times \Gamma_f^b,\\
        \bu(0)
        &= \bu_0, &&\tin \Omega_f,
    \end{aligned}
    \right.
\end{equation}
where $\bu$ as well as $\pi$ additionally satisfy the periodic boundary conditions as made precise in \eqref{NewBC}.

The lemma below discusses the solvability of \eqref{eq:inhom Stokes probl}.
Before we state the lemma, we introduce the space of functions in $\rL_{\mathrm{loc}}^1(\Omega_f)$ whose gradient is in $\rL^2(\Omega_f)$, i.e.,
\begin{equation*}
    \Dot{\rH}^1(\Omega_f) \coloneqq \{ f \in \rL_\mathrm{loc}^1(\Omega_f) \ : \ \nabla f \in \rH^1(\Omega_f)\},
\end{equation*}
and the space of solenoidal vector fields in $\rH^1$, and with zero trace on the bottom boundary $\Gamma_f^b$, i.e.,
\begin{equation}\label{eq:init data fluid vel}
    \rX_{f,\sigma} \coloneqq \left\{ f \in \rH_{\per}^1(\Omega_f)^3 \ : \ \nabla \cdot f = 0 \tand \left. f \right|_{\Gamma_f^b} = 0\right\}.
\end{equation}
We note that the subscript ``$f$'' is used to designate function spaces corresponding to the fluid variables.

\begin{lem}\label{lem:opt data for lifting}
Consider the inhomogeneous Stokes problem \eqref{eq:inhom Stokes probl}. 
Then this problem admits a unique solution $(\bu,\pi)$ with optimal regularity
\begin{equation*}
    \begin{aligned}
        \bu 
        &\in \rH^1(\R_+;\rL^2(\Omega_f)^3) \cap \rL^2(\R_+;\rH_{\per}^2(\Omega_f)^3) \eqqcolon \E_{f,\bu},\\
        \pi 
        &\in \rL^2(\R_+;\Dot{\rH}_{\per}^1(\Omega_f)) \eqqcolon \E_{f,\pi}, \qquad \pi|_\Gamma \in \rH^{\frac{1}{4}}(\R_+;\rL^2(\Gamma)),
    \end{aligned}
\end{equation*}
if and only if the data $(\bff,\bg,\bu_0)$ satisfy
\begin{equation*}
    \begin{aligned}
        \bff 
        &\in \rL^2(\R_+;\rL^2(\Omega_f)^3) \eqqcolon \F_f,\\
        \bg 
        &\in \rH^{\frac{1}{4}}(\R_+;\rL^2(\Gamma)^3) \cap \rL^2(\R_+;\rH_{\per}^{\frac{1}{2}}(\Gamma)^3) \eqqcolon \G_f,\\
        \bu_0
        &\in  \rX_{f,\sigma}.
    \end{aligned}
\end{equation*}

Moreover, the solution depends continuously on the data, and there exists a constant $C>0$ such that
\begin{equation*}
    \|\bu\|_{\E_{f,\bu}} + \|\pi\|_{\E_{f,\pi}} \le C \Bigl( \|\bff\|_{\F_f} + \|\bg\|_{\G_f} + \|\bu_0\|_{\rX_{f,\sigma}} \Bigr).
\end{equation*}
\end{lem}

\begin{proof}
The main point of the proof is to extend the classical optimal data result for the Stokes system defined on the half-space with Neumann boundary condition $\bsigma_f(\bu,\pi)\bn = \bg$ to the present setting, which is a Stokes problem defined on a cylindrical domain with an additional lower-order boundary term $\beta \sum_{i=1}^2 (\bu \cdot \btau_i)\btau_i$. 

\medskip

\noindent
\emph{\underline{Step 1}: Reduction to the classical Neumann problem.}
As in the proof of \autoref{lem:sect of op with hom bc}, one first reduces the problem posed on the cylindrical domain with horizontal periodicity to the corresponding problem on the half-space with the boundary condition $\bsigma_f(\bu,\pi)\bn = \bg$. 
This reduction is standard and relies on localization and extension arguments.

In the half-space setting, it is known (see \cite[Sec.~7.2]{PS:16}) that the Stokes problem with Neumann-type boundary condition
\[
    \bsigma_f(\bu,\pi)\bn = \bg
\]
admits the optimal data property. 
In operator-theoretic terms, this means that the map
\[
    \widetilde{\Psi_f} \colon \E_{f,\bu} \to \F_f \times \G_f, \qquad \widetilde{\Psi_f}(\bu) = \bigl(\del_t \bu + \omega \bu + A_{\max}\bu,\ \bsigma_f(\bu,\pi)\bn\bigr),
\]
is an isomorphism for a suitable shift $\omega \ge 0$.

\medskip

\noindent
\emph{\underline{Step 2}: Incorporation of the lower-order boundary term.}
In the present problem, the boundary condition contains an additional lower-order term
\[
    \cB_2(\bu) = \beta \sum_{i=1}^2 (\bu \cdot \btau_i)\btau_i.
\]
Accordingly, we define the full boundary operator
\[
    \Psi_f(\bu) = \bigl(\del_t \bu + \omega \bu + A_{\max}\bu,\ \cB_1(\bu,\pi) + \cB_2(\bu)\bigr), \twhere \cB_1(\bu,\pi) = \bsigma_f(\bu,\pi)\bn.
\]

The key observation is that $\cB_2$ is of lower order compared to $\cB_1$. 
More precisely, following \cite[Sec.~7.3.1]{PS:16}, one shows that for every $\eta > 0$ there exist constants $C>0$ and $\gamma \in (0,1]$ such that
\[
    \| \cB_2(\bu) \|_{\G_f} \le \eta \|\bu\|_{\E_{f,\bu}} + C \|\bu\|_{\E_{f,\bu}}^\gamma \|\bu\|_{\F_f}^{1-\gamma}.
\]
This estimate shows that $\cB_2$ is a relatively bounded perturbation of $\cB_1$ with arbitrarily small relative bound.

\medskip

\noindent
\emph{\underline{Step 3}: Perturbation argument.}
Since $\widetilde{\Psi_f}$ is an isomorphism, the above estimate allows us to treat $\cB_2$ as a small perturbation. 
By choosing $\eta$ sufficiently small, a Neumann series argument implies that the full operator $\Psi_f$ remains an isomorphism, at least after introducing a shift $\omega \ge 0$ sufficiently large.

\medskip

\noindent
\emph{\underline{Step 4}: Removal of the shift.}
Finally, as in \cite[Sec.~6.3.4]{PS:16}, the shift $\omega$ can be chosen larger than the spectral bound of the Stokes operator associated with the highest-order part of the boundary condition. 
In the present setting, due to the homogeneous Dirichlet boundary condition on $\Gamma_f^b$, the Stokes operator has strictly negative spectral bound. 
Consequently, no shift is needed, and one may take $\omega = 0$.

\medskip

This proves that the operator $\Psi_f$ is an isomorphism, which is equivalent to the optimal data property stated in the lemma. 
The assertion follows.
\end{proof}

In the following, we apply \autoref{lem:opt data for lifting} to the lifting problem \eqref{eq:lifting probl} in order to obtain precise regularity and stability properties of the lifting function $(\bu_*,\pi_*)$. 
In particular, we show that the lifting inherits the temporal behavior of the boundary data $q^+$ and, moreover, exhibits exponential decay provided the data decays at an appropriate rate.

The key idea is to interpret the lifting problem as an inhomogeneous Stokes system with boundary data given by $q^+$. 
The optimal data result from \autoref{lem:opt data for lifting} then allows us to transfer regularity and decay properties of the boundary data directly to the solution. 
The admissible decay rate is determined by the spectral bound $s(A_0)$ of the Stokes operator associated with homogeneous boundary conditions.

\begin{prop} [{\bf{Well-posedness and exponential decay for the lifting problem}}] \label{prop:concrete lifting}
Let $\mu_1$ denote the spectral bound $\mu_1 \coloneqq -s(A_0) > 0$, where $A_0$ denotes the Stokes operator introduced in \eqref{HomoStokes}, corresponding to homogeneous Neumann boundary conditions on $\Gamma$ and homogeneous Dirichlet boundary conditions on $\Gamma_f^b$. 

Let $\mu \in [0,\mu_1)$, and assume that the boundary data satisfies
\[
    \mre^{\mu t} q^+ \in \rH^{\frac{1}{4}}(\R_+;\rL^2(\Gamma)) \cap \rL^2(\R_+;\rH_{\per}^{\frac{1}{2}}(\Gamma)) = \G_f.
\]
Then the lifting problem \eqref{eq:lifting probl} admits a unique solution $(\bu_*,\pi_*)$ such that
\[
    \mre^{\mu t}\bu_* \in \rH^1(\R_+;\rL^2(\Omega_f)^3) \cap \rL^2(\R_+;\rH_{\per}^2(\Omega_f)^3) = \E_{f,\bu},
\]
and
\[
    \mre^{\mu t}\pi_* \in \rL^2(\R_+;\Dot{\rH}_{\per}^1(\Omega_f)) = \E_{f,\pi}, \qquad \pi_*|_\Gamma \in \rH^{\frac{1}{4}}(\R_+;\rL^2(\Gamma)).
\]

Moreover, there exists a constant $C>0$, independent of $q^+$, such that
\[
    \| \mre^{\mu t}\bu_* \|_{\E_{f,\bu}} + \| \mre^{\mu t}\pi_* \|_{\E_{f,\pi}} \le C \, \| \mre^{\mu t} q^+ \|_{\G_f}.
\]
\end{prop}

\begin{proof}
We begin with the case $\mu = 0$. 
In this case, the lifting problem \eqref{eq:lifting probl} is an instance of the inhomogeneous Stokes problem \eqref{eq:inhom Stokes probl} with data
\[
    \bff = 0, \enspace \bu_0 = 0, \tand \bg = q^+ \bn.
\]
Since $q^+$ satisfies the regularity assumptions of the lemma, it follows that $\bg \in \G_f$. 
Thus, \autoref{lem:opt data for lifting} applies directly and yields the existence, uniqueness, and the stated regularity of the solution $(\bu_*,\pi_*)$, together with the corresponding estimate.

\medskip

We now turn to the case $\mu \in (0,\mu_1)$. 
To capture exponential decay, we consider the exponentially weighted variables
\[
    \widetilde{\bu}_*(t) \coloneqq \mre^{\mu t}\bu_*(t), \qquad \widetilde{\pi}_*(t) \coloneqq \mre^{\mu t}\pi_*(t).
\]
A straightforward computation shows that $(\widetilde{\bu}_*,\widetilde{\pi}_*)$ satisfies a Stokes-type system of the same form as~\eqref{eq:lifting probl}, but with the modified operator
\[
    \del_t \widetilde{\bu}_* + (A_{\max} + \mu)\widetilde{\bu}_*,
\]
and boundary data $\mre^{\mu t} q^+ \bn$.

By the discussion in the proof of \autoref{lem:opt data for lifting}, the optimal data property remains valid for the shifted operator $A_{\max} + \mu$, provided $\mu$ is smaller than the spectral gap of the Stokes operator. 
More precisely, since $\mu_1 = -s(A_0) > 0$, the operator $A_{\max} + \mu$ retains the same mapping properties for all $\mu \in [0,\mu_1)$.

Therefore, an application of \autoref{lem:opt data for lifting} to the weighted problem yields
\[
    \|\widetilde{\bu}_*\|_{\E_{f,\bu}} + \|\widetilde{\pi}_*\|_{\E_{f,\pi}} \le C \|\mre^{\mu t} q^+\|_{\G_f},
\]
which is exactly the desired estimate for $(\bu_*,\pi_*)$.

\medskip

This proves the assertion of the lemma.
\end{proof}

\subsection{Main result: global-in-time strong well-posedness of the Cauchy problem.}
\ 

We are now in a position to state and prove the main result of this section, which establishes global-in-time strong well-posedness of the Cauchy problem \eqref{eq:FPSI Cauchy problem}--\eqref{init2} with inhomogeneous boundary data $q^+\bn$. 
In addition, we show that the solution inherits exponential decay from the boundary data $q^+$, provided the latter decays at an admissible rate.

\begin{thm}[{\bf{Global strong well-posedness and exponential decay}}]\label{thm:solvability Cauchy problem}
Let 
\[
    w_0 \in \rH_\per^4(\Gamma), \quad v_0 \in \rH_\per^2(\Gamma), \quad \qbr_0 \in \rH_\per^1(\Gamma), \quad \qbar_0 \in \rL^2(\Gamma),
\]
and $\bu_0 \in \rX_{f,\sigma}$, where $\rX_{f,\sigma}$ has been introduced in \eqref{eq:init data fluid vel}.

Let $\mu_* \coloneqq \min\{\mu_0,\mu_1\}$, where $\mu_0$ is given in \autoref{cor:exp decay} and $\mu_1 = -s(A_0)$ as in \autoref{prop:concrete lifting}. 
For $\mu \in [0,\mu_*)$, assume that the boundary data satisfies
\[
    \mre^{\mu t} q^+ \in \rH^{\frac{1}{4}}(\R_+;\rL^2(\Gamma)) \cap \rL^2(\R_+;\rH_{\per}^{\frac{1}{2}}(\Gamma)) = \G_f.
\]

Then the Cauchy problem \eqref{eq:FPSI Cauchy problem}--\eqref{init2} admits a unique global solution 
\[
    (w,v,\qbr,\qbar,\bu) \ton [0,\infty),
\]
satisfying
\begin{equation*}
    \begin{aligned}
        \mre^{\mu t} w 
        &\in \rH^1(0,\infty;\rH_\per^4(\Gamma)) \eqqcolon \E_w,\\
        \mre^{\mu t} v 
        &\in \rH^1(0,\infty;\rL^2(\Gamma)) \cap \rL^2(0,\infty;\rH_\per^4(\Gamma)) \eqqcolon \E_v,\\
        \mre^{\mu t} \qbr 
        &\in \rH^1(0,\infty;\rL^2(\Gamma)) \cap \rL^2(0,\infty;\rH_\per^2(\Gamma)) \eqqcolon \E_{\qbr},\\
        \mre^{\mu t} \qbar 
        &\in \rH^1(0,\infty;\rL^2(\Gamma)) \eqqcolon \E_{\qbar},\\
        \mre^{\mu t} \bu 
        &\in \rH^1(0,\infty;\rL_\sigma^2(\Omega_f)) \cap \rL^2(0,\infty;\rH_\per^2(\Omega_f)^3 \cap \rL_\sigma^2(\Omega_f)) = \E_{f,\bu}.
    \end{aligned}
\end{equation*}

Moreover, the solution is continuous in time,
\begin{equation*}
    (w,v,\qbr,\qbar,\bu) \in \rC\bigl([0,\infty); \rH_\per^4(\Gamma) \times \rH_\per^2(\Gamma) \times \rH_\per^1(\Gamma) \times \rL^2(\Gamma) \times \rX_{f,\sigma}\bigr),
\end{equation*}
and satisfies the estimate
\begin{equation*}
    \begin{aligned}
        &\quad \|\mre^{\mu (\cdot)} w\|_{\E_w} + \|\mre^{\mu (\cdot)} v\|_{\E_v} + \|\mre^{\mu (\cdot)} \qbr\|_{\E_{\qbr}} + \|\mre^{\mu (\cdot)} \qbar\|_{\E_{\qbar}} + \|\mre^{\mu (\cdot)} \bu\|_{\E_{f,\bu}}\\
        &\le C \Bigl(\|\mre^{\mu (\cdot)} q^+\|_{\G_f} + \|w_0\|_{\rH_\per^4(\Gamma)} + \|v_0\|_{\rH_\per^2(\Gamma)} + \|\qbr_0\|_{\rH_\per^1(\Gamma)} + \|\qbar_0\|_{\rL^2(\Gamma)} + \|\bu_0\|_{\rX_{f,\sigma}}\Bigr).
    \end{aligned}
\end{equation*}
\end{thm}

\begin{proof}
We first treat the case $\mu = 0$, and then extend the argument to $\mu > 0$.

\medskip

\noindent
\emph{\underline{Step 1}: Reduction via lifting.}
We begin by considering the lifting problem \eqref{eq:lifting probl}. 
By \autoref{prop:concrete lifting}, there exists a unique solution $(\bu_*,\pi_*) \in \E_{f,\bu} \times \E_{f,\pi}$ on $(0,\infty)$. 

Defining $\Tilde{\bu} \coloneqq \bu - \bu_*$, we reduce the original problem \eqref{eq:FPSI Cauchy problem}--\eqref{init2} to the lifted problem \eqref{eq:lifted probl}, subject to the boundary conditions \eqref{NewBC} and the homogeneous interface condition \eqref{eq:hom bc}. 

Introducing the principal variable
\[
    \ba = (w,v,\qbr,\qbar,\Tilde{\bu}),
\]
we can rewrite the system in abstract form as
\[
    \ba'(t) + \cA \ba(t) = \cF(t), \tfor t \in (0,\infty), \qquad \ba(0) = \ba_0,
\]
where
\[
    \cF(t) 
    = \begin{pmatrix}
        0\\
        0\\
        72 q^+ - 12 \left.\bu_* \right|_{\Gamma} \cdot \bn\\
        12 q^+\\
        0
    \end{pmatrix}.
\]

\medskip

\noindent
\emph{\underline{Step 2}: Regularity of the right-hand side.}
To apply maximal regularity for $\cA$, we must verify that $\cF \in \rL^2(\R_+;\rX_0)$, where we recall the ground space $\rX_0$ from  \eqref{eq:ground space}.

First, the assumptions on $q^+$ imply
\[
    q^+ \in \rL^2(\R_+;\rL^2(\Gamma)).
\]

Next, we estimate the trace term involving $\bu_*$. 
From \autoref{prop:concrete lifting}, we have
\[
    \bu_* \in \rL^2(\R_+;\rH_{\per}^2(\Omega_f)^3).
\]
By continuity of the trace operator, we get
\[
    \left.\bu_* \right|_{\Gamma} \in \rL^2(\R_+;\rH^{\frac{3}{2}}(\Gamma)^3).
\]

Moreover, arguing as in \cite[Prop.~6.4]{DHP:07}, we obtain the time regularity
\[
    \left.\bu_* \right|_{\Gamma} \in \rH^{\frac{3}{4}}(\R_+;\rL^2(\Gamma)^3).
\]

Thus,
\[
    \left.\bu_* \right|_{\Gamma} \in \rH^{\frac{3}{4}}(\R_+;\rL^2(\Gamma)^3) \cap \rL^2(\R_+;\rH^{\frac{3}{2}}(\Gamma)^3) \hookrightarrow \rL^2(\R_+;\rL^2(\Gamma)^3).
\]

In particular,
\[
    \|12 \left.\bu_* \right|_{\Gamma} \cdot \bn\|_{\rL^2(\R_+;\rL^2(\Gamma))}\le C \|\bu_*\|_{\E_{f,\bu}}.
\]
Hence, $\cF \in \rL^2(\R_+;\rX_0)$.

\medskip

\noindent
\emph{\underline{Step 3}: Maximal regularity and solution construction.}
By the sectoriality of $\cA$, see \autoref{cor:exp decay}, and the resulting maximal regularity in the present Hilbert space case, see \autoref{lem:sect & max reg on Hilbert spaces}, the abstract Cauchy problem admits a unique solution $\ba$ with the asserted regularity.

More precisely, the regularity of $(w,v,\qbr,\qbar)$ follows directly from maximal regularity. 
For the fluid velocity, we use the decomposition
\[
    \bu = \Tilde{\bu} + \bu_*,
\]
and combine the regularity of $\Tilde{\bu}$ with that of $\bu_*$ from \autoref{prop:concrete lifting}.

\medskip

\noindent
\emph{\underline{Step 4}: Continuity in time.}
The continuity of the solution in time with values in the trace space follows from the standard embedding of maximal regularity spaces into the corresponding trace space, which coincides with the space of initial data; see \cite[Thm.~III.4.5.10]{Ama:95}.

\medskip

\noindent
\emph{\underline{Step 5}: A priori estimate.}
Using the decomposition $\bu = \Tilde{\bu} + \bu_*$, we estimate
\[
    \|\bu\|_{\E_{f,\bu}} \le \|\Tilde{\bu}\|_{\E_{f,\bu}} + \|\bu_*\|_{\E_{f,\bu}}.
\]
All components except $\bu_*$ are controlled by applying maximal regularity to the abstract problem. Together with the estimate for $\cF$, this yields the estimate
\begin{equation} \label{EXinit}
    \|\ba\|_{\E} \le C\bigl(\|q^+\|_{\G_f} + \|\ba_0\|_{\rX_{\mathrm{init}}} + \|\bu_*\|_{\E_{f,\bu}}\bigr),
\end{equation} 
where the spaces are defined as
\begin{equation*}
    \E \coloneqq \E_w \times \E_v \times \E_{\qbr} \times \E_{\qbar} \times \E_{f,\bu} \quad \text{and} \quad \rX_{\mathrm{init}} \coloneqq \rH_\per^4(\Gamma) \times \rH_\per^2(\Gamma) \times \rH_\per^1(\Gamma) \times \rL^2(\Gamma) \times \rX_{f,\sigma}.
\end{equation*}
The norm on the left-hand side of \eqref{EXinit} is typically referred to as the maximal regularity norm.
Written in detail, this norm and the estimate read:
\begin{equation*}
    \begin{aligned}
        &\quad \| w \|_{\E_w} + \| v \|_{\E_v} + \| \qbr \|_{\E_{\qbr}} + \| \qbar \|_{\E_{\qbar}} + \| \Tilde{\bu}\|_{\E_{f,\bu}}\\
        &\le C\bigl(\| q^+ \|_{\G_f} + \| w_0 \|_{\rH_\per^4(\Gamma)} + \| v_0 \|_{\rH_\per^2(\Gamma)} + \| \qbr_0 \|_{\rH_\per^1(\Gamma)} + \| \qbar_0 \|_{\rL^2(\Gamma)} + \| \bu_0 \|_{\rX_{f,\sigma}} + \| \bu_*\|_{\E_{f,\bu}}\bigr).
    \end{aligned}
\end{equation*}

To bound $\bu_*$, we invoke \autoref{prop:concrete lifting}, which gives
\[
    \|\bu_*\|_{\E_{f,\bu}} \le C \|q^+\|_{\G_f}.
\]

Combining these estimates yields the desired bound:
\begin{equation*}
    \begin{aligned}
        &\quad \| w \|_{\E_w} + \| v \|_{\E_v} + \| \qbr \|_{\E_{\qbr}} + \| \qbar \|_{\E_{\qbar}} + \| \bu \|_{\E_{f,\bu}}\\
        &\le C\bigl(\| q^+ \|_{\G_f} + \| w_0 \|_{\rH_\per^4(\Gamma)} + \| v_0 \|_{\rH_\per^2(\Gamma)} + \| \qbr_0 \|_{\rH_\per^1(\Gamma)} + \| \qbar_0 \|_{\rL^2(\Gamma)} + \| \bu_0 \|_{\rX_{f,\sigma}}\bigr).
    \end{aligned}
\end{equation*}

\medskip

\noindent
\emph{\underline{Step 6}: Exponential weights.}
For $\mu > 0$, the weighted estimate follows analogously. 
By \autoref{prop:concrete lifting},
\[
    \| \mre^{\mu t} \bu_* \|_{\E_{f,\bu}} \le C \| \mre^{\mu t} q^+ \|_{\G_f}.
\]
The trace estimate carries over to the weighted setting:
\begin{equation*}
    \| \mre^{\mu t}\left.12\bu_* \right|_{\Gamma} \cdot \bn \|_{\rL^2(\R_+;\rL^2(\Gamma))} \le C \cdot \| \mre^{\mu t} \bu_* \|_{\E_{f,\bu}},
\end{equation*}
and repeating the above argument yields the desired exponential estimate.

\medskip

This completes the proof.
\end{proof}

The above theorem shows that the regularized FPSI system exhibits the characteristic behavior of a parabolic evolution problem. 
In particular, the system is globally well-posed, it regularizes the solution over time, and it dissipates energy at an exponential rate.

From a physical and analytical perspective, this behavior can be understood as follows. 
The viscosity of the fluid and the damping in the viscoelastic plate contribute to the dissipation of energy, while the elliptic nature of the spatial operators provides smoothing effects. 
Moreover, the coupling between the fluid and the structure enters as a lower-order contribution and is therefore controlled within the perturbative framework. 
Taken together, these mechanisms yield a stable evolution with solutions that decay exponentially in time.

\section{Numerics}\label{sec:numerics}

To numerically validate our model, we compare simulation results for two fluid--poroelastic structure interaction (FPSI) configurations involving poroelastic media. 
In both problems the free fluid domain is set in $\mathbb{R}^2$:

\begin{enumerate}[label=\Roman*.]
    \item \textbf{Stokes flow coupled with a fully averaged Kirchhoff poroelastic plate:}
    This is the model we introduced in the present manuscript -- it consists of a $2D$ incompressible, viscous fluid governed by the time-dependent Stokes equations interacting with a $1D$ poroelastic Kirchhoff plate described by the fully averaged plate model \eqref{eq:fully averaged plate model}. The corresponding coupled problem is given by \eqref{eq:FPSI probl} and \eqref{DynCouplingTraction}.

    \item \textbf{Stokes flow coupled with a bulk Biot medium:}
    We consider the interaction between a $2D$ time-dependent Stokes fluid and a $2D$ bulk poroelastic medium modeled by the classical Biot equations. The associated coupled formulation is specified in \eqref{Biot_sys}--\eqref{BC_BiotPlate1}.
\end{enumerate}
In both models, the free fluid flow -- governed by the Stokes equations -- is driven by a time-dependent horizontal pressure drop in a two-dimensional cylindrical domain whose upper (lateral) boundary is poroelastic. 
See \autoref{fig:2d_model_comparison_aligned}.
The fluid flow is coupled to the poroelastic medium through a set of kinematic and dynamic interface conditions. 
A symmetry boundary condition is imposed at the bottom of the fluid domain $\Gamma_{\textrm{sym}}$, see  \eqref{BC_BiotPlate1} right.
The poroelastic medium is fixed at the left and right boundaries (clamped in case of the plate), while a drained boundary condition (i.e., zero pore pressure) is prescribed on its top surface.

We compare the results of the two models for two different thicknesses of the poroelastic layer and observe that:
\begin{itemize}
    \item The numerical results (pressure, displacement, and pressure jump) of the two models are in excellent agreement.
    \item As the thickness of the poroelastic medium decreases, the differences between the models vanish, and their results become indistinguishable.
    \item The wave speed of the traveling wave generated by the inlet/outlet pressure data and supported by the poroelastic structural walls is in excellent agreement with the Moens--Korteweg pulse wave velocity, given in \eqref{eq:MK} below.
\end{itemize}

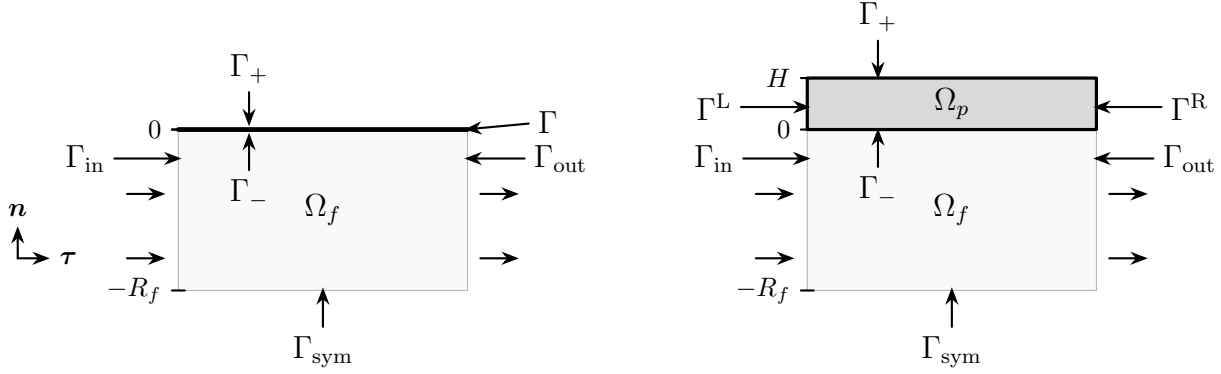
\begin{figure}[h!]
\centering
    \begin{subfigure}[b]{0.48\textwidth}
        \centering
        \begin{tikzpicture}[>=Stealth, line join=round, line cap=round, scale=0.85]
            \def\L{4.5}
            \def\H{0.0}
            \def\Rf{2.5}
            \def\arrL{0.6}

            \path[use as bounding box] (-2.2,-3.6) rectangle (\L+2.2,\H+1.6);

            \draw[gray!60, fill=gray!5] (0,-\Rf) rectangle (\L,\H);
            \node[black] at (\L/2, {(\H-\Rf)/2}) {\Large $\Omega_f$};

            \draw[black, line width=1.8pt] (0,\H) -- (\L,\H);

            \node (GammaLabel) at (\L+1.25,\H+0.12) {\Large $\Gamma$};
            \draw[thick, ->] (GammaLabel) -- (\L-0.05,\H);

            \node (GplusLabel) at (1.1,\H+0.95) {\Large $\Gamma_+$ };
            \draw[thick, ->] (GplusLabel) -- (1.1,\H+0.05);

            \node (GminusLabel) at (1.1,\H-1.00) {\Large $\Gamma_-$};
            \draw[thick, ->] (GminusLabel) -- (1.1,\H-0.05);

            \node (GsymLabel) at (\L/2,-\Rf-0.95) {\Large $\Gamma_{\mathrm{sym}}$};
            \draw[thick, ->] (GsymLabel) -- (\L/2,-\Rf+0.05);
            
            \node (GinLabel) at (-1.45,-0.45) {\Large $\Gamma_{\mathrm{in}}$};
            \draw[thick, ->] (GinLabel) -- (0.05,-0.45);

            \node (GoutLabel) at (\L+1.45,-0.45) {\Large $\Gamma_{\mathrm{out}}$};
            \draw[thick, ->] (GoutLabel) -- (\L-0.05,-0.45);

            \draw[thick] (-0.1,\H) -- (0.1,\H) node[left=5pt] {};
            \draw[thick] (-0.1,0) -- (0.1,0) node[left=5pt] {$0$};
            \draw[thick] (-0.1,-\Rf) -- (0.1,-\Rf) node[left=5pt] {$-R_f$};

            \foreach \y in {-2.0,-1.0} {
                \draw[thick, ->] (-\arrL-0.2,\y) -- (-0.2,\y);
                \draw[thick, ->] (\L+0.2,\y) -- (\L+\arrL+0.2,\y);
            }

            \begin{scope}[shift={(-2.5,-2)}]
                \draw[thick, ->] (0,0) -- (0.5,0) node[right] {$\boldsymbol{\tau}$};
                \draw[thick, ->] (0,0) -- (0,0.5) node[above] {$\boldsymbol{n}$};
            \end{scope}
        \end{tikzpicture}
    \end{subfigure}
    \hfill
    \begin{subfigure}[b]{0.48\textwidth}
        \centering
        \begin{tikzpicture}[>=Stealth, line join=round, line cap=round, scale=0.85]
            \def\L{4.5}
            \def\H{0.8}
            \def\Rf{2.5}
            \def\arrL{0.6}

            \path[use as bounding box] (-2.2,-3.6) rectangle (\L+2.2,\H+1.6);

            \draw[gray!60, fill=gray!5] (0,-\Rf) rectangle (\L, 0);
            \node[black] at (\L/2, -\Rf/2) {\Large $\Omega_f$};

            \draw[black, fill=gray!30, line width=1.2pt] (0,0) rectangle (\L, \H);
            \node[black] at (\L/2, \H/2) {\Large $\Omega_p$};

            \node (GplusR) at (1.1,\H+0.95) {\Large $\Gamma_+$};
            \draw[thick, ->] (GplusR) -- (1.1,\H-0.02);

            \node (GminusR) at (1.1,-0.95) {\Large $\Gamma_-$};
            \draw[thick, ->] (GminusR) -- (1.1,0.02);

            \node (GsymR) at (\L/2,-\Rf-0.95) {\Large $\Gamma_{\mathrm{sym}}$};
            \draw[thick, ->] (GsymR) -- (\L/2,-\Rf+0.05);

            \node (GinR) at (-1.45,-0.45) {\Large $\Gamma_{\mathrm{in}}$};
            \draw[thick, ->] (GinR) -- (0.05,-0.45);

            \node (GoutR) at (\L+1.45,-0.45) {\Large $\Gamma_{\mathrm{out}}$};
            \draw[thick, ->] (GoutR) -- (\L-0.05,-0.45);

            \node (GinR) at (-1.45,-0.45+\H) {\Large $\Gamma^{\mathrm{L}}$};
            \draw[thick, ->] (GinR) -- (0.05,-0.45+\H);

            \node (GoutR) at (\L+1.45,-0.45+\H) {\Large $\Gamma^{\mathrm{R}}$};
            \draw[thick, ->] (GoutR) -- (\L-0.05,-0.45+\H);

            \draw[thick] (-0.1,\H) -- (0.1,\H) node[left=5pt] {$H$};
            \draw[thick] (-0.1,0) -- (0.1,0) node[left=5pt] {$0$};
            \draw[thick] (-0.1,-\Rf) -- (0.1,-\Rf) node[left=5pt] {$-R_f$};

            \foreach \y in {-2.0,-1.0} {
                \draw[thick, ->] (-\arrL-0.2,\y) -- (-0.2,\y);
                \draw[thick, ->] (\L+0.2,\y) -- (\L+\arrL+0.2,\y);
            }
        \end{tikzpicture}
    \end{subfigure}

\caption{The $2D$ computational domains for the Stokes-plate problem (left) and the Stokes--Biot problem (right).  
While the Kirchhoff poroelastic plate equations are defined on the one-dimensional domain $\Gamma$, the equations incorporate the information about the physical thickness of the domain $H$, which is equal to the thickness of the Biot domain~$\Omega_p$. 
In both cases the flow is driven by the pressure gradient between $\Gamma_{\textrm{in}}$ and $\Gamma_{\textrm{out}}$.}
\label{fig:2d_model_comparison_aligned}
\end{figure}

To present those results we start by specifying the Stokes--Biot coupled problem, i.e., Problem II above. 

\medskip
\noindent{{\bfseries Problem II: Stokes flow coupled with a bulk Biot medium.}}

In this coupled problem, both the fluid domain $\Omega_f$ and the domain of the equations of the Biot poroelastic medium $\Omega_p$ are of the same dimension.
The coupled problem combines the Biot equations of poroelasticity~\eqref{Biot_sys} with the Stokes equations \eqref{Stokes_sub2}, specified as follows:
\begin{equation}\label{Biot_sys}
    \left\{
    \begin{aligned}
        \rho_{b} \partial_{t}^2 \boldsymbol{\eta}-\nabla \cdot \boldsymbol{\sigma}_{b}(\boldsymbol \eta,q) + \gamma \boldsymbol\eta 
        &= \boldsymbol{F}_b, &&\tin (0,T) \times \Omega_f ,\\
        \partial_{t}(c_{0}  q +\alpha \nabla \cdot  \boldsymbol{\eta}) + \nabla \cdot \boldsymbol u_b 
        &= G_b, &&\tin (0,T) \times \Omega_p,\\
        \boldsymbol{u}_b 
        &= - \boldsymbol\kappa \nabla q,
        &&\tin (0,T) \times \Omega_p, 
    \end{aligned}
    \right.
\end{equation}
\begin{equation}\label{Stokes_sub2}
    \left\{
    \begin{aligned}
        \rho_f \del_t \bu - \nabla \cdot \bsigma_f(\bu,\pi)
        &= 0, &&\tin (0,T) \times \Omega_f ,\\
        \nabla \cdot \bu 
        &=0,  &&\tin (0,T) \times \Omega_f.
    \end{aligned}
    \right.
\end{equation}

This problem is supplemented with the following coupling conditions, assumed of the fluid-structure interface, denoted by $\Gamma^-$ in \autoref{fig:2d_model_comparison_aligned}:

\medskip
\noindent{{\bfseries Kinematic coupling conditions:}}
\begin{itemize}
    \item Continuity of normal components of fluid velocities relative to the poroelastic matrix  motion:
    \begin{equation*}
        \left.  \bu_p \cdot \bn \right|_{\Gamma^-} = \left. \bu \right|_{\Gamma^-} \cdot \bn - \left. \bv\right|_{\Gamma^-} \cdot \bn,
    \end{equation*}
    where $\bv = \partial_t \boldsymbol\eta$, and $\bn$ is the normal to the interface $\Gamma$ between $\Omega_f$ and $\Omega_p$. 
    \item Slip in the tangential component of fluid velocity (Beavers--Joseph--Saffman):
    \begin{equation*}
        \left.\bsigma_f(\bu,\pi)\right|_{\Gamma^-} \bn  \cdot \btau 
        = - \beta \left. \left(\bu - \bv \right) \right|_{\Gamma^-} \cdot \btau, 
    \end{equation*}
    where $\beta > 0$ is friction parameter, and $\btau$ is the tangent vector to $\Gamma$, aligned with the $x$-axis.
\end{itemize}
\medskip
\noindent{{\bfseries Dynamic coupling conditions:}}
\begin{itemize}
    \item Continuity of normal components of normal fluid stress (continuity of ``pressure''):
    \begin{equation*}
        \left.\bsigma_f(\bu,\pi)\right|_{\Gamma^-}  \bn  \cdot \bn = - \left. q \right|_{\Gamma^-},
    \end{equation*}
    where $\left. q \right|_{\Gamma^-}$ is  the trace of $q$ on $\Gamma^-$.

    \item Balance of normal stress at the fluid-structure interface: 
    \begin{equation*}
        \left. \bsigma_f(\bu,\pi)\right|_{\Gamma^-} \bn  
     -  \left. \bsigma_b(\bu,q)\right|_{\Gamma^-} \bn = 0.
    \end{equation*}
\end{itemize}

Consistent with Problems I and II, the following boundary conditions are prescribed for this problem:

\medskip
\noindent{{\bfseries Boundary Conditions:}}
\begin{itemize}
    \item {Biot medium:} We prescribe fluid pressure on the top boundary of $\Omega_p$ denoted by $\Gamma^+$, and zero Biot medium displacement at the left and right boundaries of $\Omega_p$ denoted by $\Gamma^{\mathrm{L}}$ and $\Gamma^{\textrm{R}}$:
    \begin{equation*}
    \begin{array}{lll}
        \left. q \right|_{\Gamma^+} = q^+ = 0 \ &{\textrm{and}} \ 
        &\left. \boldsymbol{\sigma}_b(\boldsymbol{\eta},q)\right|_{\Gamma^+} \bn= 0,
        \\
        \left. \boldsymbol\eta \right|_{\Gamma^{\textrm{L}}}= \left. \boldsymbol\eta \right|_{\Gamma^{\textrm{R}}} =0\ \ &{\textrm{and}} \ 
        &\left. \bu_b\right|_{\Gamma^{\textrm{L}}} \cdot (- \boldsymbol{\nu}) = 
        \left. \bu_b\right|_{\Gamma^{\textrm{R}}} \cdot \boldsymbol{\nu} = 0.
    \end{array}
    \end{equation*}
    Here $\boldsymbol\nu = \pm\boldsymbol{e}_1$ is used to denote the unit vector aligned with the horizontal axis, which is normal to the boundary $\Gamma^{\textrm{R}}$.
    \item Free fluid: {We prescribe the inlet and outlet normal stress boundary conditions, and at the bottom boundary, we prescribe the symmetry boundary conditions:}
    \begin{equation}\label{BC_BiotPlate1}
        \left\{
        \begin{aligned}
            \bsigma_f(\boldsymbol u,\pi)\boldsymbol n
            &= P_{\mathrm{in}}(t)\boldsymbol n, &&\ton \Gamma_{\mathrm{in}},\\
            \bsigma_f(\boldsymbol u,\pi)\boldsymbol n
            &= 0, &&\ton \Gamma_{\mathrm{out}},
        \end{aligned}
        \right.
        \tand
        \left\{
        \begin{aligned}
            \bsigma_f(\boldsymbol u,\pi)\boldsymbol n\cdot \boldsymbol \tau
            &= 0,\\
            \boldsymbol u\cdot\boldsymbol n
            &= 0,
        \end{aligned}
        \right. \ton \Gamma_\mathrm{sym}.
\end{equation}
\end{itemize}

\medskip
\noindent
{\bf{Initial Conditions.}}
The initial conditions for both problems assume zero structural displacement (and velocity), zero fluid velocities, and zero initial pressure. 

The values of the parameters used in the simulations are presented in \autoref{tab:param}.
\begin{table}[h]
    \centering
    \begin{tabular}{|l l|l l|}
        \hline
        \footnotesize Parameters &\footnotesize Values &\footnotesize Parameters &\footnotesize Values \\
        \hline
        \footnotesize Fluid radius $R_f$ (cm) &\footnotesize 0.5 &\footnotesize Length $L$ (cm) &\footnotesize 5 \\
        \footnotesize PE wall thickness $H$ (cm) &\footnotesize .01 and  .001  &  \footnotesize Fluid viscosity $\mu_f$ (g/cm s) &\footnotesize  0.035 \\
        \footnotesize Fluid density $\rho_f$ (g/cm$^3$) &\footnotesize 1  & \footnotesize PE wall density $\rho_b$ (g/cm$^3$) &\footnotesize 1.1  \\
        \footnotesize Lam\'{e} coeff.~$\mu_b$ (dyne/cm$^2$) &\footnotesize $5.58 \times 10^5$ &\footnotesize Lam\'{e} coeff.~$\lambda_b$ (dyne/cm$^2$) &\footnotesize $1.7 \times 10^6$ \\
        \footnotesize Permeability $\kappa$ (cm$^3$ s/g) & \footnotesize $1 \times 10^{-8}$ & \footnotesize Storage coeff.~$c_0$ (cm$^2$/dyne) &\footnotesize $1 \times 10^{-3}$ \\
       \footnotesize Biot--Willis constant $\alpha$ & \footnotesize 1 &\footnotesize Spring coeff.~$\gamma$ (dyne/cm$^4$) &\footnotesize $4 \times 10^6$ \\
\footnotesize Slip coeff.~$\beta$ & \footnotesize 1 & \footnotesize Bending stiffness $D$ (dyne/cm$^2$)& \footnotesize $1.5 \times 10^5$ \\
        \hline
    \end{tabular}
\caption{Geometry, fluid, and structure parameters used in the coupled FPSI problem}
\label{tab:param}
\end{table}

\medskip
\noindent{{\bfseries The example.}}
We consider a biologically relevant example in which the model parameters are drawn from a benchmark problem in FPSI describing blood flow in an artery with poroelastic walls;
see, for example, \cite{SBC:25} and the references therein.

We model a segment of an artery with poroelastic walls by the rectangular domain $\Omega = \Omega_f \cup \Omega_p \subset \mathbb{R}^2$, where
\begin{equation}\label{Domains}
    \Omega = [0,L]\times[-R_f,H], \quad \Omega_p = [0,L]\times[0,H], \tand \Omega_f = [0,L]\times[-R_f,0].
\end{equation}
See \autoref{fig:2d_model_comparison_aligned}.
Here, $L$ denotes the length of the arterial segment, and $R_f$ represents its radius. The fluid occupies the lower subdomain $\Omega_f$, with symmetric boundary conditions imposed at the bottom.
The interface $\Gamma$ is located at the upper boundary of $\Omega_f$, corresponding to $y = 0$.

The poroelastic medium is modeled as a thin layer of thickness $H$ occupying the region $\Omega_p$, situated above the fluid domain.
In the fully averaged Kirchhoff plate, the poroelastic region is denoted by $\Gamma$.

As mentioned earlier in \eqref{BC_BiotPlate1}, we consider a pressure-driven flow through the fluid domain with the outlet normal stress/pressure equal to zero, and the inlet normal component of the normal stress given by the following time-dependent  ``pressure'' pulse, with maximum pressure $P_{\max}$ on the time interval $[0,T_{\mathrm{pulse}}]$:
\begin{equation*}
    P_{\mathrm{in}}(t) = \frac{1}{2}P_{\max}\left(1 - \cos\left(2\pi t/T_{\mathrm{pulse}}\right)\right)1_{t < T_{\mathrm{pulse}}}(t),\ {\textrm{where}} \ P_{\max} = 13333 \, {\textrm{dyne/cm}}^2, T_{\mathrm{pulse}} = 3 \,\textrm{sec}.
\end{equation*}
The pressure $q$ on the top boundary of the poroelastic domain is prescribed to be zero, corresponding to the ``drained boundary condition'' $q^+ = 0$.

We highlight an important feature of the model: 
the viscoelastic parameters and the storage coefficient~$c_0$ in the bulk Biot model and the corresponding Biot plate model are linked through the following relations:
\[ 
    \kappa^p = \kappa, \quad \rho_p = \rho_b, \quad c_0^p = c_0 + \frac{\alpha^2}{\lambda + 2\mu}, \quad \alpha^p = \frac{2\alpha \mu}{\lambda + 2\mu}, \quad  D = \frac{\mu(\lambda + \mu)}{3(\lambda + 2\mu)}, \quad \gamma_p = \gamma, 
\]
where the quantities without a superscript denote the bulk Biot model parameters.

\medskip
\noindent{{\bfseries Discretization.}}
We discretize in space using finite elements. In particular:
\begin{itemize}
    \item In the fluid problem and full Biot problem, we use Taylor--Hood ($\mathbb P_2 -\mathbb P_1$) elements for the pairs $(\boldsymbol u, \pi)$ and $ (\boldsymbol \xi, q)$;
    \item For the plate filtration velocity, we use discontinuous $DG_0$ elements;
    \item For everywhere else, we use piecewise linear $\mathbb P_1$ elements.
\end{itemize}

The computational domain for the fluid subproblem is discretized by means of a conforming triangulation generated from a structured $301 \times 26$ point grid. 
For the full linear Biot model (Problem II), the poroelastic structure is discretized  using a mesh based on a $301 \times 4$ point grid.

Time integration in both configurations (Problems I and II) is carried out using a fully implicit, monolithic backward Euler scheme with a uniform time step size $\Delta t = 5 \times 10^{-5}$. 
At each time step, the coupled system is solved in a monolithic fashion, ensuring strong coupling between the fluid and poroelastic subproblems.

In Problem II, which involves a bulk Biot poroelastic medium, the kinematic coupling condition at the fluid--structure interface is additionally enforced using Nitsche's penalty method as in \cite{BYZZ:15}. 
To this end, additional penalty terms are incorporated into the variational formulation, following the approach in \cite{BYZZ:15}. 

The resulting discrete Problem II is implemented in \texttt{FreeFEM} \cite{FreeFem}, which is an open-source finite element software package designed for solving partial differential equations and multiphysics problems. 

The discrete reduced problem (Problem I) involving the fully averaged poroelastic Kirchhoff plate model, is solved in Python using the \texttt{FEniCS} finite element framework \cite{Fenics}, which provides automated variational formulation and assembly.

\begin{figure}[htp]
\centering

\begin{subfigure}{0.48\textwidth}
    \centering
    \includegraphics[width=\linewidth]{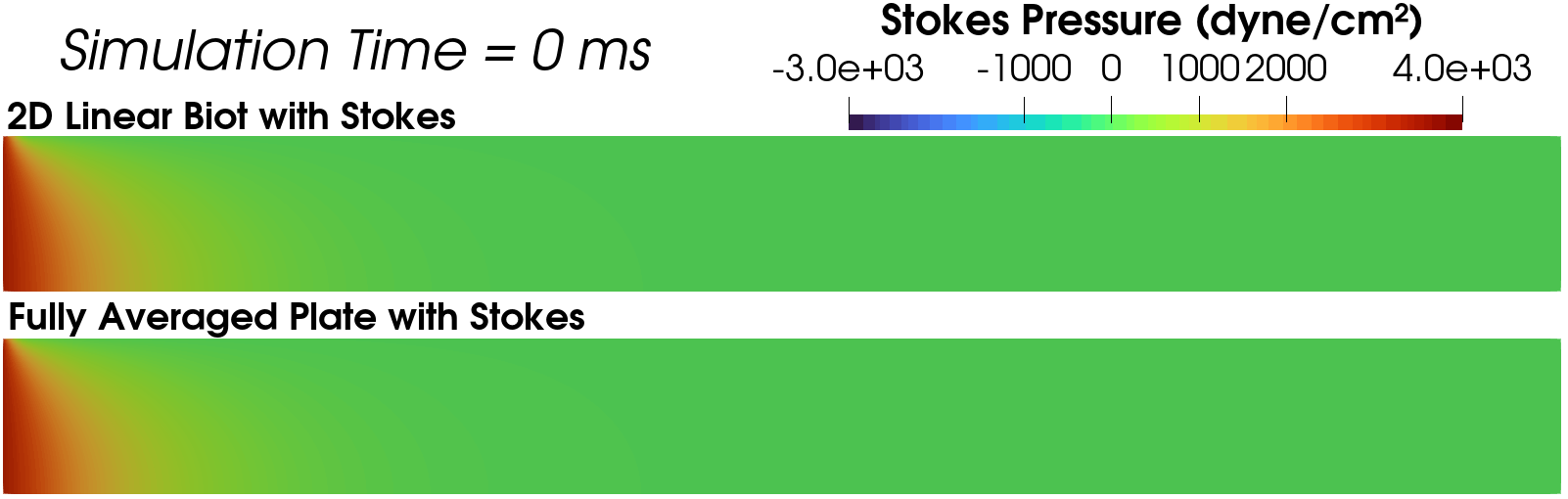}
\end{subfigure}
\hfill
\begin{subfigure}{0.48\textwidth}
    \centering
    \includegraphics[width=\linewidth]{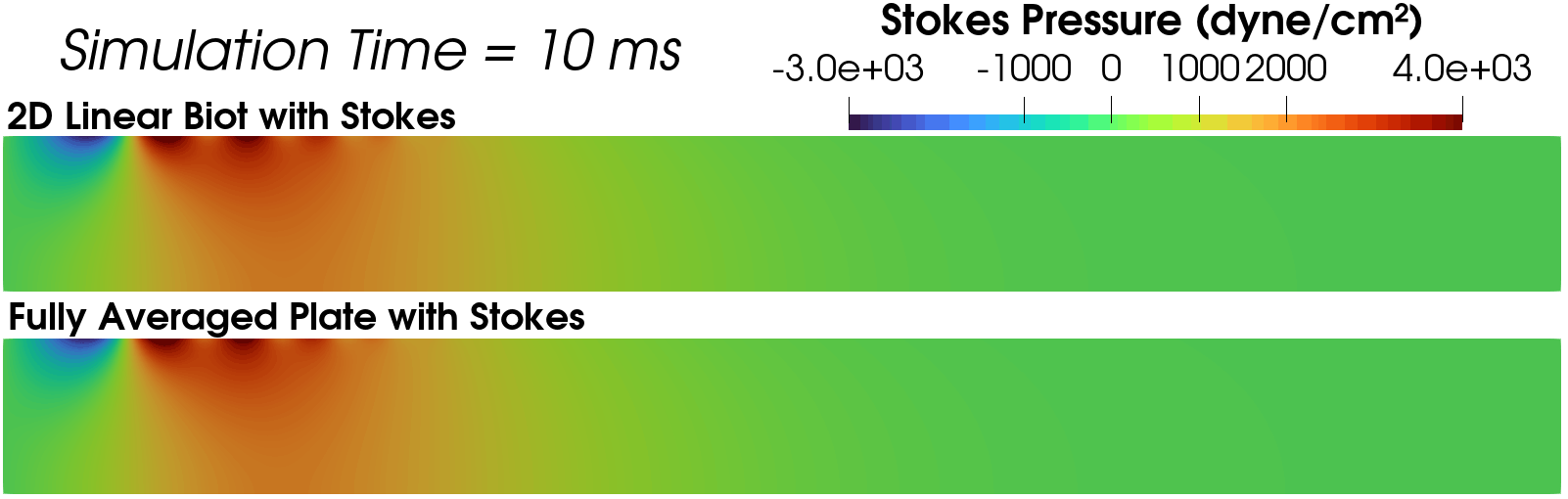}
\end{subfigure}

\begin{subfigure}{0.48\textwidth}
    \centering
    \includegraphics[width=\linewidth]{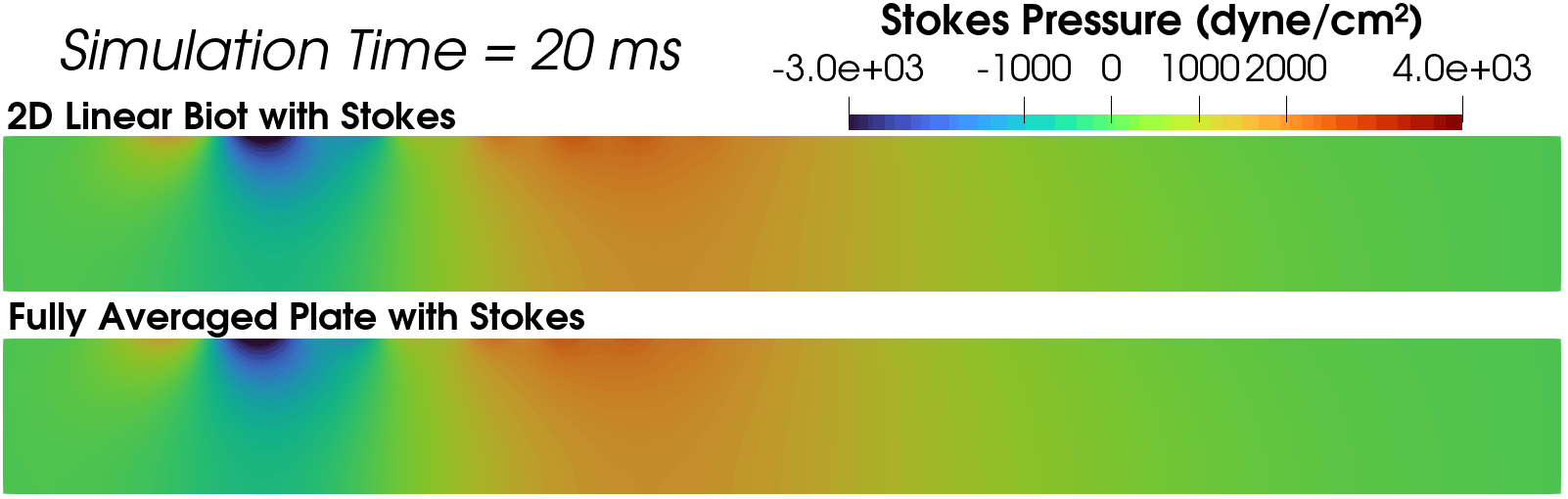}
\end{subfigure}
\hfill
\begin{subfigure}{0.48\textwidth}
    \centering
    \includegraphics[width=\linewidth]{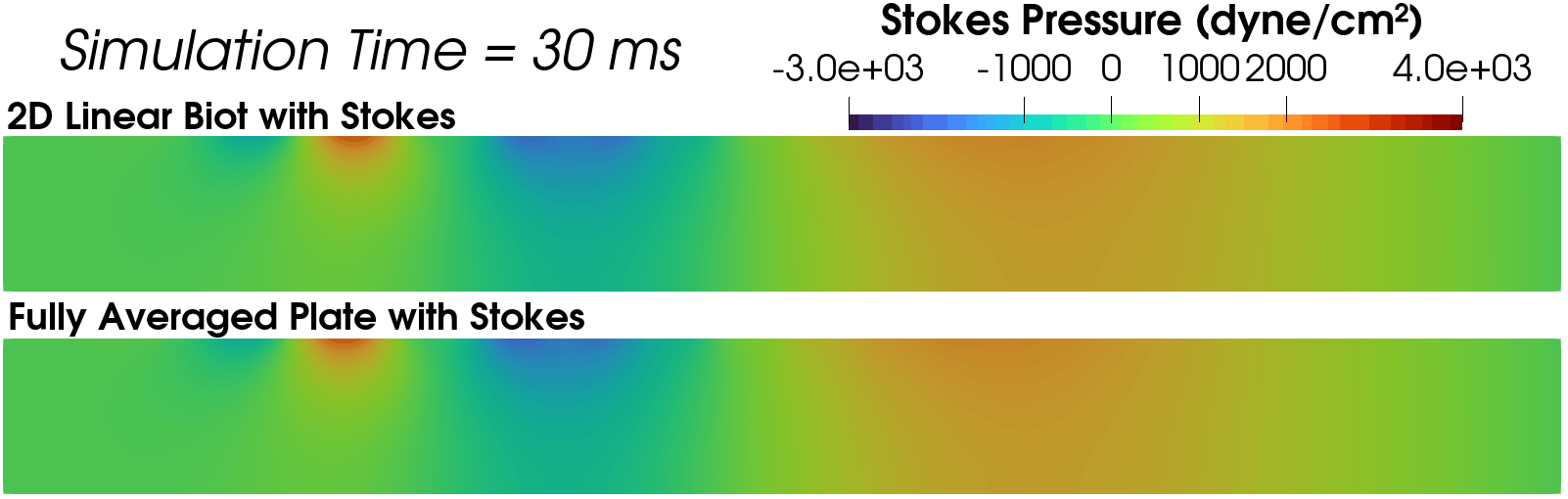}
\end{subfigure}

\caption{$H = 0.01$: Comparison of the Stokes pressure wave for the two models at four different time instances. 
The top panel corresponds to the bulk Biot model (Problem II), and the bottom panel to the fully averaged plate model (Problem I). 
In both cases, the poroelastic layer thickness is $H=0.01$.}
\label{fig:StokesP01}
\end{figure}

The simulation results are presented in Figures~\ref{fig:StokesP01}, \ref{fig:Displ01}, and \ref{fig:Jump01}. 
In particular, \autoref{fig:StokesP01} illustrates the propagation of the pressure wave at four time instances, \(t = 0,\;10,\;20,\) and \(30\) ms, for the two models (Problems~I and II) described above. 
At \(t = 30\) ms, the wave reaches the downstream end of the fluid domain.

This numerically observed speed of propagation is consistent with the classical Moens--Korteweg relation from biomechanics, which predicts the pulse wave velocity as
\begin{equation}\label{eq:MK}
    c = \sqrt{\frac{E H}{2\rho_f R_f}},
\end{equation}
where \(E = \lambda_b + 2\mu_b\) denotes the effective elastic modulus of the wall; see, e.g., \cite{Moens1878,Korteweg1878,Fung1997}.

For the parameter values used in our simulations, the time required for the wave to traverse the fluid domain (vessel) is given by
\[
    t = \frac{L}{c} = 5 \sqrt{\frac{2 \cdot 1 \cdot 0.5}{(1.7 \times 10^6 + 2 \cdot 5.58 \times 10^5)\,H}} = \frac{2.98}{\sqrt{H}} \;\text{ms}.
\]

We observe excellent agreement between the two models in terms of pressure wave propagation, indicating that the fully averaged poroelastic plate model provides an accurate approximation of the FPSI problem, capturing the essential dynamics of  the bulk Biot model with small thickness $H$ interacting with Stokes flow.

Indeed, Figures~\ref{fig:Displ01} and \ref{fig:Jump01} further compare the plate displacement (\autoref{fig:Displ01}) and the pressure jump (\autoref{fig:Jump01}) across the plate interface, for the two models, shown at four different time instances corresponding to those in \autoref{fig:StokesP01}. Both figures show results for the structure thickness of $H = 0.01$.

\begin{figure}[h]
\centering

\begin{subfigure}{0.48\textwidth}
    \centering
    \includegraphics[width=\linewidth]{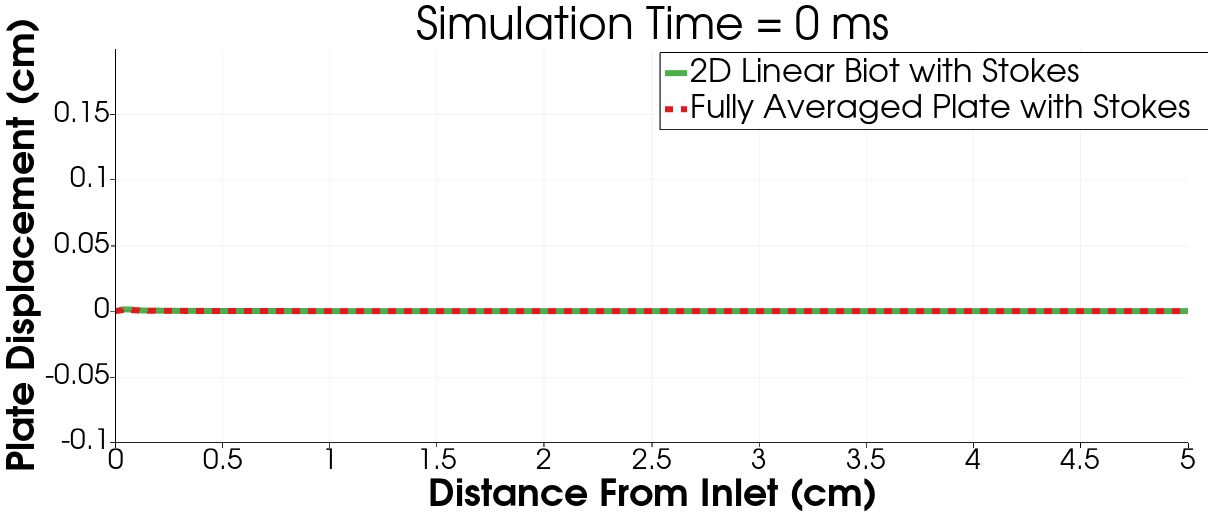}
\end{subfigure}
\hfill
\begin{subfigure}{0.48\textwidth}
    \centering
    \includegraphics[width=\linewidth]{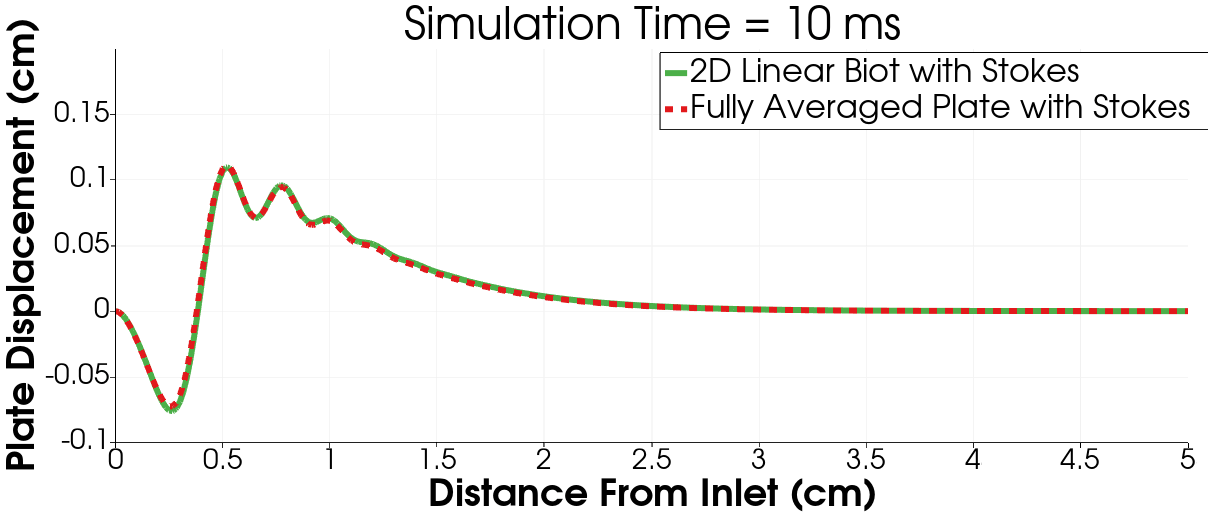}
\end{subfigure}

\begin{subfigure}{0.48\textwidth}
    \centering
    \includegraphics[width=\linewidth]{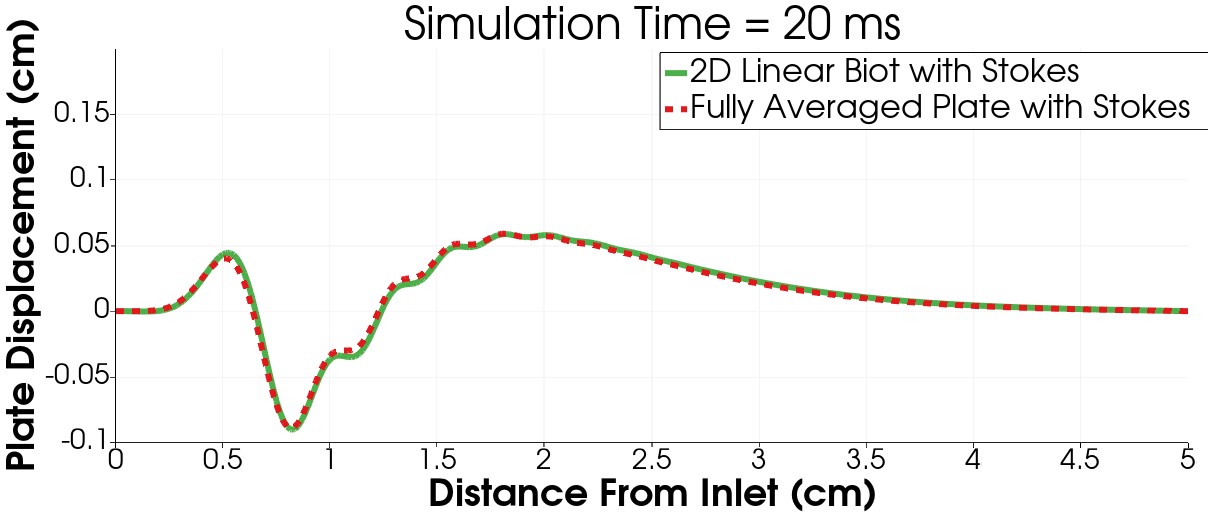}
\end{subfigure}
\hfill
\begin{subfigure}{0.48\textwidth}
    \centering
    \includegraphics[width=\linewidth]{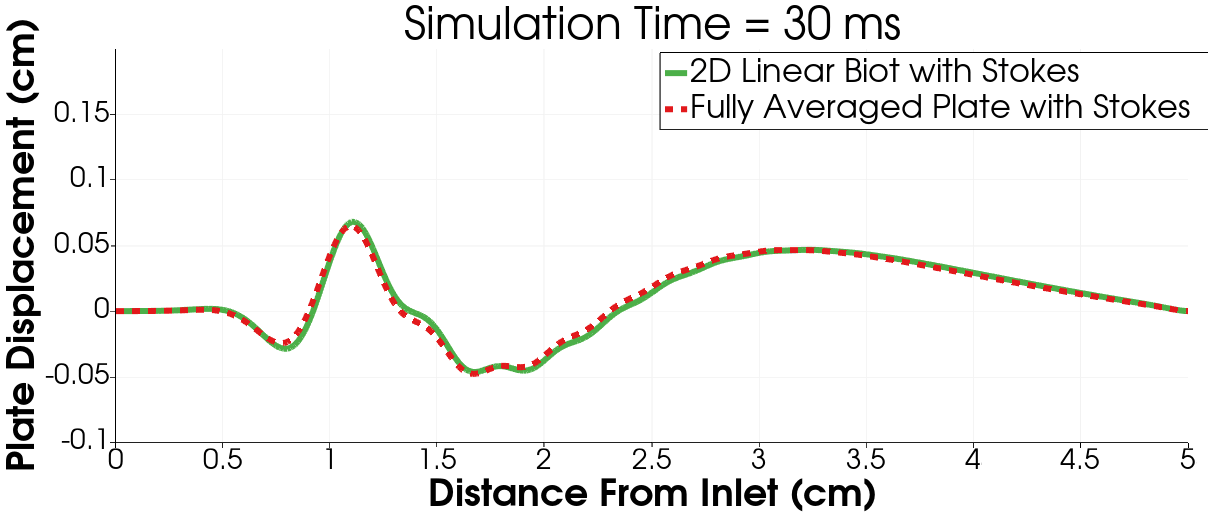}
\end{subfigure}

\caption{Comparison of plate middle surface displacement for thickness $H=0.01$.}
\label{fig:Displ01}
\end{figure}

\begin{figure}[h]
\centering

\begin{subfigure}{0.48\textwidth}
    \centering
    \includegraphics[width=\linewidth]{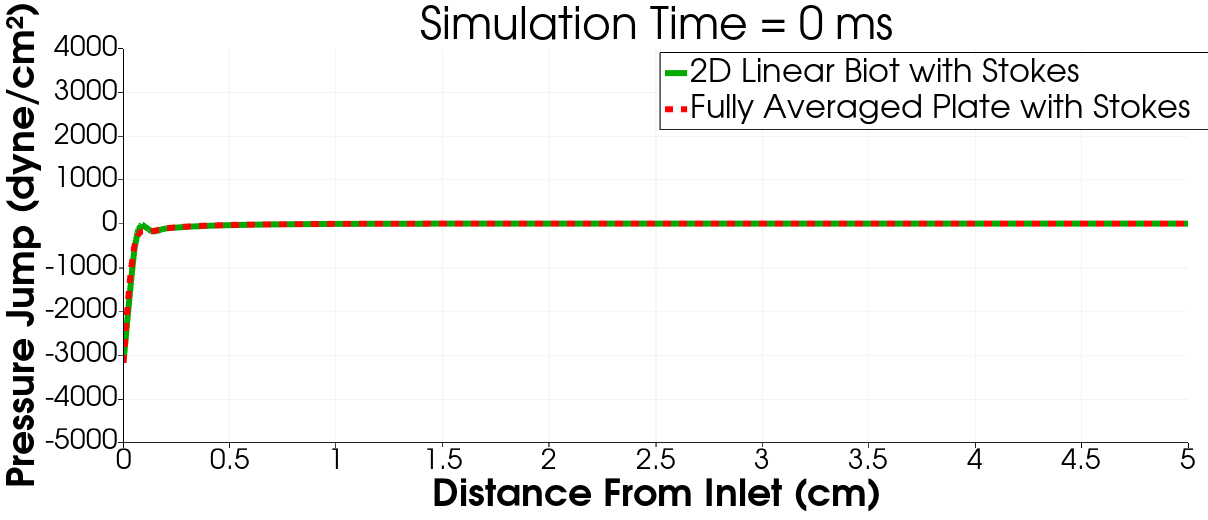}
\end{subfigure}
\hfill
\begin{subfigure}{0.48\textwidth}
    \centering
    \includegraphics[width=\linewidth]{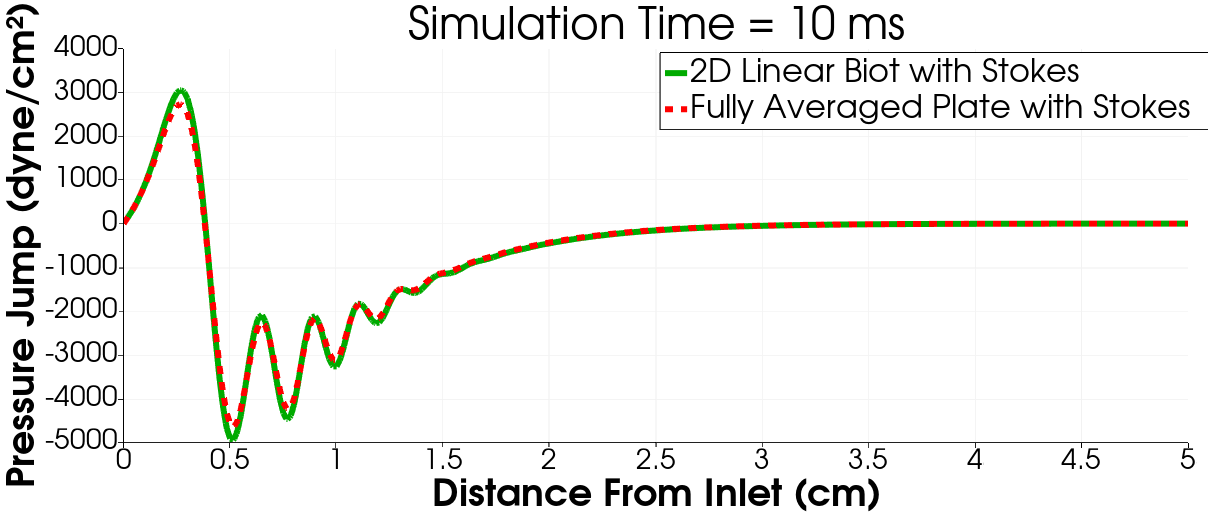}
\end{subfigure}

\begin{subfigure}{0.48\textwidth}
    \centering
    \includegraphics[width=\linewidth]{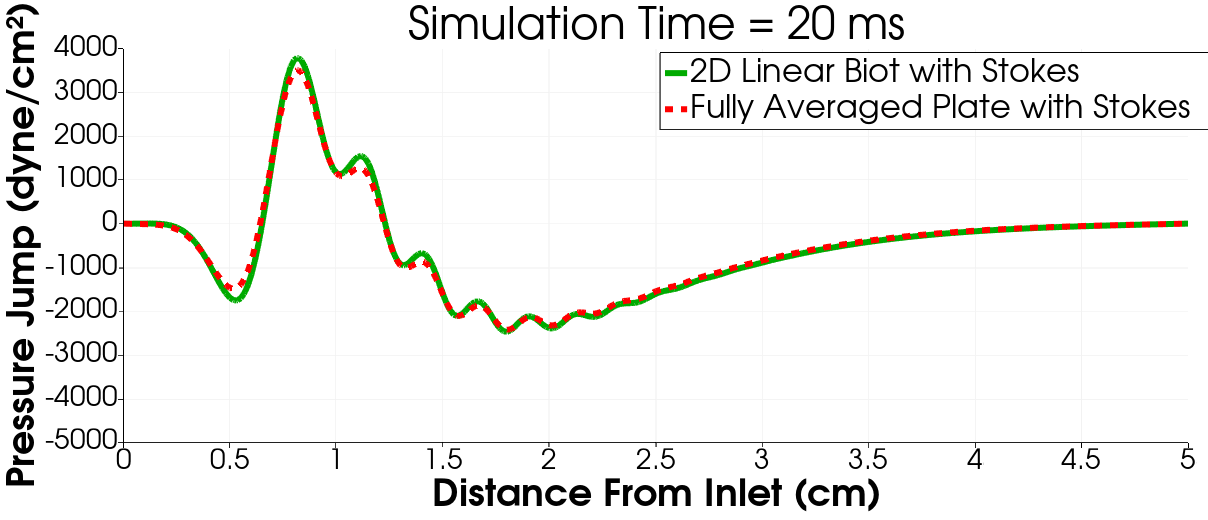}
\end{subfigure}
\hfill
\begin{subfigure}{0.48\textwidth}
    \centering
    \includegraphics[width=\linewidth]{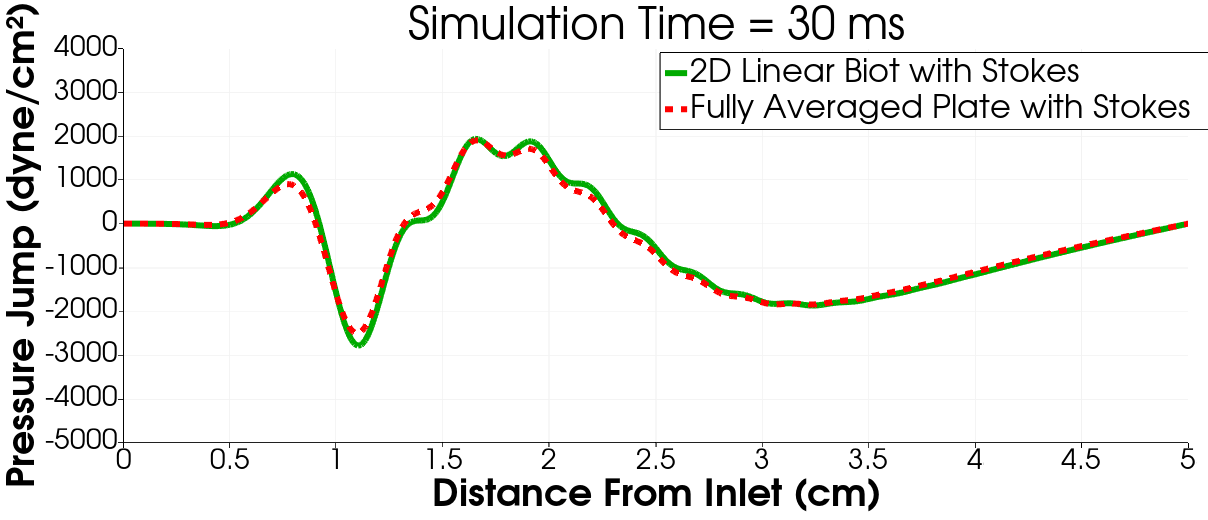}
\end{subfigure}

\caption{Comparison of pressure jump $[q]$ across the poroelastic plate for thickness $H=0.01$.}
\label{fig:Jump01}
\end{figure}

We observe very good agreement between the two models across all four snapshots. 
These simulations correspond to an aspect ratio 
\begin{equation*}
    \frac{H}{L} = 2 \times 10^{-3}.
\end{equation*}

For smaller values of the aspect ratio \(H/L\), we expect the error of the reduced (fully averaged) model to vanish as \(H/L \to 0\). 
To investigate this behavior, we consider ten times smaller value of $H$, namely $H = 0.001$, while keeping the tube length \(L\) fixed, and examine how this affects the accuracy of the approximate model relative to the full Stokes--Biot model. 
The corresponding results are presented in Figures~\ref{fig:StokesP001}, \ref{fig:Displ001}, and \ref{fig:Jump001}.

\begin{figure}[h]
\centering

\begin{subfigure}{0.48\textwidth}
    \centering
    \includegraphics[width=\linewidth]{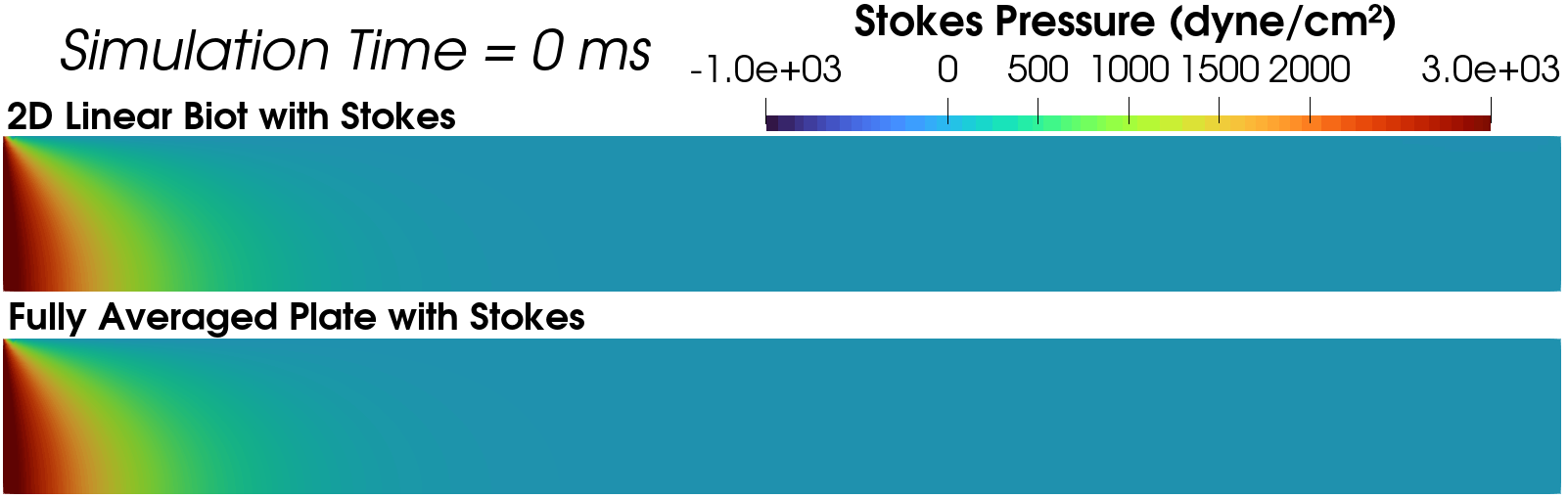}
\end{subfigure}
\hfill
\begin{subfigure}{0.48\textwidth}
    \centering
    \includegraphics[width=\linewidth]{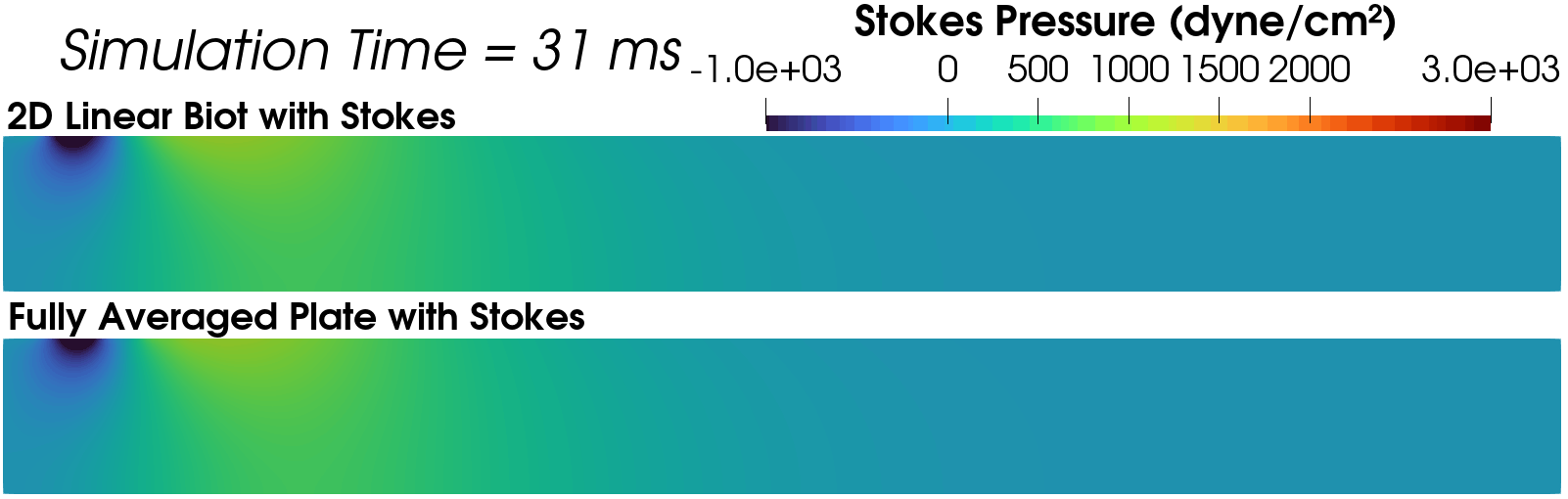}
\end{subfigure}

\begin{subfigure}{0.48\textwidth}
    \centering
    \includegraphics[width=\linewidth]{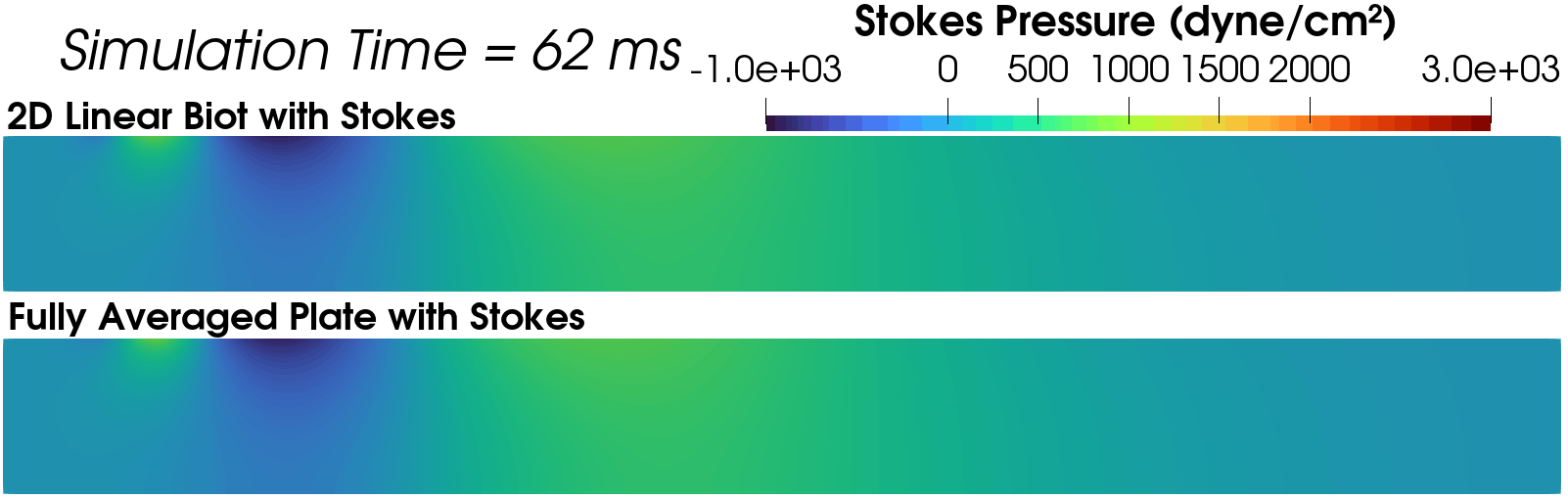}
\end{subfigure}
\hfill
\begin{subfigure}{0.48\textwidth}
    \centering
    \includegraphics[width=\linewidth]{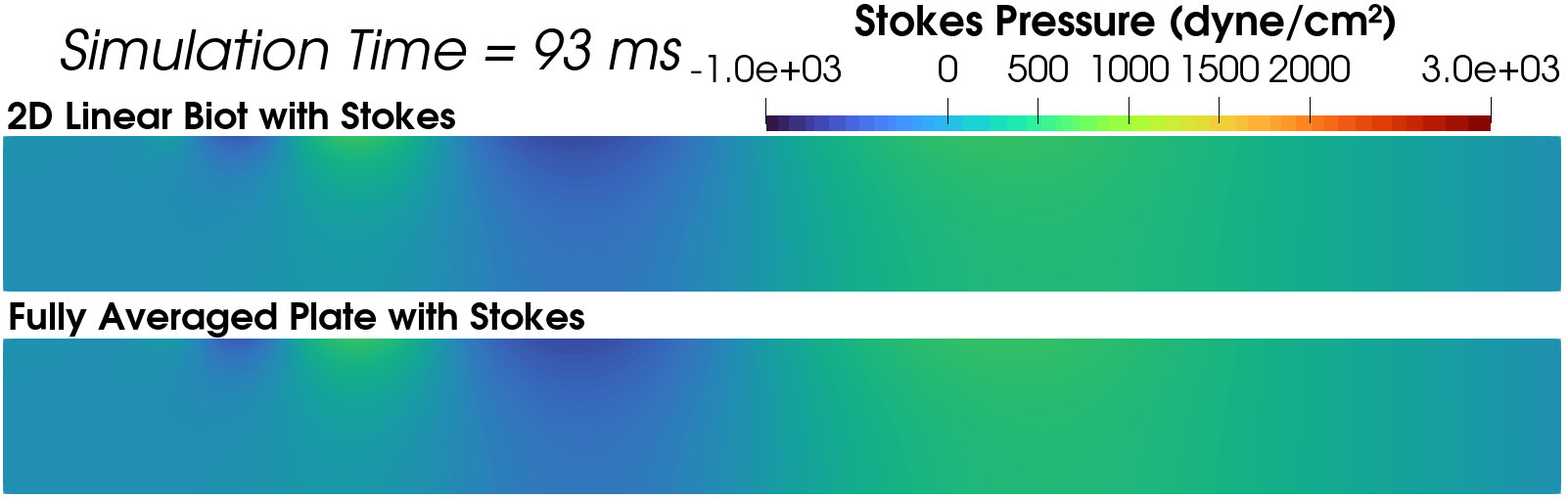}
\end{subfigure}

\caption{$H=0.001$: Comparison of the Stokes pressure wave for the two models at four different time instances. 
The top panel corresponds to the bulk Biot model (Problem II), and the bottom panel to the fully averaged plate model (Problem I). 
In both cases, the poroelastic layer thickness is $H=0.001$.}
\label{fig:StokesP001}
\end{figure}

\begin{figure}[h]
\centering

\begin{subfigure}{0.48\textwidth}
    \centering
    \includegraphics[width=\linewidth]{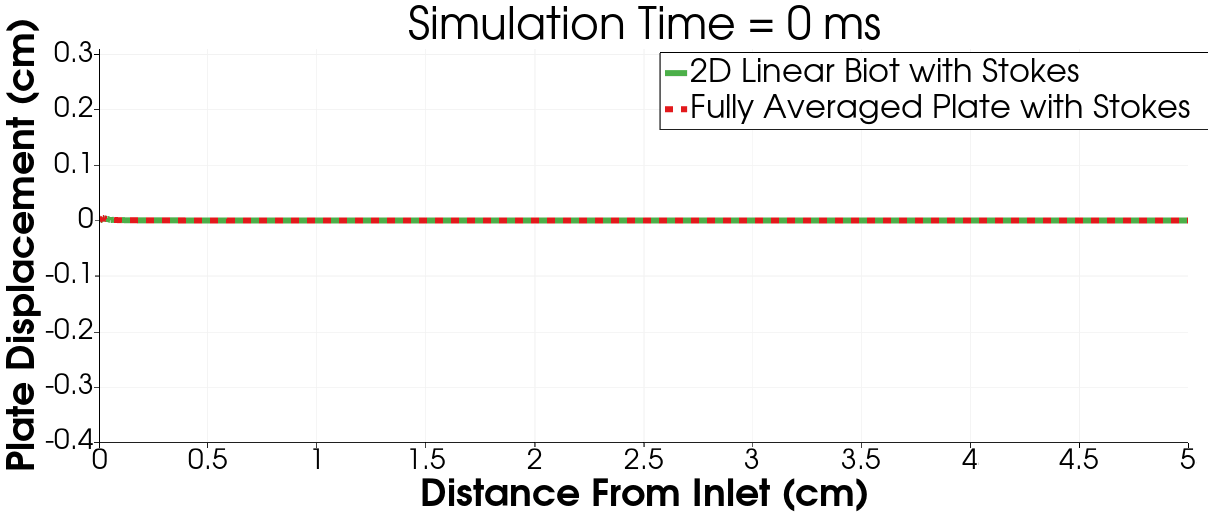}
\end{subfigure}
\hfill
\begin{subfigure}{0.48\textwidth}
    \centering
    \includegraphics[width=\linewidth]{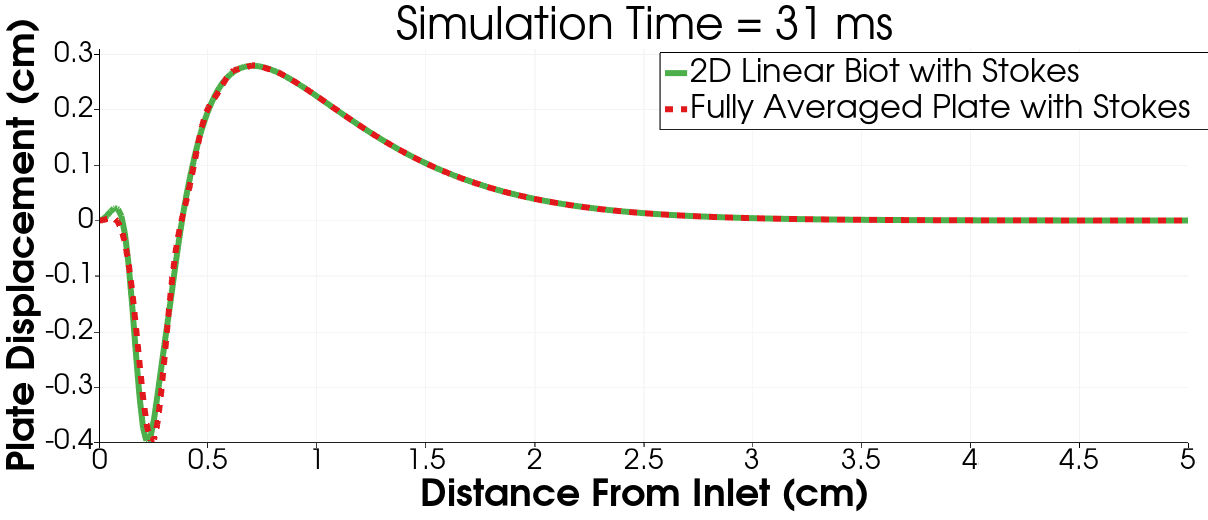}
\end{subfigure}

\begin{subfigure}{0.48\textwidth}
    \centering
    \includegraphics[width=\linewidth]{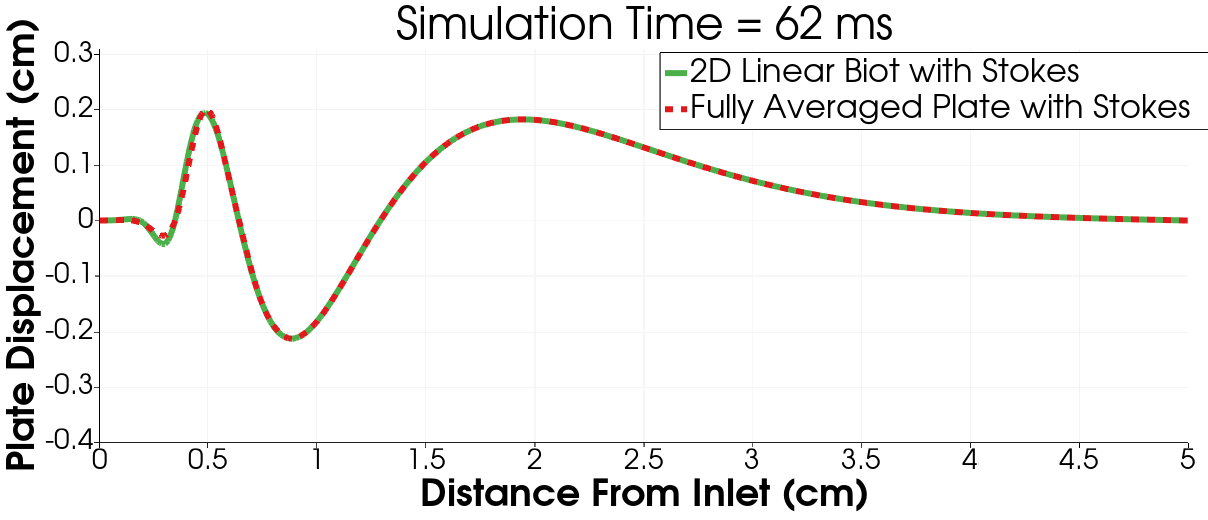}
\end{subfigure}
\hfill
\begin{subfigure}{0.48\textwidth}
    \centering
    \includegraphics[width=\linewidth]{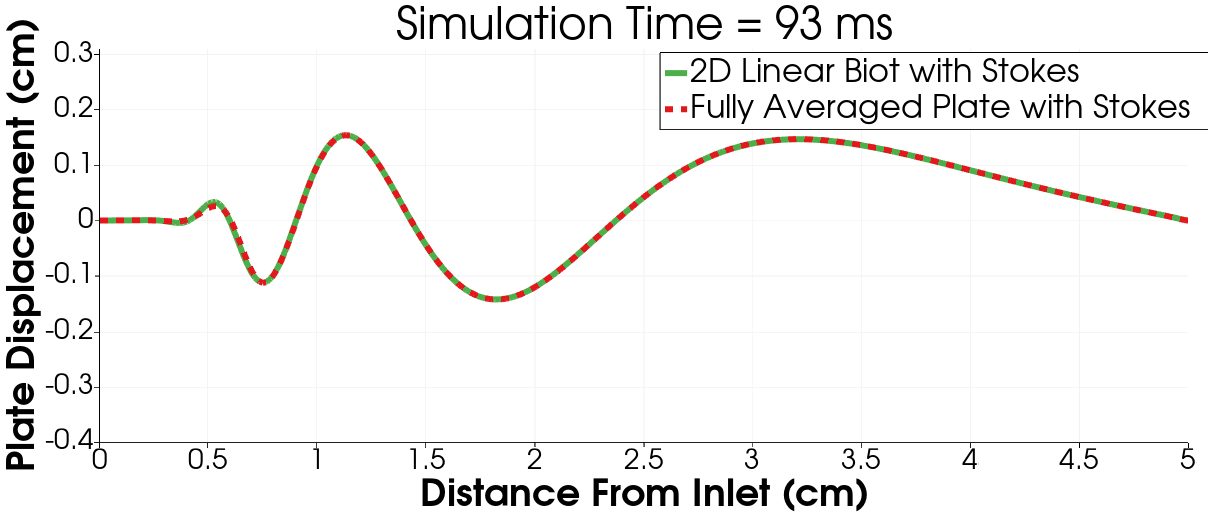}
\end{subfigure}

\caption{Comparison of plate middle surface displacement for thickness $H=0.001$.}
\label{fig:Displ001}
\end{figure}

\begin{figure}[h]
\centering

\begin{subfigure}{0.48\textwidth}
    \centering
    \includegraphics[width=\linewidth]{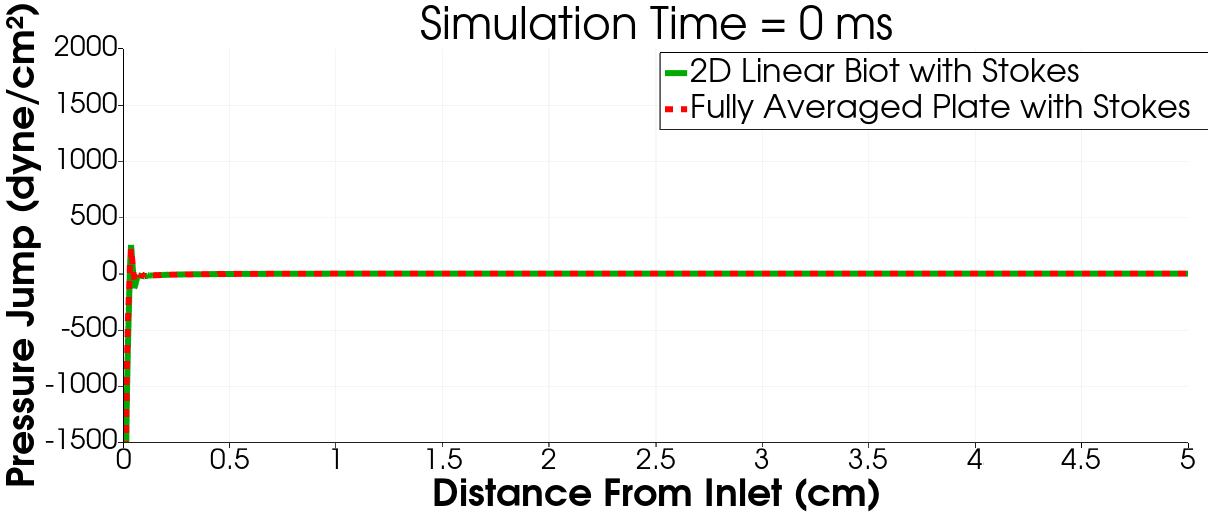}
\end{subfigure}
\hfill
\begin{subfigure}{0.48\textwidth}
    \centering
    \includegraphics[width=\linewidth]{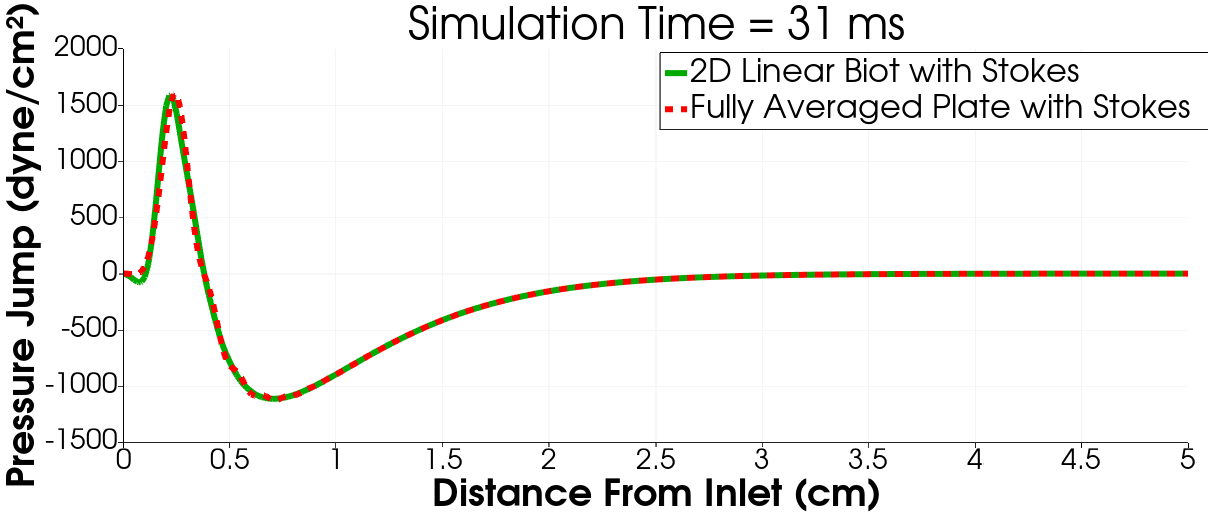}
\end{subfigure}

\begin{subfigure}{0.48\textwidth}
    \centering
    \includegraphics[width=\linewidth]{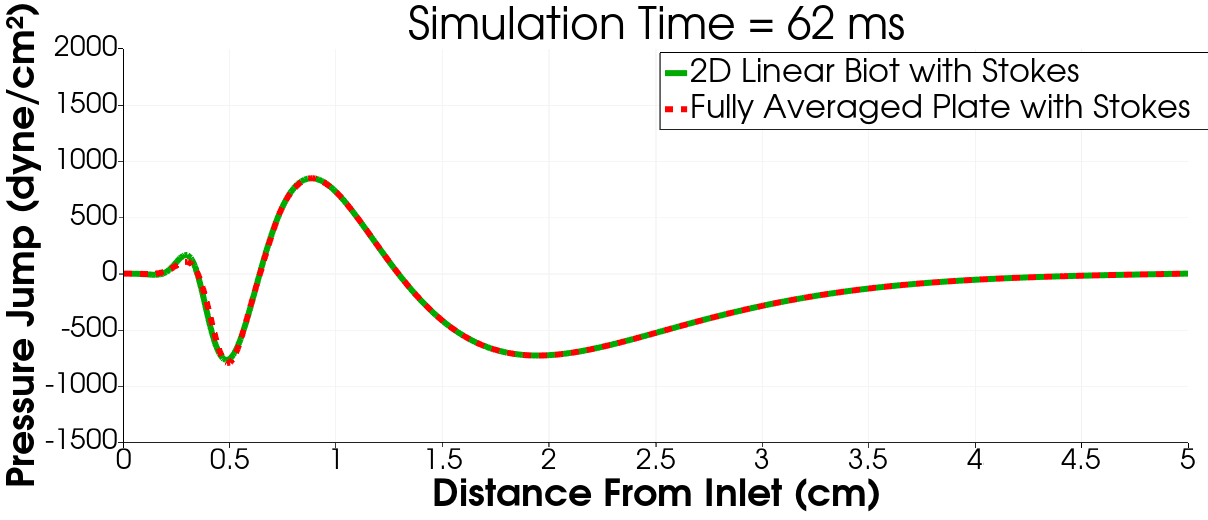}
\end{subfigure}
\hfill
\begin{subfigure}{0.48\textwidth}
    \centering
    \includegraphics[width=\linewidth]{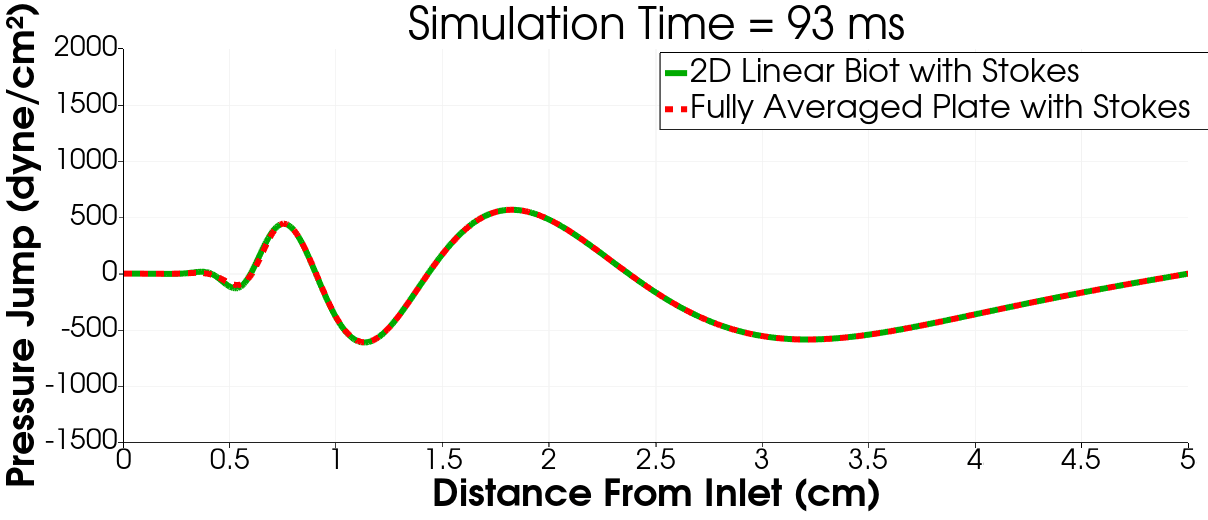}
\end{subfigure}

\caption{Comparison of pressure jump $[q]$ across the poroelastic plate for thickness $H=0.001$.}
\label{fig:Jump001}
\end{figure}

\autoref{fig:StokesP001} illustrates the corresponding Stokes pressure wave propagation, while Figures~\ref{fig:Displ001} and \ref{fig:Jump001} display the structural displacement and the pressure jump across the interface, respectively. 
We observe that for \(H = 0.001\), the displacement and pressure jump profiles for the two models are visually indistinguishable, indicating excellent agreement in this regime.

We conclude this section by highlighting the differences in the solutions corresponding to the two cases \(H = 0.01\) and \(H = 0.001\). 
Three main physical effects can be observed:
\begin{enumerate}
    \item The pressure wave in the case \(H = 0.001\) propagates more slowly, in agreement with the Moens--Korteweg relation~\eqref{eq:MK}. 
    In particular, it reaches the downstream end of the tube at approximately \(93\) ms, compared to about \(30\) ms in the case \(H = 0.01\).

    \item The displacement magnitude is larger for \(H = 0.001\) (with a maximum displacement of approximately \(0.3\) cm), compared to about \(0.1\) cm for \(H = 0.01\). 
    This is consistent with the fact that increasing the wall thickness enhances the overall structural stiffness of the tube. 
    Consequently, for a given pressure loading, a thicker wall undergoes smaller deformations, while a thinner, more compliant wall exhibits larger displacements.

    \item The magnitude of the pressure jump shown in \autoref{fig:Jump001} for \(H = 0.001\) is smaller than in the case \(H = 0.01\). 
    Since \(q^+ = 0\), the pressure jump corresponds to the trace of the fluid pressure at the interface. 
    The larger pressure amplitude observed in the case \(H = 0.01\) can be attributed to the reduced compliance of the thicker wall, which limits its ability to absorb and store elastic energy. 
    In contrast, a thinner, more compliant wall transfers a larger portion of the fluid energy into structural deformation, resulting in larger displacements but reduced pressure amplitudes in the fluid.
\end{enumerate}

This trade-off between wall deformation and pressure amplification is consistent with classical wave propagation theory in elastic tubes, where increased stiffness leads to higher wave speeds and reduced wall motion, while simultaneously promoting stronger pressure responses in the fluid.

\section{Conclusion}

In this work, we introduced and analyzed a new fluid--poroelastic structure interaction (FPSI) model describing the interaction between an incompressible, viscous fluid and a thin fully averaged Kirchhoff poroelastic plate. 
The model is based on the recently derived fully averaged poroelastic plate formulation from \cite{SCTB:26}, which reduces the classical three-dimensional Biot system to a codimension-one problem while preserving the essential coupling mechanisms and energy structure of the original system.

A principal advantage of the fully averaged formulation is the elimination of the mixed-dimensional structure present in the classical Biot poroelastic plate model. 
In the present model, both the plate elastodynamics and the pressure variables are defined on the same codimension-one manifold corresponding to the middle surface of the plate.
This dimensional consistency leads to a formulation that is significantly more tractable for both mathematical analysis and numerical simulation, while remaining energetically consistent with the underlying three-dimensional Biot system.

The coupled FPSI problem considered in this manuscript incorporates physically relevant interface conditions, including continuity of the normal velocity, balance of forces, and tangential slip described by the Beavers--Joseph--Saffman condition. These coupling conditions lead to a nontrivial interaction between the fluid velocity, the plate displacement, the pressure jump across the plate, and the averaged pressure variable.

For the resulting coupled system, we established the existence of finite-energy weak solutions by means of a constructive semi-discretization procedure combined with uniform energy estimates. 
We then proved one of the main analytical results of this manuscript:
the existence and uniqueness of a global-in-time strong solution to a regularized version of the FPSI problem.
The proof relied on a reformulation of the coupled system as an abstract Cauchy problem and on the establishment of sectoriality of the associated operator. 
In particular, we developed a framework based on extrapolation spaces, Neumann operators associated with stationary lifting problems, and the theory of diagonally dominant operator matrices, which allowed us to overcome the substantial analytical difficulties introduced by the lower-order interface coupling terms and the Beavers--Joseph--Saffman condition. 
As a consequence, we obtained maximal $\rL^p$-regularity for the regularized FPSI system, as well as exponential decay of solutions when the boundary condition $q^+$ decays exponentially.

Finally, we developed a monolithic finite element solver for the nonregularized FPSI problem and used it to compare the proposed reduced model with a benchmark Stokes--Biot coupled problem.
The numerical experiments demonstrated that the fully averaged Kirchhoff poroelastic plate model provides an excellent approximation of the full Stokes--Biot system while significantly reducing computational cost.

The present work opens several directions for future research.
From the modeling perspective, an important next step is to introduce a model with nonlinear coupling between the fully averaged Kirchhoff plate and the Navier--Stokes equations.
From the analytical perspective, an important next step is the study of well-posedness of the resulting nonlinear FPSI problem in the context of both weak and strong solutions. 
From the applications perspective, the fully averaged formulation provides a promising foundation for studying thin poroelastic interfaces arising in biological and engineering applications, including bioartificial organs, vascular devices, and filtration membranes. 
Finally, from the computational viewpoint, the reduced structure of the model makes it particularly attractive for efficient large-scale simulations and multiphysics applications involving thin poroelastic media interacting with viscous fluids.

\medskip 

\subsection*{Acknowledgements.}
F.~Brandt's research has been supported by the German National Academy of Sciences Leopoldina through the Leopoldina Fellowship Program with grant number~LPDS 2024-07. 
S.~\v{C}ani\'{c}'s research has been supported in part by the National Science Foundation under grants DMS-2408928, DMS-2247000 and by the  U.S. Department of Energy, Office of Science, Office of Advanced Scientific Computing Research's Applied Mathematics Competitive Portfolios program under Contract No.~AC02-05CH11231. 
A.~Scharf's research has been supported by the National Science Foundation under grant DMS-2247000, and by the Mathematics Department at UC Berkeley. 
J.~Tamba\v{c}a's research has been supported by the Croatian Science Foundation under the project number HRZZ-IP-2022-10-1091, by the project ``Implementation of cutting-edge research and its application as part of the Scientific Center of Excellence for Quantum and Complex Systems, and Representations of Lie Algebras'', Grant No.~PK.1.1.10.0004,  Croatia-USA project ``A comprehensive mathematical and computational framework for next generation stent design'', by the European Union -- NextGenerationEU through the National Recovery and Resilience Plan 2021-2026, institutional grant of University of Zagreb Faculty of Science (IK IA 1.1.3. Impact4Math).

\appendix

\section{Supplemental Material}\label{sec:appendix}

In this section, we collect several analytical tools that form the foundation of our approach, which include sectorial operators, their connection to maximal $\rL^p$-regularity, and the behavior of block operator matrices under perturbations. 

We begin by recalling the notion of sectoriality, which plays a central role in the analysis of parabolic-type problems. 
Indeed, sectorial operators are precisely those that generate analytic semigroups, and therefore give rise to well-posed evolution equations with strong regularity properties. 
Moreover, in Hilbert spaces, sectoriality is closely related to maximal ($\rL^p$-)regularity, which is a key ingredient in the study of nonlinear and coupled systems.

Let $\rX$ be a Banach space. 
A closed linear operator $A \colon \rD(A) \subset \rX \to \rX$ is called \emph{sectorial} if the following conditions are satisfied:
\begin{enumerate}[(a)]
    \item $\overline{\rD(A)} = \overline{\mathrm{R}(A)} = \rX$, i.e., $A$ is densely defined and has dense range,
    \item $(-\infty,0) \subset \rho(A)$, and
    \item there exists a constant $M > 0$ such that
    \[
    \| t(t+A)^{-1} \|_{\cL(\rX)} \le M \quad \text{for all } t > 0.
    \]
\end{enumerate}
We denote by $\cS(\rX)$ the class of sectorial operators on $\rX$.

Let $\Sigma_\theta \coloneqq \{ z \in \C \setminus \{0\} : |\arg z| < \theta \}$ denote the open sector of angle $\theta$ in the complex plane. 
The \emph{spectral angle} of a sectorial operator $A$ is defined by
\[
    \phi_A \coloneqq \inf \left\{ \phi \in (0,\pi) : \rho(-A) \supset \Sigma_{\pi-\phi} \tand \sup_{\lambda \in \Sigma_{\pi-\phi}} \| \lambda (\lambda + A)^{-1} \| < \infty \right\}.
\]

The importance of this notion lies in the fact that sectoriality provides a precise characterization of generators of analytic semigroups. 
In particular, if $A$ is sectorial with $\phi_A < \frac{\pi}{2}$, then the operator $-A$ generates a bounded analytic semigroup on $\rX$ with angle $\frac{\pi}{2} - \phi_A$. 
This property is crucial in our analysis, as it allows us to interpret the coupled fluid--structure system as a parabolic evolution problem and to exploit the associated regularity and stability theory.

We now recall the notion of maximal regularity.
For $J = (0,T)$, $T > 0$, or $J = \R_+$, we say that $A$, or the associated Cauchy problem, admits {\em maximal $\rL^p$-regularity} if for every $f \in \rL^p(J;\rX)$, there exists a unique solution $u \in \rW^{1,p}(J;\rX) \cap \rL^p(J;\rD(A))$ to the Cauchy problem
\begin{equation*}
    \left\{
    \begin{aligned}
        u'(t) - A u(t)
        &= f(t), \enspace \tfor t \in J,\\
        u(0)
        &= 0.
    \end{aligned}
    \right.
\end{equation*}
It is a classical result that maximal $\rL^p$-regularity implies that $A$ is a sectorial operator with spectral angle $\phi_A < \frac{\pi}{2}$, see for instance \cite[Prop.~3.5.2]{PS:16}.
Note also that maximal $\rL^p(\R_+)$-regularity already implies that $0 \in \rho(A)$.

For the converse, i.e., to deduce maximal $\rL^p$-regularity for a sectorial operator, on so-called {\em unconditional martingale differences spaces} $\rX$, one generally requires the stronger property of $\cR$-boundedness of the resolvents, or $\cR$-sectoriality, see also the celebrated result of Weis \cite[Thm.~4.2]{Wei:01} or \cite[Sec.~4.2]{DHP:03}.
However, if $\rX$ is a Hilbert space, then sectoriality of $A$ is sufficient for maximal $\rL^p$-regularity, see \cite{dSi:64}.

If the sectoriality of $A$ is only known to hold up to a shift, then one generally only obtains the maximal $\rL^p$-regularity on finite time intervals by means of a rescaling argument.
We also observe that the case of non-homogeneous initial data $u_0$ can be included by means of interpolation theory see \cite[Sec.~3.5]{PS:16}, and the closed graph theorem yields that the solution can be estimated by the data.
We summarize these observations in the lemma below.

\begin{lem}\label{lem:sect & max reg on Hilbert spaces}
Let $\rX$ be a Hilbert space, consider $J = (0,T)$ for some $T \in (0,\infty)$, and assume that there exists $\omega \ge 0$ such that $A + \omega \in \cS(\rX)$ with $\phi_{A + \omega} < \frac{\pi}{2}$. 
Then $A$ has maximal $\rL^p(J)$-regularity on~$\rX$, i.e., for every $u_0 \in \rX_\gamma \coloneqq (\rX,\rD(A))_{1-\nicefrac{1}{p},p}$ and $f \in \F_f \coloneqq \rL^p(J;\rX)$, there exists a unique solution $u \in \E \coloneqq \rW^{1,p}(J;\rX) \cap \rL^p(J;\rD(A))$ to the abstract Cauchy problem
\begin{equation*}
    \left\{
    \begin{aligned}
        u'(t) - A u(t)
        &= f(t), \enspace \tfor t \in J,\\
        u(0)
        &= u_0.
    \end{aligned}
    \right.
\end{equation*}
Moreover, there exists a constant $C > 0$ such that
\begin{equation*}
    \| u \|_{\E} \le C\bigl(\| f \|_{\F_f} + \| u_0 \|_{\rX_\gamma}\bigr).
\end{equation*}
If $0 \in \rho(A)$, then assertion of the lemma extends to the case $J = \R_+$.
\end{lem}

We now introduce the concept of \emph{diagonally dominant} block operator matrices, which  plays a central role in our analysis of the coupled system. 
The key idea is to treat the full operator as a perturbation of its principal (diagonal) parts. 
In this framework, the diagonal operators encode the main dynamics of the uncoupled subsystems (e.g., fluid and structure), while the off-diagonal operators represent coupling effects. 
Diagonal dominance ensures that these coupling terms are of lower order relative to the principal operators, allowing one to transfer analytical properties---such as sectoriality---from the diagonal blocks to the full operator matrix.

\begin{defn}\label{def:diag dom block op matrix}
Let $\rX_1$ and $\rX_2$ be Banach spaces, and consider linear operators
\begin{equation*}
    \begin{aligned}
        A \colon \rD(A) \subset \rX_1 \to \rX_1, \quad
        D \colon \rD(D) \subset \rX_2 \to \rX_2, \quad
        B \colon \rD(B) \subset \rX_2 \to \rX_1, \quad
        C \colon \rD(C) \subset \rX_1 \to \rX_2.
    \end{aligned}
\end{equation*}
Assume that
\begin{enumerate}[(a)]
    \item $A$ and $D$ are closed operators with dense domains,
    \item $\rD(D) \subset \rD(B)$ and $\rD(A) \subset \rD(C)$, and there exist constants $c_A, c_D \ge 0$ and $C_0 \ge 0$ such that
    \begin{equation*}
        \begin{aligned}
            \| C x \|_{\rX_2} 
            &\le c_A \| A x \|_{\rX_1} + C_0 \| x \|_{\rX_1}, &&\tforall x \in \rD(A),\\
            \| B y \|_{\rX_1} 
            &\le c_D \| D y \|_{\rX_2} + C_0 \| y \|_{\rX_2}, &&\tforall y \in \rD(D).
        \end{aligned}
    \end{equation*}
\end{enumerate}
Setting $\rX \coloneqq \rX_1 \times \rX_2$ and $\rD(\cA) \coloneqq \rD(A) \times \rD(D)$, we define the block operator matrix
\begin{equation}\label{eq:block op matrix}
    \cA \colon \rD(\cA) \subset \rX \to \rX, \qquad \cA \coloneqq 
    \begin{pmatrix}
        A & B\\
        C & D
    \end{pmatrix}.
\end{equation}
The operator $\cA$ is called \emph{diagonally dominant}.
\end{defn}

The importance of this notion lies in the fact that diagonal dominance provides a robust perturbative framework: 
if the diagonal operators $A$ and $D$ possess favorable analytical properties, and the coupling operators $B$ and $C$ are sufficiently controlled in the above sense, then these properties can be extended to the full operator matrix $\cA$. 
In particular, diagonal dominance allows one to establish sectoriality of $\cA$ from that of its diagonal blocks.

We now recall the corresponding results for such operator matrices. 

\begin{lem}\label{lem:sect of diag dom op matrices}
Let $\cA$ be a diagonally dominant block operator matrix as in \autoref{def:diag dom block op matrix}.
\begin{enumerate}[(a)]
    \item Suppose that $A$ and $D$ are sectorial operators on $\rX_1$ and $\rX_2$ with spectral angles $\phi_A$ and $\phi_D$, respectively. 
    If, in addition, there exists $\alpha \in (0,1)$ such that $C \in \cL(\rD(A^\alpha),\rX_2)$, then for every $\psi \in (\max\{\phi_A,\phi_D\},\pi)$ there exists $\omega_0 > 0$ such that, for all $\omega \ge \omega_0$, the shifted operator $\cA + \omega$ is sectorial on $\rX$ with spectral angle $\phi_{\cA} \le \psi$.

    \item In the special case $A = 0$, i.e.,
    \[
        \cA = 
        \begin{pmatrix}
            0 & B\\
            C & D
        \end{pmatrix},
    \]
    if $D$ is sectorial on $\rX_2$ with spectral angle $\phi_D$, then for every $\psi \in (\phi_D,\pi)$ there exists $\omega \ge 0$ such that $\cA + \omega$ is sectorial on $\rX$ with spectral angle $\phi_{\cA} \le \psi$.
\end{enumerate}
\end{lem}
For part~(a), we refer to \cite[Cor.~4.9]{AH:23}, while part~(b) can be found in \cite[Cor.~7.2]{AH:23}.

In the sequel, we formulate the abstract perturbation framework that will be used to establish sectoriality of the coupled operator. 
The main idea is to treat the coupling terms as lower-order perturbations acting between appropriate interpolation and extrapolation spaces. 
This requires a refined functional-analytic setting, in which the operator can be extended beyond its natural domain and analyzed on a scale of intermediate spaces.

To this end, we briefly recall the fractional scale associated with a sectorial operator; we refer to \cite[Sec.~V.1]{Ama:95} for further details. 
Let $A_0 \colon \rX_1 \subset \rX_0 \to \rX_0$ be the generator of a $\rC_0$-semigroup $(T_0(t))_{t \ge 0}$ on a Banach space $\rX_0$. 
In order to accommodate perturbations that are not necessarily defined on $\rX_0$, we introduce the extrapolation space $\rX_{-1}$ as the completion of $\rX_0$ with respect to the norm induced by the resolvent, namely,
\[
    \rX_{-1} \coloneqq \bigl(\rX_0, \| R(\lambda,A_0)\cdot \|_{\rX_0} \bigr)^\sim, \qquad \lambda \in \rho(A_0).
\]
This construction allows us to extend the operator and the associated semigroup to a larger space.

For $\beta \in (-1,1)$, we then define the intermediate spaces via complex interpolation,
\[
    \rX_\beta \coloneqq [\rX_{-1}, \rX_1]_{\frac{\beta+1}{2}},
\]
which yields a scale of spaces connecting $\rX_{-1}$, $\rX_0$, and $\rX_1$. 
These spaces provide a natural setting for measuring the regularity of solutions and for quantifying the size of perturbations.

The semigroup $(T_0(t))_{t \ge 0}$ admits a continuous extension $(T_{-1}(t))_{t \ge 0}$ to $\rX_{-1}$, called the extrapolated semigroup. 
Its generator is denoted by $A_{-1} \colon \rX_0 \subset \rX_{-1} \to \rX_{-1}$. 
By restriction, one obtains semigroups $(T_\beta(t))_{t \ge 0}$ on each space $\rX_\beta$, whose generators are given by the $\rX_\beta$-realizations of $A_{-1}$, namely,
\[
    A_\beta u \coloneqq A_{-1} u, \qquad \rD(A_\beta) = \{ u \in \rX_0 : A_{-1} u \in \rX_\beta \}.
\]

We are now in a position to state an abstract perturbation result that is crucial for handling the coupling terms in our problem. 
We refer to \cite[Cor.~12]{KW:01}, which in fact covers the more general $\cR$-sectorial setting.

\begin{lem}\label{lem:abstr pert result}
Let $A \colon \rX_1 \subset \rX_0 \to \rX_0$ be a sectorial operator on a Hilbert space $\rX_0$ with spectral angle $\phi_A < \pi$, and for $s \in [0,1]$ let $\rX_s \coloneqq [\rX_0,\rX_1]_s$. 
Suppose that there exist $0 \le \alpha < \beta \le 1$ and a bounded operator $B \colon \rX_\alpha \to \rX_{\beta - 1}$. 
Then there exists $\omega \ge 0$ such that, for every $\gamma \in [\alpha,\beta]$, the operator
\[
    \left.(A_{\gamma-1} + B)\right|_{\rX_0} + \omega
\]
is sectorial on $\rX_0$ with spectral angle less than or equal to $\phi_A$.
\end{lem}

This result provides a powerful mechanism for treating perturbations that act between different levels of regularity, and allows us to incorporate the coupling operators in a natural way.

Another important perturbation principle is that of relative boundedness, which applies when the perturbation is controlled directly in terms of the principal operator. 
This situation frequently arises in practice when lower-order terms are present.

\begin{lem}\label{lem:rel bdd pert}
Let $A \colon \rX_1 \to \rX_0$ be a sectorial operator with spectral angle $\phi_A < \pi$, and let $B \colon \rD(B) \to \rX_0$ be a linear operator such that $\rD(A) \subset \rD(B)$. 
Assume that there exist constants $a,b \ge 0$ such that
\[
    \| B x \| \le a \| x \| + b \| A x \| \tforall x \in \rD(A).
\]
Then there exist constants $b_0 > 0$ and $\mu_0 \ge 0$ such that, whenever $b < b_0$ and $\omega > \omega_0$, the operator $A + B + \omega$ is sectorial on $\rX_0$ with spectral angle $\phi_{A + B + \omega} \le \phi_A$.
\end{lem}

In particular, the assumptions of \autoref{lem:rel bdd pert} are satisfied if the perturbation $B$ is relatively $A$-bounded with relative bound zero. 
This situation  arises naturally in our setting, where the coupling terms can be viewed as lower-order contributions relative to the principal operators.


\begin{thebibliography}{99}

\bibitem{AH:23}
A.~Agresti and A.~Hussein, \emph{Maximal $L^p$-regularity and $H^\infty$-calculus for block operator matrices and applications}, J.~Funct.~Anal., 285 (2023), Paper No.~110146.

\bibitem{Ama:95}
H.~Amann, \emph{Linear and Quasilinear Parabolic Problems. Vol.~I. Abstract Linear Theory}, Monographs in Mathematics, vol.~89, Birkh{\"a}user Boston, 1995.

\bibitem{Ama:00}
H.~Amann, \emph{On the strong solvability of the Navier--Stokes equations}, J.~Math.~Fluid Mech., 2 (2000), 16--98.

\bibitem{AENY:19}
I.~Ambartsumyan, V.~J.~Ervin, T.~Nguyen, and I.~Yotov, \emph{A nonlinear Stokes--Biot model for the interaction of a non-Newtonian fluid with poroelastic media}, ESAIM Math.~Model.~Numer.~Anal., 53 (2019), 1915--1955.

\bibitem{AKYZ:18}
I.~Ambartsumyan, E.~Khattatov, I.~Yotov, and P.~Zunino, \emph{A Lagrange multiplier method for a Stokes--Biot fluid-poroelastic structure interaction problem}, Numer.~Math., 140 (2018), 513--553.

\bibitem{Aur:97}
J.-L.~Auriault, \emph{Poroelastic Media. Homogenization and Porous Media}, Interdiscip.~Appl.~Math., vol.~6, Springer, New York, 1997.

\bibitem{AGW:25}
G.~Avalos, E.~Gurvich, and J.~T. Webster, \emph{Weak and strong solutions for a fluid-poroelastic structure interaction via a semigroup approach}, Math.~Methods Appl.~Sci., 48 (2025), 4057--4089.

\bibitem{AW:25}
G.~Avalos and J.~T.~Webster, \emph{Uniqueness of weak solutions for Biot-Stokes interactions}, Pure Appl.~Anal., 7 (2025), 1111--1139.

\bibitem{BQQ:09}
S.~Badia, A.~Quaini, and A.~Quarteroni, \emph{Coupling Biot and Navier--Stokes equations for modelling fluid-poroelastic media interaction}, J.~Comput.~Phys., 228 (2009), 7986--8014.

\bibitem{BBMBNG:19}
H.~T.~Banks, K.~Bekele-Maxwell, L.~Bociu, M.~Noorman, and G.~Guidoboni, \emph{Local sensitivity via the complex-step derivative approximation for 1D poro-elastic and poro-visco-elastic models}, Math.~Control Relat.~Fields, 9 (2019), 623--642.

\bibitem{BGM23}
M.~Barr\'e, C.~Grandmont, and P.~Moireau.
\emph{Analysis of a linearized poromechanics model for incompressible and nearly incompressible materials},
\newblock { Evol.~Equ.~Control Theory}, 12 (2023), 846--906.

\bibitem{Biotwell1}
H.~Barucq, M.~Madaune-Tort, and P.~Saint-Macary.
\emph{Theoretical aspects of wave propagation for {B}iot's consolidation problem}, Monograf\'{\i}as del Seminario Matem{\'a}tico Garc\'{\i}a de Galdeano, 31 (2004), 449--458.

\bibitem{Biotwell2}
H.~Barucq, M.~Madaune-Tort, and P.~Saint-Macary.
\emph{On nonlinear {B}iot's consolidation models}, Nonlinear Anal., 63 (2005), e985--e995.

\bibitem{BMTSM:05}
H.~Barucq, M.~Madaune-Tort, and P.~Saint-Macary. 
\emph{Some existence-uniqueness results for a class of one-dimensional nonlinear Biot models}, Nonlinear Anal., 61 (2005), 591--612.

\bibitem{BHR:25} 
T.~Binz, M.~Hieber, and A.~Roy.
\emph{fluid--structure interaction with porous media: the Beaver-Joseph condition in the strong sense}, J.~Differential Equations, 426 (2025), 660--689.

\bibitem{Biot:55}
M.~A. Biot.
\emph{Theory of elasticity and consolidation for a porous anisotropic solid}, J.~Appl.~Phys., 26 (1955), 182--185.

\bibitem{Biot1941}
M.~A. Biot.
\emph{General theory of three-dimensional consolidation}, J.~Appl.~Phys., 12 (1941), 155--164.

\bibitem{BCMW:21}
L.~Bociu, S.~\v{C}ani\'c, B.~Muha, and J.~T. Webster.
\emph{Multilayered poroelasticity interacting with Stokes flow}, SIAM J.~Math.~Anal., 53 (2021), 6243--6279.

\bibitem{Bociu2016}
L. Bociu, G. Gie, Z. Gruji{\'c}, and J. T. Webster.
\emph{Existence and uniqueness for a quasistatic poroelastic system coupled to a fluid}, SIAM J.~Math.~Anal., 48 (2016), 2117--2149.

\bibitem{BGSW:16}
L.~Bociu, G.~Guidoboni, R.~Sacco, and J.~T. Webster.
\emph{Analysis of nonlinear poro-elastic and poro-visco-elastic models}, Arch.~Ration.~Mech.~Anal., 222 (2016), 1445--1519.

\bibitem{BMW:22}
L.~Bociu, B.~Muha, and J.~T. Webster.
\emph{Weak solutions in nonlinear poroelasticity with incompressible constituents}, Nonlinear Anal.~Real World Appl., 67 (2022), Paper No.~103563.

\bibitem{BMW:23}
L.~Bociu, B.~Muha, and J.~T. Webster.
\emph{Mathematical effects of linear visco-elasticity in quasi-static Biot models}, J.~Math.~Anal.~Appl., 528 (2023), Paper No.~127462.

\bibitem{Bociu2021}
L. Bociu and M. Verri.
\emph{Analysis of a coupled fluid--poroelastic structure interaction model}, SIAM J.~Appl.~Math., 81 (2021), 2350--2375.

\bibitem{BW:21}
L.~Bociu and J.~T. Webster.
\emph{Nonlinear quasi-static poroelasticity}, J.~Differential Equations, 296 (2021), 242--278.

\bibitem{BKP:13}
D.~Bothe, M.~K{\"o}hne, and J.~Pr{\"u}ss.
\emph{On a class of energy preserving boundary conditions for incompressible Newtonian flows}, SIAM J.~Math.~Anal., 45 (2013), 3768--3822.

\bibitem{BCM:25}
F.~Brandt, S.~\v{C}ani\'c, and B.~Muha.
\emph{Three-dimensional Navier--Stokes--Biot coupling via a moving reticular plate interface: existence of weak solutions}, arXiv:2508.14310.

\bibitem{BMR:26}
F.~Brandt, C.~M\^indril\u{a}, and A.~Roy.
\emph{Strong time-periodic solutions for a multilayered fluid--structure interaction system with nonlinear coupling}, Nonlinearity, 39 (2026), Paper No.~035018.

\bibitem{Bukac2016}
M.~Buka{\v{c}}.
\emph{Analysis of a fluid--structure interaction problem involving a poroelastic structure}, Phys.~Rev.~E, 94 (2016), Paper No.~033118.

\bibitem{BYZZ:15}
M.~Bukac, I.~Yotov, R.~Zakerzadeh, P.~Zunino.  
\emph{Partitioning strategies for the interaction of a fluid with a poroelastic material based on a Nitsche's coupling approach.} Comput.~Methods Appl.~Mech.~Engrg., 292 (2015), 138--170.

\bibitem{BYZ:15}
M.~Buka\v{c}, I.~Yotov, and P.~Zunino.
\emph{An operator splitting approach for the interaction between a fluid and a multilayered poroelastic structure}, Numer.~Methods Partial Differential Equations, 31 (2015), 1054--1100.

\bibitem{BYZ:17}
M.~Buka\v{c}, I.~Yotov, and P.~Zunino.
\emph{Dimensional model reduction for flow through fractures in poroelastic media}, ESAIM Math.~Model.~Numer.~Anal., 51 (2017), 1429--1471.

\bibitem{CTGMHR:06}
S.~\v{C}ani\'c, J.~Tamba\v{c}a, G.~Guidoboni, A.~Mikeli\'c, C.~J.~Hartley, and D.~Rosenstrauch.
\emph{Modeling viscoelastic behavior of arterial walls and their interaction with pulsatile blood flow}, SIAM J.~Appl.~Math., 67 (2006), 164--193.

\bibitem{YifanDES}
S.~\v{C}ani\'c, Y.~Wang, and M.~Buka\v{c}.
\emph{A next-generation mathematical model for drug eluting stents},
{SIAM} J.~Appl.~Math., 81 (2021), 1503--1529.


\bibitem{CGHPST:14}
P.~Causin, G.~Guidoboni, A.~Harris, D.~Prada, R.~Sacco, and S.~Terragni.
\emph{A poroelastic model for the perfusion of the lamina cribrosa in the optic nerve head}, Math.~Biosci., 257 (2014), 33--41.

\bibitem{Ces:17}
A.~Cesmelioglu.
\emph{Analysis of the coupled Navier--Stokes/Biot problem}, J.~Math.~Anal.~Appl., 456 (2017), 970--991.

\bibitem{Cou:04}
O.~Coussy.
\emph{Poromechanics}, John Wiley \& Sons, New York, 2004.

\bibitem{Ruiz2026}
F.~Dassi, R.~Khot, A.~E.~Rubiano, and R.~Ruiz-Baier. 
\emph{Analysis and virtual element discretisation of a Stokes/Biot-Kirchhoff bulk-surface model}, 
Comput.~Methods Appl.~Mech.~Engrg., 449 (2026), Paper No.~118545.

\bibitem{DHP:03}
R.~Denk, M.~Hieber, and J.~Pr{\"u}ss.
\emph{$\mathcal{R}$-boundedness, Fourier multipliers and problems of elliptic and parabolic type}, Mem.~Amer.~Math.~Soc., 166 (2003), no.~788.

\bibitem{DHP:07}
R.~Denk, M.~Hieber, and J.~Pr{\"u}ss.
\emph{Optimal $\rL^p$-$\rL^q$-estimates for parabolic boundary value problems with inhomogeneous data}, Math.~Z., 257 (2007), 193--224.

\bibitem{dSi:64}
L.~de Simon.
\emph{Un'applicazione della teoria degli integrali singolari allo studio delle equazioni differenziali lineari astratte del prime ordine}, Rend.~Sem.~Mat.~Univ.~Padova, 34 (1964), 205--223.

\bibitem{Fenics}
{\emph{The FEniCS computing platform.}}
https://fenicsproject.org/

\bibitem{FM:03}
J.~L. Ferr{\'i}n and A.~Mikeli{\'c}.
\emph{Homogenizing the acoustic properties of a porous matrix containing an incompressible inviscid fluid}, Math.~Methods Appl.~Sci., 26 (2003), 831--859.

\bibitem{FreeFem}
{\emph{The FreeFem computing platform.}}
https://freefem.org/

\bibitem{Fung1997}
Y.~C. Fung.
\emph{Biomechanics: Circulation},
\newblock Springer, New York, 1997.

\bibitem{GHL:19}
C.~Grandmont, M.~Hillairet, and J.~Lequeurre.
\emph{Existence of local strong solutions to fluid-beam and fluid-rod interaction systems}, Ann.~Inst.~H.~Poincar{\'e} C Anal.~Non Lin{\'e}aire, 36 (2019), 1105--1149.

\bibitem{HK:16}
M.~Hieber and T.~Kashiwabara.
\emph{Global strong well-posedness of the three-dimensional primitive equations in $\rL^p$-spaces}, Arch.~Ration.~Mech.~Anal., 221 (2016), 1077--1115.

\bibitem{HS:18}
M.~Hieber and J.~Saal.
\emph{The Stokes equation in the $L^p$-setting: well-posedness and regularity properties}, In: Handbook of Mathematical Analysis in Mechanics of Viscous Fluids, Y.~Giga and A.~Novotn{\'y}, eds., Springer, Cham, 2018, 117--206.

\bibitem{HAvCR:92}
J.~M.~Huyghe, T.~Arts, D.~H.~van Campen, and R.~S.~Reneman.
\emph{Porous medium finite element model of the beating left ventricle}, Amer.~J.~Physiol., 262 (1992), 1256--1267.

\bibitem{Iliev2016}
O.~P.~Iliev, A.~E.~Kolesov, and P.~N.~Vabishchevich. 
\emph{Numerical solution of plate poroelasticity problems}, Transport in Porous Media, 115 (2016), 563--580.

\bibitem{Kat:95}
T.~Kato.
\emph{Perturbation Theory for Linear Operators}, Reprint of the 1980 edition, Classics in Mathematics, Springer-Verlag, Berlin, 1995.

\bibitem{KRB:25}
R.~Khot and R.~Ruiz Baier. 
\emph{Virtual element methods for Biot-Kirchhoff poroelasticity}, Math.~Comp., 94 (2025), 1101--1146.

\bibitem{Korteweg1878}
D.~J.~Korteweg.
\emph{\"Uber die Fortpflanzungsgeschwindigkeit des Schalles in elastischen R\"ohren}, Annalen der Physik, 241 (1878), 525--542.

\bibitem{KCM:24}
J.~Kuan, S.~\v{C}ani\'c, and B.~Muha.
\emph{Fluid-poroviscoelastic structure interaction problem with nonlinear geometric coupling}, J.~Math.~Pures Appl., 188 (2024), 345--445.

\bibitem{NonlinearFPSI_CRM23}
J.~Kuan, S.~\v{C}ani\'{c}, and B.~Muha.
\emph{Existence of a weak solution to a regularized moving boundary fluid--structure interaction problem with poroelastic media}. Comptes Rendus M\'{e}canique, 351 (2023), 1--30.

\bibitem{SingularLimit}
J.~Kuan, S.~\v{C}ani\'{c}, and B.~Muha.
\emph{A regularized interface method for fluid-poroelastic structure interaction problems with nonlinear geometric coupling}, arXiv:2508.18065.

\bibitem{KBM:25}
P.~Kun{\v{s}}tek, M.~Buka{\v{c}}, and B.~Muha.
\emph{Mass conservation in the validation of fluid-poroelastic structure interaction solvers}, Appl.~Math. Comput., 487 (2025), Paper No.~129081.

\bibitem{KW:01}
P.~C. Kunstmann and L.~Weis.
\emph{Perturbation theorem for maximal $L^p$-regularity}, Ann.~Scuola Norm.~Sup.~Pisa Cl.~Sci., 30 (2001), 415--435.

\bibitem{MCM:15}
A.~Marciniak-Czochra and A.~Mikeli{\'c}.
\emph{A rigorous derivation of the equations for the clamped Biot-Kirchhoff-Love poroelastic plate}, Arch.~Ration.~Mech.~Anal., 215 (2015), 1035--1062.

\bibitem{MT:16}
A.~Mikeli{\'c} and J.~Tamba\v{c}a. 
\emph{Derivation of a poroelastic flexural shell model.} Multiscale Model.~Simul., 14 (2016), 364--397.

\bibitem{MW:12}
A.~Mikeli{\'c} and M.~F. Wheeler.
\emph{On the interface law between a deformable porous medium containing a viscous fluid and an elastic body}, Math.~Models Methods Appl.~Sci., 22 (2012), Paper No.~1250031.

\bibitem{Miotto2021}
P.~Miotto and M.~Buka{\v{c}}.
\emph{A fluid--structure interaction problem with a semi-permeable poroelastic material}, Comput.~Methods Appl.~Mech.~Engrg., 384 (2021), 113941.

\bibitem{Moens1878}
A.~I. Moens.
\emph{Die Pulskurve}, E.~J. Brill, Leiden, 1878.

\bibitem{Nau:12}
T.~Nau.
\emph{$L^p$-Theory of Cylindrical Boundary Value Problems}, PhD thesis, University of Konstanz, Springer Spektrum, 2012.

\bibitem{Nau:15}
T.~Nau.
\emph{The $L^p$-Helmholtz projection in finite cylinders}, Czechoslovak Math.~J., 65 (2015), 119--134.

\bibitem{Biotwell3}
S.~Owczarek.
\emph{A {G}alerkin method for {B}iot consolidation model}, Math.~Mech.~Solids, 15 (2010), 42--56.

\bibitem{OB:20}
O.~Oyekola and M.~Buka{\v{c}}.
\emph{Second-order, loosely coupled methods for fluid-poroelastic material interaction}, Numer.~Methods Partial Differential Equations, 36 (2020), 800--822.

\bibitem{PB:25}
C.~Parrow and M.~Buka{\v{c}}.
\emph{A Robin-Robin strongly coupled partitioned method for fluid-poroelastic structure interaction}, J.~Numer.~Math., 33 (2025), 289--312.

\bibitem{PS:16}
J.~Pr{\"u}ss and G.~Simonett.
\emph{Moving Interfaces and Quasilinear Parabolic Evolution Equations}, Monographs in Mathematics, vol.~105, Birkh{\"a}user, 2016.

\bibitem{RN:20}
E.~Rohan and S.~Naili.
\emph{Homogenization of the fluid--structure interaction in acoustics of porous media perfused by viscous fluid}, Z.~Angew.~Math.~Phys., 71 (2020), Paper No.~137.

\bibitem{SBC:25}
A.~Scharf, M.~Buka{\v{c}}, and S.~\v{C}ani\'c.
\emph{Splitting method for a multilayered poroelastic solid interacting with Stokes flow}, SIAM J.~Sci.~Comput., (accepted 2026), arXiv:2507.10538.

\bibitem{SCTB:26}
A.~Scharf, S.~\v{C}ani\'c, J.~Tamba{\v{c}}a, and F.~Brandt.
\emph{A new fully averaged poroelastic plate model}, in draft form, 2026.

\bibitem{Sho:00}
R.~E.~Showalter.
\emph{Diffusion in poro-elastic media}, J.~Math.~Anal.~Appl., 251 (2000), 310--340.

\bibitem{Sho:05}
R.~E.~Showalter.
\emph{Poroelastic filtration coupled to Stokes flow}, In: Lect.~Notes Pure Appl.~Math., Vol.~242, O.~Imanuvilov et al., eds., Chapman \& Hall/CRC, 2005, 229--241.

\bibitem{Biotwell5}
R.~E.~Showalter and N.~Su.
\emph{Partially saturated flow in a poroelastic medium}, Discrete Contin.~Dyn.~Syst.~Ser.~B, 1 (2001), 403--420.

\bibitem{Terzaghi1943}
K. Terzaghi.
\emph{Theoretical Soil Mechanics}, John Wiley \& Sons, 1943.

\bibitem{Tol:18}
P.~Tolksdorf.
\emph{$\mathcal{R}$-sectoriality of higher-order elliptic systems on general bounded domains}, J.~Evol.~Equ., 18 (2018), 323--349.

\bibitem{Tre:08}
C.~Tretter.
\emph{Spectral Theory of Block Operator Matrices and Applications}, Imperial College Press, 2008.

\bibitem{Tri:78}
H.~Triebel.
\emph{Interpolation Theory, Function Spaces, Differential Operators}, North-Holland, 1978.

\bibitem{VGBS:18}
M.~Verri, G.~Guidoboni, L.~Bociu, and R.~Sacco.
\emph{The role of structural viscoelasticity in deformable porous media with incompressible constituents: applications in biomechanics}, Math.~Biosci.~Eng., 15 (2018), 933--959.

\bibitem{FluidsCanic}
Y.~Wang, S.~\v{C}ani\'{c}, M.~Buka\v{c}, C.~Blaha, and S.~Roy.
\emph{Mathematical and computational modeling of a poroelastic cell scaffold in a bioartificial pancreas}, Fluids, 7 (2022), Paper No.~222.

\bibitem{Wei:01}
L.~Weis.
\emph{Operator-valued Fourier multiplier theorems and maximal $L^p$-regularity}, Math.~Ann., 319 (2001), 735--758.

\bibitem{Biotwell6}
A.~\v{Z}en\'{i}\v{s}ek.
\emph{The existence and uniqueness theorem in {B}iot's consolidation theory}, Aplikace Matematiky, 29 (1984), 194--211.

\end{thebibliography}
\end{document}